\documentclass{amsart}
\usepackage{amssymb}
\usepackage{graphicx}
\numberwithin{equation}{section}

\theoremstyle{plain} % This is the default.
\newtheorem{thm}[equation]{Theorem}
\newtheorem*{thm*}{Theorem}
\newtheorem{cor}[equation]{Corollary}
\newtheorem{lem}[equation]{Lemma}
\newtheorem{prop}[equation]{Proposition}

\theoremstyle{definition}
\newtheorem{defn}[equation]{Definition}
\newtheorem{conj}[equation]{Conjecture}

\theoremstyle{remark}
\newtheorem{rem}[equation]{Remark}
\newtheorem{ex}[equation]{Example}

\title[Distinguished waves]{Distinguished waves and slopes in genus two}

\author{John Berge}
\email{jberge@charter.net}
\date{\today}

\begin{document}

\begin{abstract}
If $R$ is a nonseparating simple closed curve on the boundary of a genus two handlebody $H$ and the manifold $H[R]$, obtained by adding a 2-handle to $H$ along $R$, has incompressible boundary, then there exists a unique arc $\omega$ in $\partial H$, meeting $R$ only in its endpoints, such that, $\omega$ is isotopic in $\partial H$, keeping its endpoints on $R$, to a nontrivial wave based at $R$ in each Heegaard diagram of $R$ on $\partial H$ which has no cut-vertex. Such a wave $\omega$ is a \emph{distinguished-wave} based at $R$. Then surgery on $R$ along its distinguished wave $\omega$ yields a pair of disjoint simple closed curves, say $m_1$ and $m_2$, in $\partial H$, each of which represents a $\omega$-determined-slope $m$ on $\partial H[R]$, that depends only on $R$ and $H$.

\noindent A few  consequences: 
\begin{itemize}
\item Only Dehn filling of $H[R]$ at slope $m$ can yield $S^3$, $(S^1 \times S^2)\; \#\; L(p,q)$, or $S^1 \times S^2$. So $H[R]$ embeds in at most one of $S^3$, $(S^1 \times S^2)\; \#\; L(p,q)$, or $S^1 \times S^2$. And, if such an embedding exists, it is unique. 

\item Theta curves arising from unknotting tunnels of tunnel-number-one knots in $S^3$, $S^1 \times S^2$, or $(S^1 \times S^2)\; \#\; L(p,q)$, have canonical component knots. 

\item One can recognize (1,1) tunnels of (1,1) knots in $S^3$ or $S^1 \times S^2$. 

\item Algorithms for recognizing genus two Heegaard diagrams of $S^3$, $S^1 \times S^2$, or $(S^1 \times S^2)\; \#\; L(p,q)$ that use waves can be streamlined. 

\item Efficient procedures for computing the depth of an unknotting tunnel of a knot in $S^3$, $S^1 \times S^2$, or $(S^1 \times S^2)\; \#\; L(p,q)$ exist.
\end{itemize}

Finally, examples are provided showing that, if $R_1$ and $R_2$ are nonseparating simple closed curves on $\partial H$ such that $H[R_1]$ is homeomorphic to $H[R_2]$, but $(H, R_1)$ and $(H,R_2)$ are not homeomorphic, then the $\omega$-determined slopes on $\partial H[R_1]$ and $\partial H[R_2]$ may differ. 

Continuing, the last section of the paper summarizes computational evidence suggesting however that, if $\mathcal{R}$ is a set of simple closed curves on $\partial H$ such that $R_1 \in \mathcal{R}$ and $R_2 \in \mathcal{R}$ means $H[R_1]$ is homeomorphic to $H[R_2]$, then at most two distinct slopes appear as $\omega$-determined slopes for curves in $\mathcal{R}$, and that, if such distinct $\omega$-determined slopes exist, they are never more than distance one apart.
\end{abstract}

\maketitle

\begin{center}
{CONTENTS}
\end{center}
\noindent \textbf{Section \ref{Introduction})} Introduction.

\noindent \textbf{Section \ref{Preliminaries})} Preliminaries.

\noindent \textbf{Section \ref{Basics of genus two Heegaard diagrams})} Basics of genus two Heegaard diagrams, their graphs and waves based at their curves.

\noindent \textbf{Section  \ref{Distinguished waves are well-defined.})} Distinguished waves are well-defined.

\noindent \textbf{Section \ref{Only distinguished-wave-determined slopes are meridians})} Only distinguished-wave-determined slopes are meridians of tunnel-number-one knots in $S^3$, $S^1 \times S^2$, or $(S^1 \times S^2)\; \#\; L(p,q)$.

\noindent \textbf{Section \ref{Distinguished waves in genus two Heegaard diagrams of closed manifolds})} Distinguished waves in genus two Heegaard diagrams of closed manifolds.

\noindent \textbf{Section \ref{Distinguished waves in genus two Heegaard diagrams of S^3, S1 X S2 etc})} Distinguished waves in genus two Heegaard diagrams of $S^3$, $S^1 \times S^2$, or $(S^1 \times S^2)\; \#\; L(p,q)$.

\noindent \textbf{Section \ref{Some properties of meridian representatives})} Some properties of meridian representatives of tunnel-number-one knots in $S^3$, $S^1 \times S^2$, or $(S^1 \times S^2)\; \#\; L(p,q)$.

\noindent \textbf{Section \ref{Unknotting tunnels and canonical constituent knots of theta curves})} Unknotting tunnels and canonical constituent knots of theta curves in $S^3$, $S^1 \times S^2$, or $(S^1 \times S^2)\; \#\; L(p,q)$.

\noindent \textbf{Section \ref{Recognizing (1,1) tunnels of (1,1) knots})} Recognizing (1,1) tunnels of (1,1) knots in $S^3$ or $S^1 \times S^2$.

\noindent \textbf{Section \ref{Recognizing S^3 etc})} Recognizing genus two Heegaard diagrams of $(S^1 \times S^2)\; \#\; L(p,q)$, $S^3$, or $S^1 \times S^2$.

\noindent \textbf{Section \ref{Computing depth})} Computing the depth of a tunnel of a tunnel-number-one knot in $S^3$, $S^1 \times S^2$, or $(S^1 \times S^2)\; \#\; L(p,q)$.

\noindent \textbf{Section \ref{Other waves based at a fixed curve})} Other waves based at a fixed curve.

\noindent \textbf{Section \ref{Distinguished-wave-determined slope dependence on tunnels})} Distinguished-wave-determined slope dependence on tunnels in tun-nel-number-one manifolds which do not embed in $S^3$, $S^1 \times S^2$, or $(S^1 \times S^2)\; \# \; L(p,q)$.

\section{Introduction}
\label{Introduction}

If $R$ is an essential simple closed curve in a Heegaard diagram $\mathcal{D}$ on the boundary of a handlebody $H$, and $F$ is the surface obtained by cutting $\partial H$ open along $R$, then a \emph{wave} based at $R$ in $\mathcal{D}$ is a properly embedded essential arc $w$ in $F$ such that both endpoints of $w$ lie on the same boundary component of $F$ and $w$ is otherwise disjoint from the curves in $\mathcal{D}$.

When $H$ has genus two, the boundary $\partial H[R]$ of the manifold $H[R]$, obtained by adding a 2-handle to $H$ along $R$, is a torus. In this case, when $w$ is a wave based at $R$ in a Heegaard diagram $\mathcal{D}$ of $R$ on $\partial H$, the two boundary components, $m_1$ and $m_2$ of the regular neighborhood of $R \cup w$ in $\partial H$, which are not isotopic to $R$ in $\partial H$, are isotopic in $\partial H[R]$ and so specify a \emph{wave-determined} \emph{slope}, say $m$, in $\partial H[R]$. 

The results of this paper focus on such wave-determined slopes. Every essential nonseparating simple closed curve $R$ on $\partial H$ such that $H[R]$ has incompressible boundary has at least one wave-determined slope $w$, and may have many. For instance, a later example shows that an infinite number of cyclic surgery slopes of the trefoil knot are wave-determined. While most such wave-determined slopes are ephemeral in the sense that they disappear when the Heegaard diagram $\mathcal{D}$ of $R$ is changed, it is demonstrated here that there always exists a \emph{distinguished-wave} based at $R$ such that, $\omega$ is isotopic in $\partial H$, keeping its endpoints on $R$, to a nontrivial wave based at $R$ in each Heegaard diagram of $R$ on $\partial H$ which has no cut-vertex. Some consequences that follow readily are: 
\begin{itemize}
\item Only Dehn filling of $H[R]$ at slope $m$ can yield $S^3$, $(S^1 \times S^2)\; \#\; L(p,q)$,$S^3$, or $S^1 \times S^2$. So $H[R]$ embeds in at most one of $S^3$, $(S^1 \times S^2)\; \#\; L(p,q)$,$S^3$, or $S^1 \times S^2$. And, if such an embedding exists, it is unique. 

\item Theta curves arising from unknotting tunnels of tunnel-number-one knots in $S^3$, $S^1 \times S^2$, or $(S^1 \times S^2)\; \#\; L(p,q)$, have canonical component knots. 

\item One can recognize (1,1) tunnels of (1,1) knots in $S^3$ or $S^1 \times S^2$. 

\item Algorithms for recognizing genus two Heegaard diagrams of $S^3$, $S^1 \times S^2$, or $(S^1 \times S^2)\; \#\; L(p,q)$ that use waves can be streamlined. 

\item Efficient procedures for computing the depth of an unknotting tunnel of a knot in $S^3$, $S^1 \times S^2$, or $(S^1 \times S^2)\; \#\; L(p,q)$ exist. 
\end{itemize}

Finally the paper considers how distinguished-wave-determined slopes compare in cases in which $R$ and $R'$ are nonseparating curves on the boundaries of genus two handlebodies $H$ and $H'$ respectively such that $H[R]$ and $H'[R']$ are homeomorphic.

In this case, generically, the distinguished-wave-determined slope is an invariant of the manifold $H[R]$, as results of \cite{J06} and \cite{ST06} show that when $R$ is at distance greater than 5 in the curve-complex of $\partial H$ from each curve in $\partial H$ that bounds an essential disk in $H$, and $R'$ is a curve in the boundary of a handlebody
$H'$ of genus two, such that $H[R]$ and $H'[R']$ are homeomorphic, then the pairs $(H,R)$ and $(H',R')$ are also homeomorphic.

The situation is more complicated when $H[R]$ and $H'[R']$ are homeomorphic but $(H,R)$ and $(H',R')$ are not homeomorphic. To investigate how distinguished-wave-determined slopes behave in this situation, the author's Heegaard program \ref{OCC data} was used to look for tunnel-number one manifolds and their distinguished-wave-determined slopes in the subset of 59,107 one-cusped hyperbolic manifolds in SnapPy's `OrientableCuspedCensus' or `OCC' of the set of all hyperbolic, orientable, cusped manifolds which have triangulations using no more than 9 ideal tetrahedra.

A search using Heegaard found 113,378 tunnels for 49,933 of these 59,107 one-cusped manifolds in the OCC. So at least 49,933, or approximately 84.5\% of the one-cusped manifolds in the OCC are tunnel-number one manifolds. Discarding the tunnel-number one manifolds which embed in $S^3$, $S^1 \times S^2$, or $(S^1 \times S^2)\, \# \, L(p,q)$ from this set of 49,933 manifolds left a set of 48,688 manifolds of which 7,917, or about 16.3\%, have distinct distinguished-wave-determined slopes. (Section \ref{Distinguished-wave-determined slope dependence on tunnels} has more details on the OCC manifolds with distinct distinguished-wave-determined slopes.)

\begin{rem}
See \ref{OCC data} for links to source code for Heegaard and data files for Heegaard's survey of distinguished-wave-determined slopes of OCC manifolds.
\end{rem}

\section{Preliminaries}
\label{Preliminaries}

Pairs of simple closed curves in a surface, say $F$, are always assumed to have only essential intersections, up to isotopy in $F$.
A properly embedded nonseparating disk in a handlebody $H$ is a \emph{meridian disk} of $H$.
A set $\Delta$ of pairwise disjoint meridian disks in a handlebody $H$ is a \emph{complete set} of meridian disks of $H$ if cutting $H$ open along the disks in $\Delta$ yields a 3-ball. Then the \emph{genus} of $H$ is equal to the number of disks in $\Delta$, which in turn is equal to the genus of the closed connected orientable surface $\partial H$.

A finite, nonempty set $\mathcal{C}$ of pairwise disjoint essential simple closed curves in the boundary of a handlebody $H$ together with a set $\partial \Delta$ of simple closed curves in $\partial H$ which bound the disks of a complete set of meridian disks $\Delta$ of $H$ is a \emph{Heegaard diagram} $\mathcal{D}$. (We usually assume there is no curve $c \in \mathcal{C}$ that bounds a disk in $H$. And, since $\mathcal{D}$ is generally nonplanar, usually work with the planar diagram $\mathcal{D}'$ obtained by cutting $H$ open along the disks in $\Delta$.)

The operation of cutting $H$ open along the disks in $\Delta$ cuts $H$ into a 3-ball $\mathcal{B}$ with two copies of each disk of $\Delta$ in $\partial \mathcal{B}$. These disks in $\partial \mathcal{B}$ form the `fat'-vertices of a planar graph $G(\mathcal{D})$ whose edges are the arcs obtained by cutting each curve $c \in \mathcal{C}$ open at the points in $c \cap \partial \Delta$. Then $\mathcal{D}$ and $\mathcal{D}'$ can recovered from $G(\mathcal{D})$ together with some additional information which indicates how each pair of `fat'-vertices of $G$ arising from cutting $H$ open along a disk in $\Delta$, must be identified to reconstitute $H$ and the simple closed curves in $\mathcal{C}$.

The \emph{complexity} of the Heegaard diagram of $\mathcal{C}$ and $\partial \Delta$ in $\partial H$ is the total number of essential intersections of $\mathcal{C}$ and $\partial \Delta$ in $\partial H$.

If $c$ is a simple closed curve and $\mathcal{D}$ is a Heegaard diagram to which $c$ contributes edges, let $|c|$ denote the \emph{complexity} of $c$ in $\mathcal{D}$, which we take to be equal to the total number of edges which $c$ contributes to $\mathcal{D}$. More generally, we use $|X|$ to denote the number of elements of a finite set $X$.

The graph terminology we use is standard and basic.
 
If $v$ is a vertex of a connected graph $G$ such that deleting $v$ and the interiors of edges of $G$ meeting $v$ from $G$ disconnects $G$, we say $v$ is a \emph{cut-vertex} of $G$.

An edge $e$ of a graph $G$ is a \emph{loop} in $G$ if there is a vertex $v$ of $G$ containing both ends of $e$.

An element $x$ of a free group $F$ is a \emph{primitive} element of $F$ if $x$ is a free generator of $F$. If $H$ is a handlebody of genus $g$, then $\pi_1(H)$ is a free group of rank $g$, and it is well-known that if $c$ is a simple closed curve in $\partial H$, then $c$ is primitive in $\pi_1(H)$ if and only if there is a meridian disk of $H$ whose boundary intersects $c$ transversely in a single point. 

Let $\Sigma$ be a closed orientable surface, and let $\mathcal{C}$ be a finite set of simple closed curves in $\Sigma$. Let $c$ be a member of $\mathcal{C}$ and let $F$ be the surface with two boundary components $c^+$, $c^-$ obtained by cutting $\Sigma$ open along $c$. Then a \emph{wave based} at $c$ is a properly embedded arc $\omega$ in $F$ such that the interior of $\omega$ is disjoint from each curve in $\mathcal{C}$, both endpoints of $\omega$ lie on the same boundary component of $F$, and $\omega$ is essential in $F$. (If we need to be specific about the \emph{side} of $c$ at which $\omega$ is based, we speak of $\omega$ as \emph{based at} $c^+$, respectively $c^-$, depending upon whether the endpoints of $\omega$ lie on $c^+$, respectively $c^-$.)

\begin{rem}
\label{wave pairs}
Note that because of the symmetry of genus two Heegaard diagrams induced by the omnipresent hyperelliptic involution of genus two Heegaard surfaces  \cite{HS89}, waves based at nonseparating simple closed curves in genus two Heegaard diagrams always arise in pairs whose members are exchanged by the action of the hyperelliptic involution, with one member of each pair based at the `+' side of a curve and the other member of the pair based at the `-' side of a curve, as in Figure~\ref{DPCFig8cs}b, for example. 

With this understood, we will often speak of a wave as being `unique' when we should perhaps say something like ``the pair of waves equivalent under the action of the hyperelliptic involution is unique up to proper isotopy in the Heegaard surface \dots''.
\end{rem}

\begin{rem}
Figures in this paper will often display only one member of each pair of waves related by the hyperelliptic involution.
\end{rem}

\begin{defn}\textbf{Wave-determined slope representatives.}\hfill
\label{wave-determined slope representatives}

Suppose $\mathcal{D}$ is a Heegaard diagram of a nonseparating simple closed curve $R$ on the boundary of a genus two handlebody $H$, $\omega$ is a wave to $R$ in $\mathcal{D}$, and $N$ is a regular neighborhood of $R \cup \omega$ in $\partial H$. Then $N$ is a pair of pants, and we say the two boundary components of $N$, not isotopic to $R$ in $\partial H$, say $m_1$ and $m_2$, are \emph{wave-determined slope representatives}.
\end{defn}

\begin{defn}\textbf{Wave-determined slopes.}\hfill
\label{wave-determined slopes}

While wave-determined slope representatives are not isotopic in $\partial H$, they are isotopic unoriented curves in the torus boundary of $H[R]$, and so represent a \emph{wave-determined slope} $s$ on $\partial H[R]$.
\end{defn}

\begin{defn}\textbf{Surgery along a wave.}\hfill
\label{Surgery Along a Wave}

Suppose $R$ is a nonseparating simple closed curve in an orientable surface $\Sigma$ and $\omega$ is a wave in $\Sigma$ based at $R$. Then a regular neighborhood $N$ of $ R \cup \omega $ in $\Sigma$ is a pair of pants with three simple closed curve boundary components, one of which is isotopic to $R$. If $m$ is one of the other boundary components of $N$, we say $m$ was obtained by \emph{surgery along $R$}.
\end{defn}

%%%%%%%%%%%%%%%%

\section{Basics of genus two Heegaard diagrams, their graphs and waves based at their curves}
\label{Basics of genus two Heegaard diagrams}

\begin{lem}\emph{\textbf{Graphs of genus two Heegaard diagrams.}}\hfill
\label{3 types of graphs of genus two diagrams}

If $\mathcal{C}$ is a finite set of disjoint essential simple closed curves on the boundary of a genus two handlebody $H$,\;$\{D_A, D_B\}$\;is a complete set of meridian disks for $H$, and no curve in $\mathcal{C}$ is disjoint from $\partial D_A \cup \partial D_B$, then the graph of the Heegaard diagram of the curves in $\mathcal{C}$ with respect to $\{D_A,D_B\}$ has the form of one of the three graphs in Figure~\emph{\ref{DPCFig8as}}.
\end{lem}

\begin{proof}
This is a result of \cite{O79}.
\end{proof}

\begin{figure}[ht]
\includegraphics[width = 1.0\textwidth]{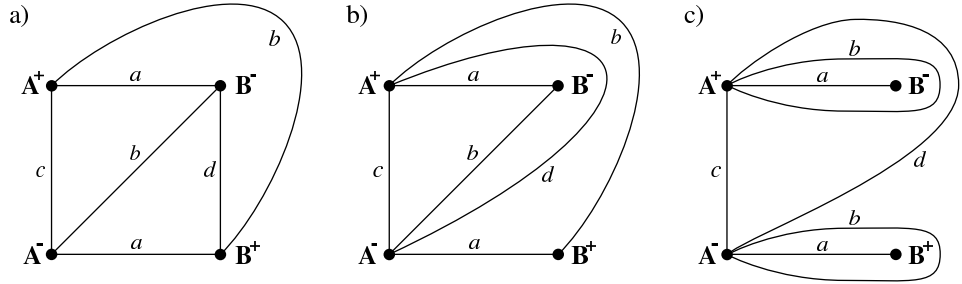}
\caption{\textbf{Three types of graphs of genus two Heegaard diagrams.}
Here $a,b,c,d \geq 0$ indicate how many parallel edges connect each pair of vertices.}
\label{DPCFig8as}
\end{figure}

\begin{lem}\emph{\textbf{Reducing the complexity of Heegaard diagrams with cut-vertices.}}\hfill
\label{wave reduction}

If a Heegaard diagram $\mathcal{D}$ of a set of pairwise disjoint essential simple closed curves in the boundary of a genus two handlebody $H$ has a cut-vertex, $\mathcal{D}$ has the form of Figure \emph{\ref{DPCFig8as}c} or Figure~\emph{\ref{DPCFig8cs}a}. Then there are waves based at $\partial D_A^+$ and $\partial D_A^-$ in $\mathcal{D}$, as shown in Figure~\emph{\ref{DPCFig8cs}b}, and, if $a+b > 0$ in Figure \emph{\ref{DPCFig8cs}a}, the complexity of $\mathcal{D}$ can be reduced by an amount equal to $a+2b$ by replacing the meridional disk $D_A$ of $H$ with the meridional disk $D_C$ of $H$ bounded by the curve $C$ in Figure~\emph{\ref{DPCFig8cs}c}. 
\end{lem}

\begin{proof}
Examine Figure \ref{DPCFig8cs}.
\end{proof}

\begin{figure}[ht]
\includegraphics[width = 1.0\textwidth]{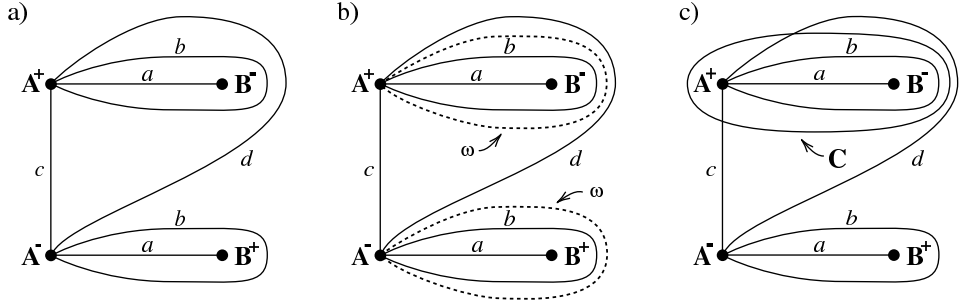}
\caption{\textbf{Reducing the complexity of Heegaard diagrams with cut-vertices.}
}
\label{DPCFig8cs}
\end{figure}

\subsection{Finding minimal complexity genus two Heegaard diagrams}\hfill
\label{Finding minimal complexity genus two Heegaard diagrams}
\smallskip
 
Although the following two lemmas are certainly well-known, because of the central roles they play here, we outline proofs.

\begin{lem}
\label{Lemma 1 of Finding minimal complexity genus two Heegaard diagrams}
Suppose $\mathcal{C}$ is a finite set of pairwise disjoint simple closed curves on the boundary of a genus two handlebody $H$, and $\Delta$ is a complete set of meridian disks for $H$.

If there exists a complete set of meridian disks $\Delta'$ for $H$ such that $|\partial \Delta \cap \mathcal{C}| > |\partial \Delta' \cap \mathcal{C}|$, then there also exists a complete set of meridian disks $\Delta^*$ for $H$ such that $|\partial \Delta \cap \mathcal{C}| > |\partial \Delta^* \cap \mathcal{C}|$, and $\Delta^*$ is obtained from $\Delta$ by replacing one disk of $\Delta$ with a disk in $H$ whose boundary is a bandsum in $\partial H$ of the boundaries of the two disks in $\Delta$. 

It follows that there exists a finite sequence $\{\Delta_i\}$, $0 \leq i \leq n$, of complete sets of meridian disks of $H$ such that: $\Delta_0 = \Delta$; $\partial \Delta_n \cap \mathcal{C}$ is minimal; for each $i$ with $0 \leq i < n$, $|\partial \Delta_i \cap \mathcal{C}| > |\partial \Delta_{i+1} \cap \mathcal{C}|$; and, for each $i$ with $0 \leq i < n$, $\Delta_{i+i}$ is obtained from $\Delta_i$ by replacing one of the disks of $\Delta_i$ with a disk whose boundary is a bandsum in $\partial H$ of the boundaries of the two disks in $\Delta_i$.
\end{lem}

\begin{proof}
Given $\mathcal{C}$ on $\partial H$, let $min(\mathcal{C})$ consist of the complete sets of meridian disks of $H$ which intersect the curves in $\mathcal{C}$ minimally. And then choose $\Delta' \in min(\mathcal{C})$ so that $|\Delta' \cap \Delta|$ consists of a minimal number of arcs. (Recall that because handlebodies are aspherical, simple closed curves of intersection of two disks in $H$ can be removed. So we may assume disks in $\Delta$ and $\Delta'$ intersect only in arcs.) Then there are two possibilities.

If $\Delta$ and $\Delta'$ are disjoint, then either $\Delta$ and $\Delta'$ are isotopic in $H$, or $\Delta'$ is obtained from $\Delta$ by replacing one disk of $\Delta$ with a disk in $H$ whose boundary is a bandsum in $\partial H$ of the boundaries of the two disks in $\Delta$. In which case, we may take $\Delta^*$ to be $\Delta'$.

If $\Delta$ and $\Delta'$ are not disjoint, then there exist outermost subdisks of both $\Delta$ and $\Delta'$ that are cut off of disks of $\Delta$ and $\Delta'$ by arcs of intersection in $\Delta \cap \Delta'$. In this case, choose $D$ among the outermost subdisks of $\Delta$ so that $|D \cap \mathcal{C}|$ is minimal, and set $m = |D \cap \mathcal{C}|$. Similarly, choose $D'$ among the outermost subdisks of $\Delta'$ so that $|D' \cap \mathcal{C}|$ is minimal, and set $m' = |D' \cap \mathcal{C}|$.

Now suppose $\Delta = \{D_A,D_B\}$. Then, since $D'$ is outermost, $D'$ meets one member of $\{D_A,D_B\}$ in a single arc of intersection, say $\alpha$, lying in $\partial D'$, and $D'$ is disjoint from the other member of $\{D_A,D_B\}$. Supposing $\alpha$ lies in, say $D_A$, $\alpha$ cuts $D_A$ into two subdisks, say $D_{A_1}$ and $D_{A_2}$, each of which contains an outermost subdisk meeting $\mathcal{C}$ at least $m$ times. 

Next, note both $D_{A_1} \cup D_{m'}$ and $D_{A_2} \cup D_{m'}$ are nonseparating disks in $H$, with one member of $\{D_{A_1} \cup D_{m'}, D_{A_2} \cup D_{m'}\}$ isotopic to $D_B$ in $H$, while the boundary of the other member of $\{D_{A_1} \cup D_{m'}, D_{A_2} \cup D_{m'}\}$ is a bandsum in $\partial H$ of $\partial D_A$ and $\partial D_B$. In which case, the other member of $\{D_{A_1} \cup D_{m'}, D_{A_2} \cup D_{m'}\}$ together with $D_B$ forms a complete set of meridian disks for $H$. 

Supposing for the moment that $m > m'$, the replacement of $D_A$ with the member of $\{D_{A_1} \cup D_{m'}, D_{A_2} \cup D_{m'}\}$, not isotopic to $D_B$ in $H$, yields a new complete set of meridian disks for $H$, which can be taken as the complete set of meridian disks $\Delta^*$ of the hypothesis. 

It is possible to see that $m > m'$ holds by observing that the alternative $m \leq m'$ is not compatible with the choice of $\Delta'$. For, if $m \leq m'$ held, the outermost disk $D$ of $\Delta$ could be used to perform surgery on a member of $\Delta'$, demonstrating that either $\Delta'$ does not lie in $min(\mathcal{C})$, or $\Delta'$ lies in $min(\mathcal{C})$, but there exists another complete set of meridian disks for $H$, say $\Delta'^*$, in $min(\mathcal{C})$ such that $|\Delta \cap \Delta'^*| < |\Delta \cap \Delta'|$.

Finally, it should be clear that the procedure above can be performed repeatedly and will produce a sequence of complete sets of meridian disks for $H$ with the properties claimed in the statement of the lemma.
\end{proof}

\begin{lem}
\label{Lemma 2 of Finding minimal complexity genus two Heegaard diagrams}
Suppose $\mathcal{C}$ is a finite set of pairwise disjoint simple closed curves on the boundary of a genus two handlebody $H$, and $\Delta$ is a complete set of meridian disks for $H$ such that no complete set $\Delta^*$ of meridian disks for $H$ obtained by replacing one member of $\Delta$ with a disk in $H$ whose boundary is a bandsum in $\partial H$ of the boundaries of the two disks in $\Delta$ satisfies $|\partial \Delta \cap \mathcal{C}| > |\partial \Delta^* \cap \mathcal{C}|$. If $\Delta'$ is also a complete set of meridian disks for $H$ such that no complete set $\Delta'^*$ of meridian disks for $H$ obtained by replacing one member of $\Delta'$ with a disk in $H$ whose boundary is a bandsum in $\partial H$ of the boundaries of the two disks in $\Delta'$ satisfies $|\partial \Delta' \cap \mathcal{C}| > |\partial \Delta'^* \cap \mathcal{C}|$, then there exists a finite sequence $\{\Delta_i\}$, $0 \leq i \leq n$, of complete sets of meridian disks of $H$ such that: $\Delta_0 = \Delta$; $\Delta_n = \Delta'$; for each $i$ with $0 \leq i < n$, $|\partial \Delta \cap \mathcal{C}| = |\partial \Delta' \cap \mathcal{C}|$; and, for each $i$ with $0 \leq i < n$, $\Delta_{i+i}$ is obtained from $\Delta_i$ by replacing one of the disks of $\Delta_i$ with a disk whose boundary is a bandsum in $\partial H$ of the boundaries of the two disks in $\Delta_i$. In particular, $|\partial \Delta \cap \mathcal{C}| = |\partial \Delta' \cap \mathcal{C}|$.
\end{lem}

\begin{proof}
The argument is similar to that above, and again there are two cases.

If $\Delta$ and $\Delta'$ are disjoint in $H$, then either $\Delta$ and $\Delta'$ are isotopic in $H$, or $\Delta'$ is obtained from $\Delta$ by replacing one of the disks of $\Delta$ with a disk whose boundary in $\partial H$ is a bandsum of the boundaries of the disks in $\Delta$, and similarly, $\Delta$ is obtained from $\Delta'$ by replacing one of the disks of $\Delta'$ with a disk whose boundary in $\partial H$ is a bandsum of the boundaries of the disks in $\Delta'$. Then the hypothesis have ruled out the possibility that $|\partial \Delta \cap \mathcal{C}| < |\partial \Delta' \cap \mathcal{C}|$ or that $|\partial \Delta' \cap \mathcal{C}| < |\partial \Delta \cap \mathcal{C}|$. So $|\partial \Delta \cap \mathcal{C}| = |\partial \Delta' \cap \mathcal{C}|$.

If $\Delta$ and $\Delta'$ are not disjoint, then there exist outermost subdisks of both $\Delta$ and $\Delta'$ that are cut off of disks of $\Delta$ and $\Delta'$ by arcs of intersection in $\Delta \cap \Delta'$. In this case, choose $D$ among the outermost subdisks of $\Delta$ so that $|D \cap \mathcal{C}|$ is minimal, and set $m = |D \cap \mathcal{C}|$. Similarly, choose $D'$ among the outermost subdisks of $\Delta'$ so that $|D' \cap \mathcal{C}|$ is minimal, and set $m' = |D' \cap \mathcal{C}|$.

Now suppose $\Delta = \{D_A,D_B\}$. Then, since $D'$ is outermost, $D'$ meets one member of $\{D_A,D_B\}$ in a single arc of intersection, say $\alpha$, lying in $\partial D'$, and $D'$ is disjoint from the other member of $\{D_A,D_B\}$. Supposing $\alpha$ lies in, say $D_A$, $\alpha$ cuts $D_A$ into two subdisks, say $D_{A_1}$ and $D_{A_2}$, each of which contains an outermost subdisk meeting $\mathcal{C}$ at least $m$ times. 

Next, note both $D_{A_1} \cup D_{m'}$ and $D_{A_2} \cup D_{m'}$ are nonseparating disks in $H$, with one member of $\{D_{A_1} \cup D_{m'}, D_{A_2} \cup D_{m'}\}$ isotopic to $D_B$ in $H$, while the boundary of the other member of $\{D_{A_1} \cup D_{m'}, D_{A_2} \cup D_{m'}\}$ is a bandsum in $\partial H$ of $\partial D_A$ and $\partial D_B$. In which case, the other member of $\{D_{A_1} \cup D_{m'}, D_{A_2} \cup D_{m'}\}$ together with $D_B$ forms a complete set of meridian disks $\Delta^*$ for $H$.

Recalling that by hypothesis, $|\Delta^* \cap \mathcal{C}| \geq |\Delta \cap \mathcal{C}|$, we see $m \leq m'$.

A similar argument, replacing $D'$ with $D$ and $\Delta$ with $\Delta'$, shows that $D$ can be used to produce a complete set $\Delta'^*$ of meridian disks for $H$ which is obtained from $\Delta'$ by replacing one disk of $\Delta'$ with a disk whose boundary is a bandsum in $\partial H$ of the boundaries of the two disks in $\Delta'$. However, by hypothesis, $|\Delta'^* \cap \mathcal{C}| \geq |\Delta' \cap \mathcal{C}|$, which implies $m' \leq m$.

So $m = m'$, $|\Delta^* \cap \mathcal{C}| = |\Delta \cap \mathcal{C}|$ and $|\Delta'^* \cap \mathcal{C}| = |\Delta' \cap \mathcal{C}|$. In addition, since $D'$ is a subdisk of $\Delta'$, $|\Delta \cap \Delta'| > |\Delta^* \cap \Delta'|$. And similarly, since $D$ is a subdisk of $\Delta$, $|\Delta \cap \Delta'| > |\Delta'^* \cap \Delta|$. 

It follows that the desired finite sequence $\{\Delta_i\}$, $0 \leq i \leq n$, of complete sets of meridian disks of $H$ exists such that: $\Delta_0 = \Delta$; $\Delta_n = \Delta'$; for each $i$ with $0 \leq i < n$, $|\partial \Delta \cap \mathcal{C}| = |\partial \Delta' \cap \mathcal{C}|$; and, for each $i$ with $0 \leq i < n$, $\Delta_{i+i}$ is obtained from $\Delta_i$ by replacing one of the disks of $\Delta_i$ with a disk whose boundary is a bandsum in $\partial H$ of the boundaries of the two disks in $\Delta_i$.
\end{proof}

\subsection{Three types of bandsums in genus two}\hfill
\label{Three types of bandsums in genus two}

Suppose $\Delta$ = $\{D_A,D_B\}$ is a complete set of meridian disks for $H$ comprised of meridian disks $D_A$ and $D_B$ for $H$, and $\mathcal{D}$ is the Heegaard diagram of the curves in $\mathcal{C}$ with respect to $\partial \Delta$.

Consider the effect on $\mathcal{D}$ of replacing one of $D_A,D_B$ with a meridian disk $D_C$ of $H$ whose boundary $C$ in $\partial H$ is a bandsum of $\partial D_A$ with $\partial D_B$. This replaces $\mathcal{D}$ with a new Heegaard diagram $\mathcal{D}'$  of the curves in $\mathcal{C}$ with respect to the new set of complete meridian disks for $H$. The change in diagrams in passing from $\mathcal{D}$ to $\mathcal{D}'$ is determined by the band $\beta$ used to band $\partial D_A$ and $\partial D_B$ together to form $\partial D_C$. Figure \ref{DPCFig8fff} shows three relevant configurations of $\beta$ with respect to the edges and faces of $\mathcal{D}$. Namely, configurations in which:
\begin{enumerate}
\item $\beta$ has essential intersections with edges of $\mathcal{D}$. 
\item $\beta$ lies in a face of $\mathcal{D}$, but $\beta$ is not properly isotopic to an edge of $\mathcal{D}$.
\item $\beta$ lies in a face of $\mathcal{D}$ and $\beta$ is properly isotopic to an edge of $\mathcal{D}$.
\end{enumerate}
Then examination of Figure \ref{DPCFig8fff} provides some details about the changes that occur to $\mathcal{D}$ in passing from $\mathcal{D}$ to $\mathcal{D}'$ in each of the three cases enumerated above. (Here let $|\{D_A,D_B\}|$ denote the complexity of $\mathcal{D}$, and let $|\{D_A,D_C\}|$, respectively $|\{D_B,D_C\}|$, denote the complexity of $\mathcal{D}'$ when $\mathcal{D}'$ uses compete sets of meridian disks $|\{D_A,D_C\}|$, respectively $|\{D_B,D_C\}|$ of $H$.)
\begin{enumerate}
\item $\beta$ has essential intersections with edges of $\mathcal{D}$.\\ 
In this case, $\mathcal{D}'$ has a cut-vertex, since there is a wave $\omega$ to $C$ in $\mathcal{D}'$. Also $|\{D_A,D_C\}|$ -  $|\{D_A,D_B\}| > a > 0$, and $|\{D_B,D_C\}|$ - $|\{D_A,D_B\}| > b > 0$. So the complexity of $\mathcal{D}'$ is always greater than the complexity of $\mathcal{D}$ in this case.
\item $\beta$ lies in a face of $\mathcal{D}$ and $\beta$ is not properly isotopic to an edge of $\mathcal{D}$.\\
In this case, $\mathcal{D}'$ has a cut-vertex, since there is a wave $\omega$ to $C$ in $\mathcal{D}'$. Also $|\{D_A,D_C\}|$ -  $|\{D_A,D_B\}|$ = $a > 0$, and $|\{D_B,D_C\}|$ - $|\{D_A,D_B\}|$ = $b > 0$. So the complexity of $\mathcal{D}'$ is again always greater than the complexity of $\mathcal{D}$ in this case.
\item $\beta$ lies in a face of $\mathcal{D}$ and $\beta$ is properly isotopic to an edge of $
\mathcal{D}$.\\
In this case, Lemmas \ref{Bandsums along edges of nonpositive diagrams yield connected diagrams without cut-vertices} and \ref{Bandsums along edges of positive diagrams yield connected diagrams without cut-vertices} guarantee that $\mathcal{D}'$ is connected and has no cut-vertex. Then $|\{D_A,D_C\}|$ -  $|\{D_A,D_B\}|$ = $a - c$, and $|\{D_B,D_C\}|$ - $|\{D_A,D_B\}|$ = $b - c$. So, depending on the relative magnitudes of $a,b$, and $c$, and whether $D_C$ replaces $D_A$ or $D_B$, the complexity of $\mathcal{D}'$ may be less than, equal to, or greater than that of $\mathcal{D}$.
\end{enumerate}

In any case, we see the following lemmas hold:

\begin{lem}\textbf{\emph{Only bandsums along edges of Heegaard diagrams do not necessarily increase complexity.}}\hfill
\label{Only bandsums along Heegaard diagram edges do not necessarily increase complexity}
\end{lem}

\begin{lem}\textbf{\emph{Only bandsums along edges of Heegaard diagrams do not necessarily introduce cut-vertices.}}\hfill
\label{Only bandsums along diagram edges do not introduce cut-vertices}
\end{lem}

\begin{figure}[ht]
\includegraphics[width = 0.70\textwidth] {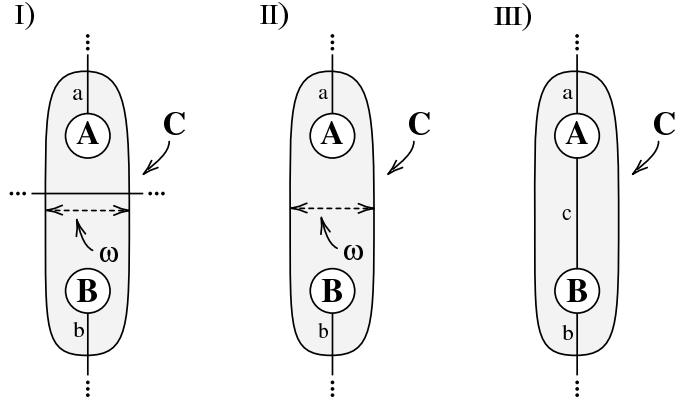}
\label{three types of bandsums figure}
\caption {\textbf{Three types of bandsums.} Here $a,b,c > 0$, and boundary components $A$ and $B$ of each shaded pair of pants should be thought of as connected by a band $\beta$ disjoint from the curve $C$.}
\label{DPCFig8fff}
\end{figure}

\begin{defn}\textbf{Positive and Nonpositive connected diagrams.}\hfill
\label{positive and nonpositive diagrams}

Say a connected Heegaard diagram is \emph{positive} if the curves of the diagram can be oriented so that all intersections of curves in the diagram are positive. Otherwise, say the diagram is \emph{nonpositive}.
\end{defn}

\begin{lem}\textbf{\emph{Bandsums along edges of nonpositive diagrams yield connected diagrams without cut-vertices.}} \hfill
\label{Bandsums along edges of nonpositive diagrams yield connected diagrams without cut-vertices}

Suppose $H$ is a genus two handlebody with a complete set of meridian disks $\{D_A,D_B\}$, and $\mathcal{D}$ is a connected nonpositive Heegaard diagram, without cut-vertices, of a nonseparating simple closed curve $R$ on $\partial H$ with respect to $\{ D_A, D_B\}$. If one of $D_A,D_B$, say $D_B$, is replaced by a meridian disk $D_C$ in $H$ whose boundary $C$ is a bandsum of $\partial D_A$ and $\partial D_B$ along an edge of $\mathcal{D}$, then the resulting Heegaard diagram $\mathcal{D}'$ of $R$ with respect to $\{ D_A,D_C \}$ is connected and has no cut-vertices.
\end{lem}

\begin{proof}
Consider Figure \ref{DPCFig8g1}, which displays the relevant subdiagram of $\mathcal{D}$ together with the curve $C$. Observe that $\mathcal{D}'$ will be connected and have no cut-vertices if each vertex of $\mathcal{D}'$ is connected to at least two other vertices in $\mathcal{D}'$. 

To see this is the case, consider, for example, vertex $A^+$ of $\mathcal{D}'$. Since $b > 0$, there is an edge $e$ of $\mathcal{D}'$ connecting $A^+$ to vertex $C^+$ of $\mathcal{D}'$. And, since $a > 0$, there is an edge $f$ of $\mathcal{D}'$ which is the union of an edge of $\mathcal{D}$ connecting $A^+$ to $B^-$ and an edge of $\mathcal{D}$ connecting $B^+$ to one of $A^-$ or $C^-$. So $A^+$ is connected to at least two distinct vertices in $\mathcal{D}'$. 

Similarly, one verifies that each of the other vertices of $\mathcal{D}'$ is connected to at least two vertices of $\mathcal{D}'$. It follows $\mathcal{D}'$ is connected and has no cut-vertices.
\end{proof}

\begin{figure}[ht]
\includegraphics[width = 0.40\textwidth] {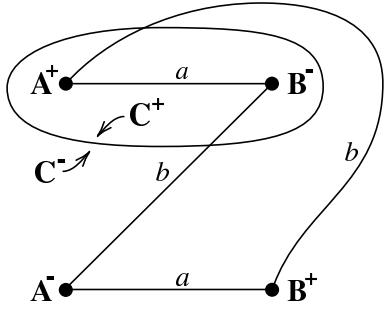}
\caption {\textbf{Bandsums along edges of nonpositive diagrams yield connected diagrams without cut-vertices.} Here $a > 0$ and $b > 0$.}
\label{DPCFig8g1}
\end{figure}

\begin{lem}\textbf{\emph{Bandsums Along Edges of Positive Diagrams Yield Connected, Nonpositive Diagrams Without Cut-Vertices.}} \hfill
\label{Bandsums along edges of positive diagrams yield connected diagrams without cut-vertices}

Suppose $H$ is a genus two handlebody with a complete set of meridian disks $\{D_A,D_B\}$, and $\mathcal{D}$ is a connected positive Heegaard diagram, without cut-vertices, of a nonseparating simple closed curve $R$ on $\partial H$ with respect to $\{D_A, D_B\}$. If one of $D_A,D_B$, say $D_B$, is replaced by a meridian disk $D_C$ in $H$ whose  boundary $C$ is a bandsum of $\partial D_A$ and $\partial D_B$ along an edge of $\mathcal{D}$, then the resulting Heegaard diagram $\mathcal{D}'$ of $R$ with respect to $\{D_A, D_C\}$ is connected, nonpositive, and has no cut-vertices.
\end{lem}

\begin{proof}
Figure \ref{DPCFig8hs} displays $\mathcal{D}$, the curve $C$, and a horizontal wave $\omega_h$ based at $R$. 
First, observe that $c > 0$ and $d > 0$ in Figure \ref{DPCFig8hs} imply $C$ has both positive and negative signed intersections with $R$. So $\mathcal{D}'$ is nonpositive. Next, note it is enough to show that each vertex of $\mathcal{D}'$ is connected to at least two other vertices of $\mathcal{D}'$ in order to conclude $\mathcal{D}'$ is connected and has no cut-vertices. 

Consider, for example, vertex $A^+$ of $\mathcal{D}'$. Since $a > 0$, there is an edge $f$ of $\mathcal{D}'$ which is the union of an edge of $\mathcal{D}$ connecting $A^+$ to $B^-$ and an edge of $\mathcal{D}$ connecting $B^+$ to one of $A^-$ or $C^-$. In addition, since $c > 0$, $A^+$ is connected to $C^+$ in $\mathcal{D}'$. So $A^+$ is connected to at least two distinct vertices in $\mathcal{D}'$.

Similarly, one verifies that each of the other vertices of $\mathcal{D}'$ is connected to at least two other vertices of $\mathcal{D}'$. It follows $\mathcal{D}'$ is connected and has no cut-vertices.
\end{proof}

\begin{figure}[ht]
\includegraphics[width = 0.35\textwidth] {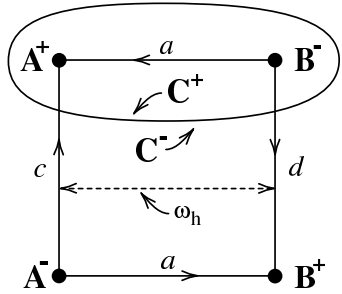}
\caption {\textbf{Bandsums along edges of positive diagrams yield connected diagrams without cut-vertices.} Here $a,c,d > 0$.} 
\label{DPCFig8hs}
\end{figure}

\subsection{Waves based at nonseparating simple closed curves in connected genus two diagrams without cut-vertices}\hfill

\begin{lem}\textbf{\emph{A nonseparating curve in a nonpositive diagram has unique waves.}} \hfill
\label{nonpositive diagrams have unique waves}

Suppose H is a genus two handlebody, $\{D_A,D_B\}$ is a complete set of meridian disks for H and R is a nonseparating simple closed curve on $\partial H$ such that the Heegaard diagram $\mathcal{D}$ of R with respect to $\{D_A,D_B\}$ is nonpositive. Then there is a unique wave based at $R^+$, and a unique wave based at $R^-$, in $\mathcal{D}$.
\end{lem}

\begin{proof}
Consider Figure \ref{PPFig5ja}.
\end{proof}

\begin{figure}[ht]
\includegraphics[width = 0.50\textwidth]{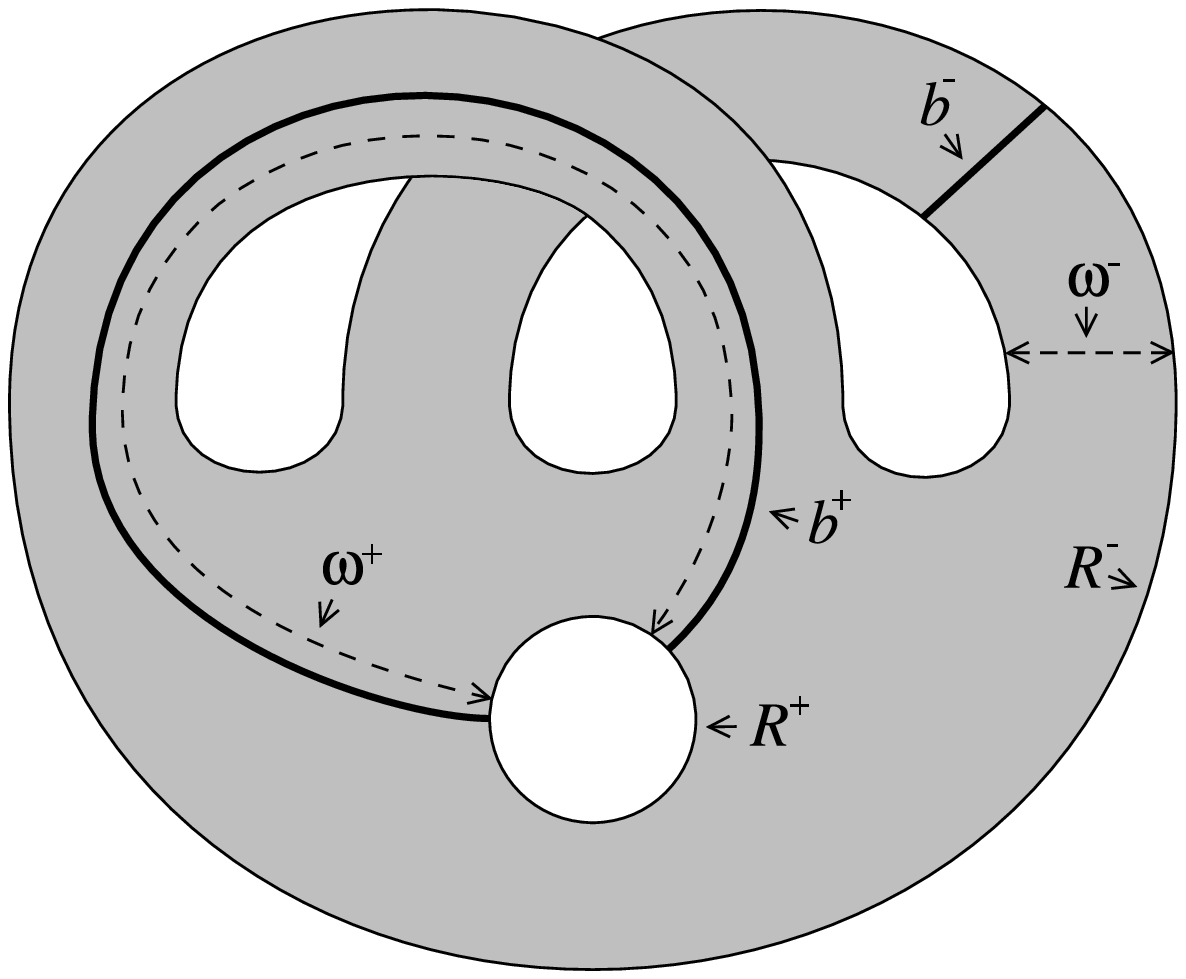}
\caption{\textbf{Genus two nonpositive Heegaard diagrams of nonseparating simple closed curves have unique waves.} Figure \ref{PPFig5ja} shows a twice-punctured torus $F$ obtained by cutting the boundary of a genus two handlebody $H$ open along a nonseparating simple closed curve $R$ in $\partial H$. If $\{D_A,D_B\}$ is a complete set of meridian disks of $H$ such that the Heegaard diagram $\mathcal{D}$ of $R$ with respect to $\{D_A, D_B\}$ is nonpositive, there is a properly embedded subarc $b^+$ of $\partial D_A \cup \partial D_B$ in $F$ with its endpoints on $R^+$, and a properly embedded subarc $b^-$ of $\partial D_A \cup \partial D_B$ in $F$ with its endpoints on $R^-$. Then it is easy to see that any wave $\omega^+$ based at $R^+$ in $\mathcal{D}$ is properly isotopic in $F$ to $b^+$. Similarly, any wave $\omega^-$ based at $R^-$ in $\mathcal{D}$ is properly isotopic in $F$ to $b^-$.}
\label{PPFig5ja}
\end{figure}

\begin{rem}
Although nonpositive genus two Heegaard diagrams with cut-vertices have unique waves, the waves obtained from different nonpositive diagrams may be distinct. Example \ref{Nonpositive distinct waves} shows there exist an infinite number of distinct nonpositive genus two Heegaard diagrams of a trefoil relator $R$ such that waves based at $R$ in distinct nonpositive diagrams are also distinct.
\end{rem}

\begin{defn}\textbf{Horizontal and Vertical waves in positive diagrams without cut-vertices.} \hfill
\label{horizontal and vertical waves in positive diagrams def}

Suppose $\mathcal{D}$ is a positive, connected genus two Heegaard diagram, without cut-vertices, of a nonseparating simple closed curve. Then $\mathcal{D}$ has the form of Figure~\ref{DPCFig8bs}a with vertices $A^+$, $A^-$, $B^+$, and $B^-$, and weights $a, c, d > 0$.

A wave $\omega$ in $\mathcal{D}$ is a \emph{vertical} wave if one endpoint of $\omega$ lies on an edge of $\mathcal{D}$ connecting $A^+$ to one member of $\{B^+, B^-\}$, while the other endpoint of $\omega$ lies on an edge of $\mathcal{D}$ connecting $A^-$ to the other member of $\{B^+, B^-\}$.

A wave $\omega$ in $\mathcal{D}$ is a \emph{horizontal} wave if one endpoint of $\omega$ lies on an edge of $\mathcal{D}$ connecting vertices $A^+$ and $A^-$, while the other endpoint of $\omega$ lies on an edge of $\mathcal{D}$ connecting vertices $B^+$ and $B^-$.
\end{defn}

\begin{figure}[ht]
\includegraphics[width = 0.75\textwidth] {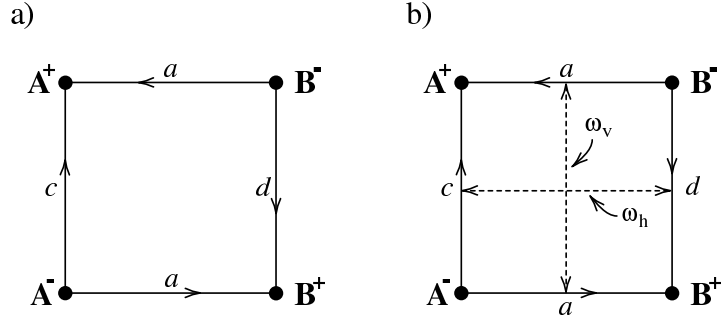}
\caption {\textbf{Horizontal and vertical waves in positive diagrams without cut-vertices.}
A Heegaard diagram $\mathcal{D}$ of a nonseparating simple closed curve $R$ on the boundary of a genus two handlebody $H$ is a \emph{positive diagram} if the graph of $\mathcal{D}$ has the form of Figure~\ref{DPCFig8bs}a with $a,c,d > 0$. Then $\mathcal{D}$ has two waves $\omega_h$ and $\omega_v$, as shown in Figure~\ref{DPCFig8bs}b. In particular, $\mathcal{D}$ has a \emph{horizontal wave} $\omega_h$ with one endpoint of $\omega_h$ on an edge of Figure~\ref{DPCFig8bs}b connecting vertices $A^-$ to $A^+$, and the other endpoint of $\omega_h$ on an edge of Figure~\ref{DPCFig8bs}b connecting vertices $B^-$ to $B^+$.
In addition, $\mathcal{D}$ has a \emph{vertical wave} $\omega_v$ with one endpoint of $\omega_v$ on an edge of Figure~\ref{DPCFig8bs}b connecting vertices $B^-$ to $A^+$, and the other endpoint of $\omega_v$ on an edge of Figure~\ref{DPCFig8bs}b connecting vertices $A^-$ to $B^+$.}
\label{DPCFig8bs}
\end{figure}

\begin{defn}\textbf{The distinguished wave at a nonseparating curve.} \hfill

Suppose $R$ is a nonseparating simple closed curve on the boundary of a genus two handlebody $H$, $H[R]$ has incompressible boundary, and $\{D_A,D_B\}$ is a complete set of meridian disks for $H$ such that the Heegaard diagram $\mathcal{D}$ of $R$ with respect to $\{ D_A, D_B\}$ is connected, and has no cut-vertex. 
\begin{enumerate}
\item If $\mathcal{D}$ is nonpositive, let $\omega$ be the unique wave based at $R$ in $\mathcal{D}$.
\item If $\mathcal{D}$ is positive, let $\omega$ be the horizontal wave $\omega_h$ based at $R$ in $\mathcal{D}$.
\end{enumerate}
Then $\omega$ is the \emph{distinguished wave} based at $R$ in $\partial H$. 
\end{defn}

\section{Distinguished waves are well-defined.}
\label{Distinguished waves are well-defined.}

\begin{thm}\textbf{\emph{Distinguished waves are well-defined.}}\hfill
\label{Thm Distinguished waves are well-defined}

Suppose $R$ is a nonseparating simple closed curve on the boundary of a genus two handlebody $H$, $\Delta$ is a complete set of meridian disks for $H$, the Heegaard diagram $\mathcal{D}$, of $R$ with respect to $\Delta$ is connected and has no cut-vertex, and $\omega$ is a distinguished wave based at $R$ in $\mathcal{D}$.
Then the distinguished wave $\omega$ based at $R$ is well-defined. In other words, $\omega$ depends only on the embedding of $R$ in $\partial H$.

In particular, if $\Delta'$ is another complete set of meridian disks for $H$ such that the Heegaard diagram $\mathcal{D}'$ of $R$ with respect to $\Delta'$ is connected and has no cut-vertex, then $\omega$ also appears as a distinguished wave based at $R$ in $\mathcal{D}'$, up to proper isotopy of $\omega$ in $\partial H$ keeping the endpoints of $\omega$ on $R$.
\end{thm}

\begin{proof}
Since both $\mathcal{D}$ and $\mathcal{D}'$ are connected and have no cut-vertex, Lemma \ref{Sequences of connected Heegaard diagrams without cut-vertices} implies there exist finite sequences $\{\Delta_i\}$, $0 \leq i \leq n$, of complete sets of meridian disks for $H$, and Heegaard diagrams $\{\mathcal{D}_i\}$, where $\mathcal{D}_i$ is the Heegaard diagram of the curves in $C$ with respect to $\Delta_i$, such that: 
\begin{enumerate}
\item $\Delta$ = $\Delta_0$, $\mathcal{D}$ = $\mathcal{D}_0$, $\Delta'$ = $\Delta_n$, $\mathcal{D}'$ = $\mathcal{D}_n$.
\item For each $i$, $0 \leq i < n$, $\Delta_{i+1}$ is obtained from $\Delta_i$ by replacing a disk of $\Delta_i$ with a disk whose boundary in $\partial H$ is a bandsum of the boundaries of the two disks of $\Delta_i$ in $\partial H$.
\item For each $i$, $0 \leq i \leq n$, $\mathcal{D}_i$ is connected and has no cut-vertex.
\end{enumerate}

It then follows, using Lemma \ref{Distinguished Wave persistence under a bandsum} inductively on $i$, that $\omega$ appears as a distinguished wave based at $R$ in each Heegaard diagram $\mathcal{D}_i$ with $0 \leq i \leq n$. In particular, $\omega$ appears in both $\mathcal{D}$ = $\mathcal{D}_0$ and $\mathcal{D}'$ = $\mathcal{D}_n$ as a distinguished wave based at $R$.
\end{proof}

\begin{lem}\textbf{\emph{Sequences of connected Heegaard diagrams without cut-vertices.}} \hfill
\label{Sequences of connected Heegaard diagrams without cut-vertices}

Suppose $\mathcal{C}$ is a finite set of pairwise disjoint simple closed curves on the boundary of a genus two handlebody $H$, $\Delta$ and $\Delta'$ are two complete sets of meridian disks for $H$, and $\mathcal{D}$, respectively $\mathcal{D}'$, is the Heegaard diagram of the curves in $\mathcal{C}$ with respect to $\Delta$, respectively $\Delta'$. If both $\mathcal{D}$ and $\mathcal{D}'$ are connected and have no cut-vertex, then there exists a finite sequence $\{\Delta_i\}$, $0 \leq i \leq n$, of complete sets of meridian disks for $H$, and an associated sequence of Heegaard diagrams $\{\mathcal{D}_i\}$, where $\mathcal{D}_i$ is the Heegaard digram of the curves in $C$ with respect to $\Delta_i$, such that: 
\begin{enumerate}
\item $\Delta$ = $\Delta_0$, $\mathcal{D}$ = $\mathcal{D}_0$, $\Delta'$ = $\Delta_n$, $\mathcal{D}'$ = $\mathcal{D}_n$.
\item For each $i$, $0 \leq i < n$, $\Delta_{i+1}$ is obtained from $\Delta_i$ by replacing a disk of $\Delta_i$ with a disk whose boundary in $\partial H$ is a bandsum of the boundaries of the two disks of $\Delta_i$ in $\partial H$.
\item For each $i$, $0 \leq i \leq n$, $\mathcal{D}_i$ is connected and has no cut-vertex.
\end{enumerate}
\end{lem}

\begin{proof}
Given $\mathcal{C}$ on $\partial H$, let $min(\mathcal{C})$ consist of the complete sets of meridian disks of $H$ which intersect the curves in $\mathcal{C}$ minimally. 

Then the goal is to show that a sequence $\{\Delta_i\}$, of complete sets of meridian disks for $H$ satisfying the lemma can be constructed by concatenating up to three subsequences of complete sets of meridian disks for $H$, depending on whether $\Delta \notin min(\mathcal{C})$, $\Delta \in min(\mathcal{C})$, $\Delta' \notin min(\mathcal{C})$, or $\Delta' \in min(\mathcal{C})$.

Perhaps a detailed example will suffice. 

Suppose, for instance, that neither $\Delta$ or $\Delta'$ belongs to $min(\mathcal{C})$. 
Then, since $\Delta \notin min(\mathcal{C})$, it follows from Lemma \ref{Lemma 1 of Finding minimal complexity genus two Heegaard diagrams} that there exists a finite sequence $S_1 = \{\Delta_h\}$, $0 \leq h \leq n_1$, of complete sets of meridian disks for $H$ such that: $\Delta_0 = \Delta$; $\Delta_{n_1} \in min(\mathcal{C})$; for each $h$ with $0 \leq h < n_1$, $|\partial \Delta_h \cap \mathcal{C}| > |\partial \Delta_{h+1} \cap \mathcal{C}|$; and, for each $h$ with $0 \leq h < n_1$, $\Delta_{h+1}$ is obtained from $\Delta_h$ by replacing one of the disks of $\Delta_h$ with a disk whose boundary is a bandsum in $\partial H$ of the boundaries of the two disks in $\Delta_h$.

Furthermore, since, for $0 \leq h < n_1$, $|\partial \Delta_h \cap \mathcal{C}| > |\partial \Delta_{h+1} \cap \mathcal{C}|$, we have $|\mathcal{D}_h| > |\mathcal{D}_{h+1}|$ in the sequence $\{\mathcal{D}_h\}$ of associated Heegaard diagrams, where $\mathcal{D}_h$ is the Heegaard diagram of the curves in $\mathcal{C}$ with respect to $\Delta_h$. Then Lemma \ref{Only bandsums along Heegaard diagram edges do not necessarily increase complexity} implies that each $\mathcal{D}_{h+1}$, $0 \leq h < n_1$, is obtained from $\mathcal{D}_h$ by a bandsum along an edge of $\mathcal{D}_h$. And then Lemma \ref{Bandsums along edges of nonpositive diagrams yield connected diagrams without cut-vertices} or Lemma \ref{Bandsums along edges of positive diagrams yield connected diagrams without cut-vertices} applies and implies each $\mathcal{D}_{h+1}$, with $0 \leq h < n_1$, is connected and has no cut-vertex.

Next, since $\Delta' \notin min(\mathcal{C})$, the argument used above for the case in which $\Delta \notin min(\mathcal{C})$, applies again, and it implies there exists a finite sequence $S_2 = \{\Delta_j\}$, $0 \leq j \leq n_2$, of complete sets of meridian disks for $H$ such that: $\Delta_0 = \Delta'$; $\Delta_{n_2} \in min(\mathcal{C})$; for each $j$ with $0 \leq j < n_2$, $|\partial \Delta_j \cap \mathcal{C}| > |\partial \Delta_{j+1} \cap \mathcal{C}|$; and, for each $j$ with $0 \leq j < n_2$, $\Delta_{j+1}$ is obtained from $\Delta_j$ by replacing one of the disks of $\Delta_j$ with a disk whose boundary is a bandsum in $\partial H$ of the boundaries of the two disks in $\Delta_j$.

And again, since, for $0 \leq j < n_2$, $|\partial \Delta_j \cap \mathcal{C}| > |\partial \Delta_{j+1} \cap \mathcal{C}|$, we have $|\mathcal{D}_j| > |\mathcal{D}_{j+1}|$ in the sequence $\{\mathcal{D}_j\}$ of associated Heegaard diagrams, where $\mathcal{D}_j$ is the Heegaard diagram of the curves in $\mathcal{C}$ with respect to $\Delta_j$. Then Lemma \ref{Only bandsums along Heegaard diagram edges do not necessarily increase complexity} implies that each $\mathcal{D}_{j+1}$, $0 \leq j < n_2$, is obtained from $\mathcal{D}_j$ by a bandsum along an edge of $\mathcal{D}_j$. And then Lemma \ref{Bandsums along edges of nonpositive diagrams yield connected diagrams without cut-vertices} or Lemma \ref{Bandsums along edges of positive diagrams yield connected diagrams without cut-vertices} applies and implies each $\mathcal{D}_{j+1}$, with $0 \leq j < n_2$, is connected and has no cut-vertex.

Now while $\Delta_{n_1} \in min(\mathcal{C})$ and $\Delta_{n_2} \in min(\mathcal{C})$, it may be the case that $\Delta_{n_1}$ and $\Delta_{n_2}$ are distinct. In this case, it follows from Lemma \ref{Lemma 2 of Finding minimal complexity genus two Heegaard diagrams} that there exists a finite sequence $S_3 = \{\Delta_k\}$, $0 \leq k \leq n_3$, of complete sets of meridian disks for $H$ such that: $\Delta_0 = \Delta_{n_1}$; $\Delta_{n_3} = \Delta_{n_2}$; for each $k$ with $0 \leq k < n_3$, $|\partial \Delta_k \cap \mathcal{C}| = |\partial \Delta_{k+1} \cap \mathcal{C}|$; and, for each $k$ with $0 \leq k < n_3$, $\Delta_{k+1}$ is obtained from $\Delta_k$ by replacing one of the disks of $\Delta_k$ with a disk whose boundary is a bandsum in $\partial H$ of the boundaries of the two disks in $\Delta_k$.

And again, since, for $0 \leq k < n_3$, $|\partial \Delta_k \cap \mathcal{C}| = |\partial \Delta_{k+1} \cap \mathcal{C}|$, we have $|\mathcal{D}_k| = |\mathcal{D}_{k+1}|$ in the sequence $\{\mathcal{D}_k\}$ of associated Heegaard diagrams, where $\mathcal{D}_k$ is the Heegaard diagram of the curves in $\mathcal{C}$ with respect to $\Delta_k$. Then Lemma \ref{Only bandsums along Heegaard diagram edges do not necessarily increase complexity} once more implies that each $\mathcal{D}_{k+1}$, $0 \leq k < n_3$, is obtained from $\mathcal{D}_k$ by a bandsum along an edge of $\mathcal{D}_k$. After which Lemma \ref{Bandsums along edges of nonpositive diagrams yield connected diagrams without cut-vertices} or Lemma \ref{Bandsums along edges of positive diagrams yield connected diagrams without cut-vertices} applies and implies each $\mathcal{D}_{k+1}$, with $0 \leq k < n_3$, is connected and has no cut-vertex.

Finally, a sequence of complete sets of meridian disks $\{\Delta_i\}$ for $H$ and an associated sequence of connected Heegaard diagrams $\{\mathcal{D}_i\}$, without cut-vertices, of the curves in $\mathcal{C}$, where $\mathcal{D}_i$ is the Heegaard diagram of the curves in $\mathcal{C}$ with respect to $\Delta_i$, satisfying the statement of the lemma, can be assembled from $S_1$ followed by $S_3$ and $-S_2$, where $-S_2$ is the sequence of complete sets of meridian disks for $H$ obtained by visiting the complete sets of meridian disks for $H$ in $S_2$ in reverse order.
\end{proof}

\begin{lem}\textbf{\emph{Distinguished wave persistence under a bandsum.}} \hfill
\label{Distinguished Wave persistence under a bandsum}

Suppose $R$ is a nonseparating simple closed curve on the boundary of a genus two handlebody $H$, $\{D_A,D_B\}$ is a complete set of meridian disks for $H$, the Heegaard diagram $\mathcal{D}$ of $R$ with respect to $\{ D_A, D_B\}$ is connected and has no cut-vertex, and $\omega$ is the distinguished wave based at $R$ in $\mathcal{D}$.
In addition, suppose $\{D_A,D_C\}$ is another complete set of meridian disks of $H$ in which the boundary $C$ of $D_C$ is a bandsum of $\partial D_A$ with $\partial D_B$ in $\partial H$, $\mathcal{D}'$ is the Heegaard diagram of $R$ with respect to $\{ D_A, D_C \}$, and either the complexity of $\mathcal{D}'$ is not greater than the complexity of $\mathcal{D}$, or $\mathcal{D}'$ is connected and has no cut-vertex. Then $\omega$ also appears in $\mathcal{D}'$ as a distinguished wave based at $R$.
\end{lem}

\begin{proof}
It follows from Lemmas \ref{Only bandsums along Heegaard diagram edges do not necessarily increase complexity} and \ref{Only bandsums along diagram edges do not introduce cut-vertices} that the curve $C$ is a bandsum of $\partial D_A$ with $\partial D_B$ along an edge of $\mathcal{D}$. Then there are two cases depending on whether $\mathcal{D}$ is positive or nonpositive.

First, suppose $\mathcal{D}$ is positive. Then $\mathcal{D}$ and $C$ have the form shown in Figure~\ref{DPCFig8hs}, and the distinguished wave $\omega$ based at $R$ appears as the horizontal wave $\omega_h$ based at $R$. 

Since $c > 0$ and $d > 0$, $C$ has both positive and negative signed intersections with $R$, and so $\mathcal{D}'$ is nonpositive. In addition, since $\omega_h$ is disjoint from $\partial D_A$ and $C$ in Figure~\ref{DPCFig8hs}, $\omega_h$ appears in $\mathcal{D}'$ as a wave based at $R$. But waves based at $R$ in nonpositive connected Heegaard diagrams without cut-vertices are unique. So $\omega_h$ appears as the distinguished wave based at $R$ in $\mathcal{D}'$.

Next, suppose $\mathcal{D}$ is nonpositive. In this case, there are two possibilities depending on whether $\mathcal{D}'$ is positive or nonpositive. 

First, suppose $\mathcal{D}'$ is nonpositive. Since $\omega$ in $\mathcal{D}$ is properly isotopic to a subarc of $\partial D_A$ or $\partial D_B$ in $\mathcal{D}$, and the curve $C$ is disjoint from $\partial D_A$ and $\partial D_B$ in  $\mathcal{D}$, $\omega$ appears as a wave based at $R$ in $\mathcal{D}'$. However, since $\mathcal{D}'$ is nonpositive, the distinguished wave based at $R$ is the only wave based at $R$. So $\omega$ is the distinguished wave based at $R$.

Second, suppose $\mathcal{D}'$ is positive. As in the case in which $\mathcal{D}'$ is nonpositive, $\omega$ appears as a wave based at $R$ in $\mathcal{D}'$. But now there is the possibility that $\omega$ appears as a vertical wave $\omega_v$ based at $R$ in $\mathcal{D}'$, rather than as a horizontal wave $\omega_h$ based at $R$ in $\mathcal{D}'$. So the possibility that $\omega$ appears as a vertical wave $\omega_v$ in $\mathcal{D}'$ needs to be ruled out.

To do this, consider Figure \ref{DPCFig8hab}. In passing from $\mathcal{D}$ to $\mathcal{D}'$, cutting disk $D_B$ was discarded, and $D_B$ was replaced with $D_C$. But the curve $B = \partial D_B$ is still present in $\mathcal{D}'$, and because $\mathcal{D}$ is connected, and has no cut-vertex, $B$ must appear in $\mathcal{D}'$ as a bandsum of $\partial D_A$ with $\partial D_C$ along an edge of $\mathcal{D}'$. So $B$ must appear in Figure \ref{DPCFig8hab} as shown, in which case $B$ has an essential intersection with $\omega_v$ in $\mathcal{D}'$. Conclude that, since $\omega$ is disjoint from $B = \partial D_B$ in $\mathcal{D}$,  $\omega$ appears as the horizontal wave $\omega_h$ based at $R$ in $\mathcal{D}'$.
\end{proof}

\begin{figure}[ht]
\includegraphics[width = 0.35\textwidth] {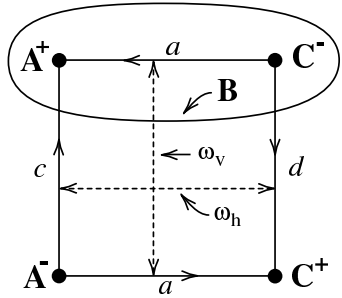}
\caption {\textbf{Bandsums along edges of positive diagrams yield connected diagrams without cut-vertices in which the distinguished wave $\boldsymbol{\omega_h}$ persists.} Here $a,c,d > 0$.} 
\label{DPCFig8hab}
\end{figure}

\begin{defn}\textbf{Distinguished-wave-determined slope representatives  \\ and distinguished-wave-determined slopes.}\hfill

Suppose $R$ is a nonseparating simple closed curve on the boundary of a genus two handlebody $H$, $H[R]$ has incompressible boundary, $\Delta$ is a complete set of meridian disks for $H$ such that the Heegaard diagram $\mathcal{D}$ of $R$ with respect to $\Delta$ is connected, has no cut-vertex, and $\omega$ is the distinguished wave based at $R$ in $\mathcal{D}$. Then the wave-determined slope representatives of Definition \ref{wave-determined slope representatives} are \emph{distinguished-wave-determined slope representatives}, and wave-determined slopes of Definition \ref{wave-determined slopes} as \emph{distinguished-wave-determined slopes}. (Usually shortened to \emph{distinguished-slope representatives} and \emph{distinguished-slopes} respectively.)
\end{defn}

\begin{lem}\textbf{\emph{Distinguished-wave-determined slope representatives never bound disks.}}

Suppose $R$ is a nonseparating simple closed curve on the boundary of a genus two handlebody $H$, $\Delta$ is a complete set of meridian disks for $H$ such that the Heegaard diagram $\mathcal{D}$ of $R$ with respect to $\Delta$ is connected, has no cut-vertex, $\omega$ is the distinguished wave based at $R$ in $\mathcal{D}$, and $m_1$ and $m_2$ are the two distinguished slope representatives obtained by surgery on $R$ along $\omega$. Then neither $m_1$ or $m_2$ bound disks in $H$.
\end{lem}

\begin{proof}
If $m_1$ or $m_2$ bounds a disk in $H$, $\mathcal{D}$ is either not connected, or has a cut-vertex, contrary to hypothesis.
\end{proof}

%%%%%%%%%%%%%%%%

\clearpage
\section{Only distinguished-wave-determined slopes are meridians of tunnel-number-one knots in $S^3$, $S^1 \times S^2$, or $(S^1 \times S^2)\; \#\; L(p,q)$.}
\label{Only distinguished-wave-determined slopes are meridians}

\begin{lem}\textbf{\emph{Waves provide meridians.}} \hfill
\label{waves provide meridians}

Suppose H is a genus two handlebody, R is a nonseparating simple closed curve on $\partial H$ such that the manifold $H[R]$ has incompressible boundary, and $H[R]$ is the exterior of a knot in $S^3$, $S^1 \times S^2$, or $(S^1 \times S^2)\; \#\; L(p,q)$. Then a simple closed curve $m$ in $\partial H$ representing the meridian of $H[R]$ can be obtained by surgery on $R$ along a wave based at $R$.
\end{lem}

\begin{proof}
Since $H[R]$ has incompressible boundary, there exists a complete set of meridian disks $\{D_A,D_B\}$ for $H$ such that the Heegaard diagram $\mathcal{D}$ of $R$ with respect to $\{D_A,D_B\}$ is connected and has no cut-vertex. And, since $H[R]$ is the exterior of a knot in $S^3$, $S^1 \times S^2$, or $(S^1 \times S^2)\; \#\; L(p,q)$, there are simple closed curves in $\partial H$, disjoint from $R$, which represent the meridian of $H[R]$. Among all such representatives of the meridian of $H[R]$, suppose $m$ is chosen so that the complexity of the Heegaard diagram $\mathcal{D}'$ of $\{m, R\}$ with respect to $\{D_A,D_B\}$ is minimal. Then, since $\mathcal{D}$ is connected and has no cut-vertex, $\mathcal{D}'$ is also connected, has nlabelo cut-vertex, and certainly has complexity greater than two.

It now follows from the main results of \cite{HOT80} and \cite{NO85} that $\mathcal{D}'$ must have a wave $\omega$ based at one of $\partial D_A$, $\partial D_B$, $m$, or $R$. However, since $\mathcal{D}'$ is connected, and has no cut-vertex, $\omega$ can not be based at $\partial D_A$ or $\partial D_B$. Furthermore, the minimality condition used to choose $m$ implies $\omega$ is not based at $m$. So $\omega$ must be based at $R$. Then, since the boundary components of a regular neighborhood $N$ of $R \cup \omega$ in $\partial H$ form a pant decomposition of $\partial H$, $m$ must be isotopic to one of the two boundary components of $N$ not isotopic to $R$ in $\partial H$. In other words, $m$ can be obtained by surgery on $R$ along a wave $\omega$ based at $R$.
\end{proof}

\begin{thm}\textbf{\emph{Only distinguished-wave-determined slopes are meridians of tunnel-number-one knots in $\boldsymbol{S^3}$, $\boldsymbol{S^1 \times S^2}$, or $\boldsymbol{(S^1 \times S^2)\; \#\; L(p,q)}$.}}\hfill
\label{Only distinguished waves provide meridians}
\smallskip

Suppose $R$ is a nonseparating simple closed curve on the boundary of a genus two handlebody $H$ such that $R$ has a connected Heegaard diagram $\mathcal{D}$ without cut-vertices with respect to a complete set of meridian disks $\{D_A,D_B\}$ of $H$. Then only Dehn filling of $H[R]$ along the distinguished-wave-determined slope can yield $S^3$, $S^1 \times S^2$, or $(S^1 \times S^2)\; \#\; L(p,q)$.
\end{thm}

\begin{proof}
Lemma \ref{waves provide meridians} shows that, if $H[R]$ embeds in $S^3$, $S^1 \times S^2$, or $(S^1 \times S^2)\; \#\; L(p,q)$, a representative of the meridian of $H[R]$ can be obtained by surgery on $R$ along a wave $\omega$ based at $R$ in $\mathcal{D}$. It remains to show that in this case $\omega$ is actually a distinguished wave based at $R$.

There are two cases to consider, depending upon whether the Heegaard diagram $\mathcal{D}$ of $R$ with respect to $\{D_A,D_B\}$ is positive or nonpositive. If $\mathcal{D}$ is nonpositive, Lemma \ref{nonpositive diagrams have unique waves} shows $\omega$ is the unique wave based at $R$ in $\mathcal{D}$, and so, in this case, $\omega$ is the distinguished wave based at $R$ by default.

The case in which $\mathcal{D}$ is positive requires some work, since in this case there are always two distinct waves based at $R$ in $\mathcal{D}$, only one of which is a distinguished wave. In particular, when $\mathcal{D}$ is connected, positive and has no cut-vertex, the graph of $\mathcal{D}$ has the form of Figure \ref{DPCFig8bs}a in \ref{horizontal and vertical waves in positive diagrams def}, and $\omega$ is either the wave $\omega_h$ or the wave $\omega_v$ in Figure \ref{DPCFig8bs}b in \ref{horizontal and vertical waves in positive diagrams def}.

Lemma \ref{horizontal wave lemma} below rules out the possibility that $\omega$ is a vertical wave $\omega_v$, as in Figure \ref{DPCFig8bs}b in \ref{horizontal and vertical waves in positive diagrams def}.
\end{proof}

\begin{lem}\textbf{\emph{Only horizontal waves yield meridians of positive diagrams of $\boldsymbol{S^3}$, $\boldsymbol{S^1 \times S^2}$, or $\boldsymbol{(S^1 \times S^2)\; \#\; L(p,q)}$.} }\hfill
\label{horizontal wave lemma}

Suppose $R$ is a nonseparating simple closed curve on the boundary of a genus two handlebody $H$, $R$ has a connected, positive Heegaard diagram without cut-vertices with respect to a complete set of meridian disks $\{D_A,D_B\}$ of $H$, and $H[R]$ embeds in $S^3$, $S^1 \times S^2$, or $(S^1 \times S^2)\; \#\; L(p,q)$. Then only surgery on $R$ along a horizontal wave $\omega_h$ based at $R$ will yield a representative of the meridian of $H[R]$.
\end{lem}

\begin{proof}
The hypotheses imply that $R$ has a genus two Heegaard diagram of the form shown in Figure~\ref{DPCFig8bs}a with parameters $a, c, d > 0$. By Lemma~\ref{waves provide meridians}, if $H[R]$ embeds in $S^3$, $S^1 \times S^2$, or $(S^1 \times S^2)\; \#\; L(p,q)$, a representative $m$ of the meridian of $H[R]$ can be obtained by surgery on $R$ along a wave $\omega$ in the Heegaard diagram of $R$ with respect to the set of meridian disks $\{D_A,D_B\}$ of $H$. As Figure~\ref{DPCFig8bs}b shows, there are two such waves, $\omega_v$ and $\omega_h$ along which $R$ can be surgered to obtain a candidate for the meridian of $H[R]$. We want to rule out the possibility that a representative $m$ of the meridian of $H[R]$ can be obtained by surgery on $R$ along the vertical wave $\omega_v$.

To do this, start with the enhanced version of Figure~\ref{DPCFig8bs}b displayed in Figure~\ref{DPCFig8dd1}. Figure~\ref{DPCFig8dd1} is a positive, connected Heegaard diagram $\mathcal{D}$, without cut-vertices, of a nonseparating simple closed curve $R$ with respect to a complete set of meridian disks $\{D_A,D_B\}$ of a genus two handlebody $H$. 
In addition, Figure \ref{DPCFig8dd1} displays a vertical wave $\omega_v$ based at $R$, along with the arcs in which the curves $m_1$ and $m_2$, obtained by surgery on $R$ along $\omega_v$, intersect the two major non-rectangular faces, say $F^+$ and $F^-$, of $\mathcal{D}$. Also, a curve $C$, which is a bandsum of $\partial D_A$ and $\partial D_B$ along an edge of $\mathcal{D}$, has been added to Figure~\ref{DPCFig8dd1}. 

Next, supposing $m$ has been chosen to be one of $m_1$, $m_2$ in Figure~\ref{DPCFig8dd1}, let $D_C$ be the meridian disk of $H$, such that $\partial D_C$ = $C$ in Figure \ref{DPCFig8dd1}, let $\{D_A,D_C\}$ be the complete set of meridian disks of $H$, obtained by replacing the meridian disk $D_B$ of $H$ with the meridian disk $D_C$ of $H$, and let $\mathcal{D}'$ be the Heegaard diagram of $m$ and $R$ with respect to $\{D_A,D_C\}$.

Observe that, since the curve $C$ in Figure~\ref{DPCFig8dd1} is a bandsum of $\partial D_A$ with $\partial D_B$ along an edge of $\mathcal{D}$, the subdiagram of $\mathcal{D}'$ consisting of the Heegaard diagram of $R$ with respect to $\{D_A,D_C\}$ is connected, and has no cut-vertex. It follows $\mathcal{D}'$ is also connected, and has no cut-vertex.

Next, recall that the main results of \cite{HOT80} and \cite{NO85} show that if $m$ represents the meridian of $H[R]$, then any connected, genus two Heegaard diagram, without cut-vertices, of $\{R, m\}$ must have a wave based at $R$ or $m$. We show this is not the case by observing that the Heegaard diagram $\mathcal{D}'$ of $m$ and $R$ with respect to $\{D_A,D_C\}$ does not have a wave based at $R$ or $m$.

To see that the Heegaard diagram $\mathcal{D}'$ of $m$ and $R$ with respect to $\{D_A,D_C\}$ does not have a wave based at $R$ or $m$, return to Figure~\ref{DPCFig8dd1}, and observe that $C \cap F^+$ is a single arc $e^+$, and $C \cap F^-$ is a single arc $e^-$, while each of $m_1 \cap e^+$, $m_1 \cap e^-$, $m_2 \cap e^+$, and $m_2 \cap e^-$ is a single point. 

Then consider Figure~\ref{PPFig5k1} which displays a twice-punctured torus $F$ obtained by cutting the Heegaard surface $\partial H$ of Figure \ref{DPCFig8dd1} open along the curve $R$. The two arcs $e^+$ and $e^-$ in which the curve $C$ of Figure \ref{DPCFig8dd1} intersects the two major non-rectangular faces of the Heegaard diagram $\mathcal{D}$ are also shown in this figure, along with the two curves $m_1$ and $m_2$ obtained by surgery on $R$ along the vertical wave $\omega_v$ in Figure~\ref{DPCFig8dd1}.

Next, observe that any properly embedded essential arc, say $w$, in $F$, which is a candidate for a wave based at $R$ in $\mathcal{D}'$, must be properly isotopic in $F$ to one of $e^+$, $e^-$. However, then $w$ has essential intersections with both $m_1$ and $m_2$ in $F$, so $w$ has essential intersections with $m$ in $\mathcal{D}'$, and $w$ is not a wave in $\mathcal{D}'$. Conclude there is no wave
based at $R$ in $\mathcal{D}'$.

Finally, observe that cutting $F$ open along $e^+$ and $e^-$ yields two annuli, say $\mathcal{A}_1$ and $\mathcal{A}_2$, and that $m_1$ and $m_2$ each intersect each of these annuli in a single essential arc. Any wave based at $m$ in $\mathcal{D}'$ must be based at one of $m_1$, $m_2$ and lie in one of $\mathcal{A}_1$, $\mathcal{A}_2$. Clearly no such wave exists. Conclude there is no wave based at $m$ in $\mathcal{D}'$. This finishes the proof of Theorem \ref{horizontal wave lemma}.
\end{proof}

Below are two immediate corollaries. Among other things, the first corollary implies the tunnel-number-one knot case of Gabai's result \cite{G87} that surgery on a knot in $S^3$ never yields $S^1 \times S^2$, while the second corollary implies the tunnel-number-one knot case of Gordon and Luecke's result \cite{GL89} that knots in $S^3$ are determined by their complements.

\begin{cor}
\label{At most one embedding}
$H[R]$ embeds in at most one of $S^3$, $S^1 \times S^2$, or $(S^1 \times S^2)\; \#\; L(p,q)$.
\end{cor}

\begin{cor} 
\label{Tunnel-number-one knots are determined by their complements}
Tunnel-number-one knots in $S^3$, $S^1 \times S^2$, or $(S^1 \times S^2)\; \#\; L(p,q)$ are determined by their complements.
\end{cor} 

\begin{figure}[ht]
\includegraphics[width = 0.40\textwidth] {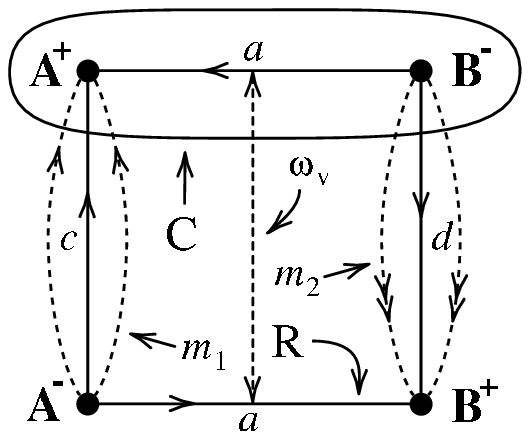}
\caption {}
\label{DPCFig8dd1}
\end{figure}

\begin{figure}[ht]
\includegraphics[width = .55\textwidth]{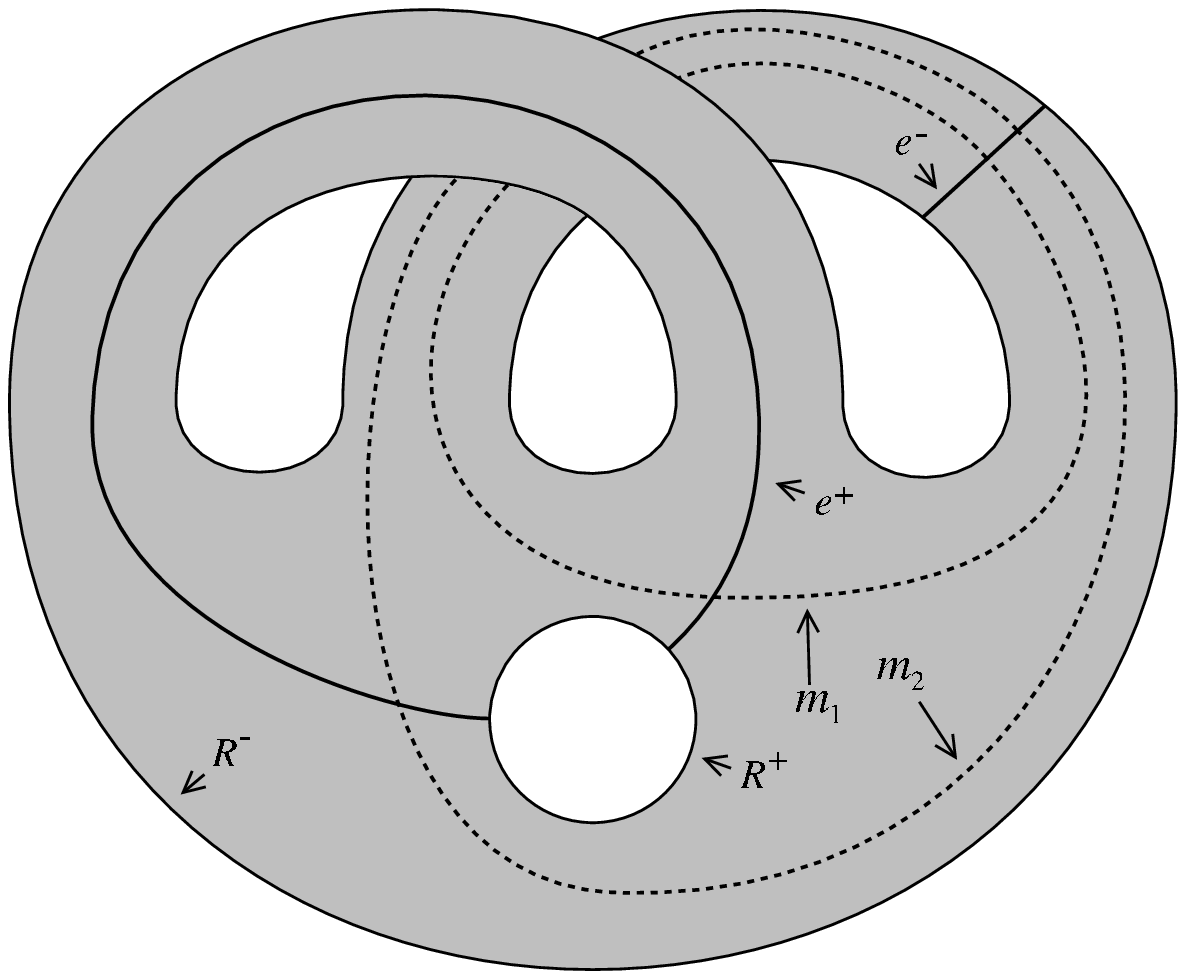}
\caption{}
\label{PPFig5k1}
\end{figure}
 
%%%%%%%%%%%%

\section{Distinguished waves in genus two Heegaard diagrams of closed manifolds.}
\label{Distinguished waves in genus two Heegaard diagrams of closed manifolds}

\begin{thm}\textbf{\emph{Distinguished waves in genus two Heegaard diagrams of closed manifolds are well-defined.}}\hfill
\label{Distinguished waves in diagrams of closed manifolds are  well-defined}

Suppose $R_1$ and $R_2$ are disjoint, nonparallel, nonseparating simple closed curves in the boundary of a genus two handlebody $H$, $\Delta$ and $\Delta'$ are two complete sets of meridian disks for $H$, and both the Heegaard diagram $\mathcal{D}$ of $R_1$ and $R_2$ with respect to $\Delta$, and the Heegaard diagram $\mathcal{D}'$ of $R_1$ and $R_2$ with respect to $\Delta'$ are connected and have no cut-vertices. In addition, 
suppose the subdiagram $\mathcal{D}^*$ of $\mathcal{D}$, consisting of the edges of $R_1$ in $\mathcal{D}$, is connected, has no cut-vertex, and $\omega$ is a distinguished wave based at $R_1$ in $\mathcal{D}$. Then the subdiagram $\mathcal{D}'^*$ of $\mathcal{D}'$, consisting of the edges of $R_1$ in $\mathcal{D}'$, is also connected, has no cut-vertex, and $\omega$ is a distinguished wave based at $R_1$ in $\mathcal{D}'$.
\end{thm}

\begin{proof}
Since $\mathcal{D}$ and $\mathcal{D}'$ are both connected Heegaard diagrams without cut-vertices, Lemma \ref{Sequences of connected Heegaard diagrams without cut-vertices} applies, and it shows there exists a finite sequence $\{\Delta_i\}$, $0 \leq i \leq n$, of complete sets of meridian disks for $H$, and an associated sequence of Heegaard diagrams $\{\mathcal{D}_i\}$, where $\mathcal{D}_i$ is the Heegaard digram of $R_1$ and $R_2$ with respect to $\Delta_i$, such that: 
\begin{enumerate}
\item $\Delta_0$ = $\Delta$, $\mathcal{D}_0$ = $\mathcal{D}$, $\Delta_n$ = $\Delta'$, $\mathcal{D}_n$ = $\mathcal{D}'$;
\item For each $i$, $0 \leq i < n$, $\Delta_{i+1}$ is obtained from $\Delta_i$ by replacing a disk of $\Delta_i$ with a disk whose boundary in $\partial H$ is a bandsum of the boundaries of the two disks of $\Delta_i$ in $\partial H$;
\item For each $i$, $0 \leq i \leq n$, $\mathcal{D}_i$ is connected and has no cut-vertex.
\end{enumerate}

Now Lemma \ref{Distinguished wave persistence under bandsums preserving connectivity and no cut-vertices when two relators are present} below shows that more is true. Namely, the subdiagram of each $\mathcal{D}_i$, $0 \leq i \leq n$, consisting of the edges of $\mathcal{D}_i$ belonging to $R_1$, is connected and has no cut-vertex. This implies $\omega$ appears as a distinguished wave based at $R_1$ in each $\mathcal{D}_i$ for $0 \leq i \leq n$. In particular, $\omega$ appears as a distinguished wave based at $R_1$ in $\mathcal{D}$.
\end{proof}

\begin{lem}\textbf{\emph{Distinguished waves persist under bandsums that preserve connectivity and no cut-vertices when two relators are present.}}\hfill
\label{Distinguished wave persistence under bandsums preserving connectivity and no cut-vertices when two relators are present}

Suppose $R_1$ and $R_2$ are disjoint nonseparating simple closed curves on the boundary of a genus two handlebody $H$, $\Delta$ is a complete set of meridian disks for $H$, the Heegaard diagram $\mathcal{D}$ of $R_1 \cup R_2$ with respect to $\Delta$ is connected and has no cut-vertex, and the subdiagram $\mathcal{D}^*$ of $\mathcal{D}$ consisting of the edges of $R_1$ in $\mathcal{D}$ is also connected and has no cut-vertex. 

In addition, suppose $\omega$ is a distinguished wave based at $R_1$ in $\mathcal{D}$,
and $\Delta'$ is a complete set of meridian disks for $H$, obtained by replacing one member of $\Delta$ with a disk whose boundary is a bandsum $\beta$ in $\partial H$ of the boundaries of the disks of $\Delta$, such that the Heegaard diagram $\mathcal{D}'$ of $R_1 \cup R_2$ with respect to $\Delta'$ is connected and has no cut-vertex.

Then the subdiagram $\mathcal{D}''$ of $\mathcal{D}'$ consisting of the edges of $R_1$ in $\mathcal{D}'$ is connected, has no cut-vertex, and $\omega$ is a distinguished wave based at $R_1$ in $\mathcal{D}'$.
\end{lem}

\begin{proof}
Lemma \ref{Only bandsums along diagram edges do not introduce cut-vertices} implies that, because $\mathcal{D}'$ is connected and has no cut-vertex, the bandsum $\beta$ must be along a band of parallel edges of $\mathcal{D}$. Now there are two possibilities depending on whether $\mathcal{D}^*$ is nonpositive or positive.

If $\mathcal{D}^*$ is nonpositive, $\beta$ is also a bandsum along an edge of $\mathcal{D}^*$, and it follows from Lemma \ref{Bandsums along edges of nonpositive diagrams yield connected diagrams without cut-vertices} that $\mathcal{D}''$ is connected and has no cut-vertex. But then $\omega$ is a distinguished wave based at $R_1$ in $\mathcal{D}'$, per Lemma \ref{Distinguished Wave persistence under a bandsum}.

So, suppose $\mathcal{D}^*$ is positive. In this case $\omega$ must be horizontal.
Then, because $R_2$ is disjoint from $\omega$, $\mathcal{D}$ must also be positive. This implies $\beta$ is again a bandsum along an edge of $\mathcal{D}^*$. Then Lemma \ref{Bandsums along edges of positive diagrams yield connected diagrams without cut-vertices} shows $\mathcal{D}''$ is connected and has no cut-vertex. So again $\omega$ is a distinguished wave based at $R_1$ in $\mathcal{D}'$, per Lemma \ref{Distinguished Wave persistence under a bandsum}.
\end{proof}

\begin{lem}\textbf{\emph{At most one distinguished-wave in a two-relator diagram.}}\hfill

Suppose $R_1$ and $R_2$ are disjoint, nonparallel, nonseparating simple closed curves on the boundary of a genus two handlebody $H$, $\Delta$ is a complete set of meridian disks for $H$, $\mathcal{D}$ is the Heegaard diagram of $R_1 \cup R_2$ on $\partial H$ with respect to $\Delta$, and the subdiagrams of $\mathcal{D}$, consisting of the edges of $R_1$ in $\mathcal{D}$, respectively the edges of $R_2$ in $\mathcal{D}$, are each connected and have no cut-vertex. Then either there is no distinguished wave based at $R_1$ in $\mathcal{D}$, or there is no distinguished wave based at $R_2$ in $\mathcal{D}$.
\end{lem}

\begin{proof}
Suppose $\omega_1$ is a distinguished wave based at $R_1$ in $\mathcal{D}$, and let $m_1$ and $m_2$ be the two distinguished curves obtained by surgery on $R_1$ along $\omega_1$. Then the complexities $|m_1|$, $|m_2|$, and $|R_1|$ of $m_1$, $m_2$ and $R_1$ in $\mathcal{D}$ satisfy $|m_1| < |R_1|$ and $|m_2| < |R_1|$. 

Since $R_2$ and $\omega_1$ are disjoint, $R_2$ is isotopic in $\partial H$ to one of $m_1$, $m_2$. So $|R_2| < |R_1|$ in $\mathcal{D}$. This implies there is no distinguished wave $\omega_2$ based at $R_2$ in $\mathcal{D}$.
\end{proof}

%%%%%%%%%%%%%%%%

\clearpage

\section{Distinguished waves in genus two Heegaard diagrams of $S^3$, $S^1 \times S^2$, or $(S^1 \times S^2)\; \#\; L(p,q)$.}
\label{Distinguished waves in genus two Heegaard diagrams of S^3, S1 X S2 etc}

\begin{thm}\textbf{\emph{A distinguished wave appears in each connected genus two Heegaard diagram, without cut-vertices, of $\boldsymbol{(S^1 \times S^2)\; \#\; L(p,q)}$, $\boldsymbol{S^3}$, or $\boldsymbol{S^1 \times S^2}$.}}
\label{A distinguished wave appears in each connected H-diagram of S^3 without cut-vertices}

Suppose $R_1$ and $R_2$ are disjoint, nonseparating simple closed curves on the boundary of a genus two handlebody $H$, $\{D_A,D_B\}$ is a complete set of meridian disks for $H$, and the Heegaard diagram $\mathcal{D}$ of $\{R_1, R_2\}$ with respect to $\{D_A,D_B\}$ is a connected Heegaard diagram, without cut-vertices, of $S^3$, $S^1 \times S^2$, or $(S^1 \times S^2)\; \#\; L(p,q)$. Then there is a distinguished wave $\omega$ in $\mathcal{D}$ based at one of $R_1, R_2$, and if $\omega$ is based at $R_1$, respectively $R_2$, the subdiagram of $\mathcal{D}$ consisting of the edges of $\mathcal{D}$ belonging to $R_1$, respectively $R_2$, is connected and has no cut-vertex.
\end{thm}

\begin{proof}
Lemma \ref{Connected, no cut-vertex ==> Not both primitives} shows that one of $R_1$, $R_2$, say $R_1$, is not primitive or a proper-power in $H$. Then it follows from Lemma \ref{wave reduction} that there exists a complete set of meridian disks $\{D_A',D_B'\}$ of $H$ and a Heegaard diagram $\mathcal{D}'$ of $\{R_1, R_2\}$ with respect to $\{D_A',D_B'\}$ such that the subdiagram of $R_1$ with respect to $\{D_A',D_B'\}$ is connected and has no cut-vertex. Next, ignoring $R_2$ for the moment, Theorem \ref{Only distinguished waves provide meridians} implies there exists a distinguished wave $\omega'$ based at $R_1$ in $\mathcal{D}'$. 

However, $R_2$ is present in $\mathcal{D}'$, and there are two possibilities: either $R_2$ is disjoint from $\omega'$, or $R_2$ and $\omega'$ have essential intersections. 

Suppose first that $R_2$ is disjoint from $\omega'$. Then, because the subdiagram of $R_1$ with respect to $\{D_A',D_B'\}$ is connected and has no cut-vertex, $\mathcal{D}'$ is also connected and has no cut-vertex. Then Lemma \ref{Sequences of connected Heegaard diagrams without cut-vertices} applies and it shows there exists a finite sequence $\{\Delta_i\}$, $0 \leq i \leq n$, of complete sets of meridian disks for $H$, and an associated sequence of Heegaard diagrams $\{\mathcal{D}_i\}$, where $\mathcal{D}_i$ is the Heegaard digram of $R_1$ and $R_2$ with respect to $\Delta_i$, such that: 
\begin{enumerate}
\item $\Delta_0$ = $\{D_A',D_B'\}$, $\mathcal{D}_0$ = $\mathcal{D}'$, $\Delta_n$ = $\{D_A,D_B\}$, $\mathcal{D}_n$ = $\mathcal{D}$;
\item For each $i$, $0 \leq i < n$, $\Delta_{i+1}$ is obtained from $\Delta_i$ by replacing a disk of $\Delta_i$ with a disk whose boundary in $\partial H$ is a bandsum of the boundaries of the two disks of $\Delta_i$ in $\partial H$;
\item For each $i$, $0 \leq i \leq n$, $\mathcal{D}_i$ is connected and has no cut-vertex.
\end{enumerate}

Now Lemma \ref{Distinguished wave persistence under bandsums preserving connectivity and no cut-vertices when two relators are present} shows that more is true. Namely, the subdiagram of each $\mathcal{D}_i$, $0 \leq i \leq n$, consisting of the edges of $\mathcal{D}_i$ belonging to $R_1$, is connected and has no cut-vertex. This implies $\omega'$ appears as a distinguished wave based at $R_1$ in each $\mathcal{D}_i$ for $0 \leq i \leq n$. In particular, $\omega'$ appears as a distinguished wave based at $R_1$ in $\mathcal{D}$. So, in this case, $\omega'$ is the desired distinguished wave $\omega$ in $\mathcal{D}$.

Next, suppose $R_2$ has essential intersections with $\omega'$ in $\mathcal{D}'$. Then Lemma~\ref{m_1 and m_2 are the only possible primitive or proper-power meridian reps} applies and shows that the subdiagram of $\mathcal{D}'$ consisting of the edges of $R_2$ in $\mathcal{D}'$ is connected and has no cut-vertex. So there exists a subarc $\omega''$ of $\omega'$, such that $\omega''$ is a wave based at $R_2$ in $\mathcal{D}'$, and $R_1$, disjoint from $\omega''$ represents the meridian of $H[R_2]$. It follows from Theorem \ref{Only distinguished waves provide meridians}, that $\omega''$ is a distinguished wave based at $R_2$ in $\mathcal{D}'$.

Finally, since $R_1$ and $\omega''$ are disjoint in $\mathcal{D}'$, the argument used above, when $\omega'$ is disjoint from $R_2$, works with $R_1$ and $\omega'$ swapped with $R_2$ and $\omega''$, to show that $\omega''$, based at $R_2$, is the desired distinguished wave $\omega$ in $\mathcal{D}$.
\end{proof}

\begin{lem}\textbf{\emph{Connected genus two Heegaard diagrams of $\boldsymbol{S^3}$, $\boldsymbol{S^1 \times S^2}$ or $\boldsymbol{(S^1 \times S^2)\; \#\; L(p,q)}$, without cut-vertices, have a relator which is not primitive or a proper-power.}} \hfill
\label{Connected, no cut-vertex ==> Not both primitives}

Suppose $R_1$ and $R_2$ are disjoint nonseparating simple closed curves on the boundary of a genus two handlebody $H$, $\{D_A,D_B\}$ is a complete set of meridian disks for $H$, the Heegaard diagram $\mathcal{D}$ of $\{R_1, R_2\}$ with respect to $\{D_A,D_B\}$ is connected, has no cut-vertex, and is a Heegaard diagram of $S^3$, $S^1 \times S^2$, or $(S^1 \times S^2)\; \#\; L(p,q)$. Then at least one of $R_1,R_2$ is not primitive  or a proper-power in $H$.
\end{lem}

\begin{proof}
Suppose first, to the contrary, that both $R_1$ and $R_2$ are primitive in $H$. Then, since $R_ 1$ is primitive in $H$, there exists a complete set of meridian disks $\{D_A',D_B'\}$ of $H$ with the property that $R_1$ intersects $\partial D_A'$ in a single point of transverse intersection, and $R_1$ is disjoint from $D_B'$. Furthermore, among all complete sets of meridian disks of $H$, with $R_1 \cap \, \partial D_A'$ a single point of transverse intersection, and $R_1 \cap \, \partial D_B'$ = $\emptyset$, suppose $R_2 \cap \partial D_A'$ is minimal. This implies that the Heegaard diagram of $R_2$ with respect to $\{\partial D_A',\partial D_B'\}$ has no loops. And then, because $R_2$ is primitive in $H$, all of the intersections of $R_2$ with $\partial D_B$ must have the same sign. 

So suppose $|R_2 \cap \partial D_B| = n \geq 0$. Then, since $\mathcal{D}$ is a Heegaard diagram of $S^3$, $S^1 \times S^2$, or $(S^1 \times S^2)\; \#\; L(p,q)$, it must be the case that $n \leq 1.$

However, if $n = 0$, $\partial D_B$ is disjoint from $R_1$ and $R_2$, which implies $\mathcal{D}$ is either not connected, or has a cut-vertex, contrary to hypothesis. While, if $n = 1$, the boundary $\Gamma$ of a regular neighborhood of $R_2 \cup \partial D_B$ in $\partial H$ is an essential separating simple closed curve in $\partial H$ bounding a disk in $H$ such that $\Gamma$ is disjoint from $R_1$ and $R_2$. This again implies $\mathcal{D}$ is either not connected, or has a cut-vertex, contrary to hypothesis.

The remaining possibilities to be ruled out are that one member of $\{R_1,R_2\}$ is primitive in $H$,  while the other member of $\{R_1,R_2\}$ is a proper-power in $H$, or that both members of $\{R_1,R_2\}$ are proper-powers in $H$. It is shown in \cite{B09} that in both of these cases, either there exists an essential separating simple closed curve $\Gamma$ in $\partial H$, disjoint from $R_1$ and $R_2$, such that $\Gamma$ bounds a disk in $H$, or $R_1$ and $R_2$ are the two boundary components of a nonseparating annulus $\mathcal{A}$ in $H$. 

In the latter case, boundary compression of $\mathcal{A}$ in $H$ yields a nonseparating disk in $H$, disjoint from $R_1$ and $R_2$. It then follows, as above, that each of these possibilities implies $\mathcal{D}$ is either not connected, or has a cut-vertex, contrary to hypothesis.
\end{proof}

%%%%%%%%%%%%%

\section{Some properties of meridian representatives of tunnel-number-one knots in $S^3$, $S^1 \times S^2$, or $(S^1 \times S^2)\; \#\; L(p,q)$.}
\label{Some properties of meridian representatives}

\begin{lem}\textbf{\emph{If $\boldsymbol{m}$ is a meridian representative whose subdiagram is not connected or has a cut-vertex in a connected Heegaard diagram, without cut-vertices, of $\boldsymbol{S^3}$, $\boldsymbol{S^1 \times S^2}$, or $\boldsymbol{(S^1 \times S^2)\; \#\; L(p,q)}$, then $\boldsymbol{m}$ is distinguished.}}\hfill
\label{m_1 and m_2 are the only possible primitive or proper-power meridian reps}

Suppose $m$ and $R$ are disjoint nonseparating simple closed curves in the boundary of a genus two handlebody $H$, $H[R]$ has incompressible boundary, $H[R]$ embeds in $S^3$, $S^1 \times S^2$, or $(S^1 \times S^2)\; \#\; L(p,q)$, and $m$ represents the meridian of $H[R]$. In addition, suppose $\{D_A,D_B\}$ is a complete set of meridian disks for $H$ such that the Heegaard diagram $\mathcal{D}$ of $R$ with respect to $\{D_A,D_B\}$ is connected and has no cut-vertex, $\omega$ is the distinguished wave based at $R$ in $\mathcal{D}$, $m_1$ and $m_2$ are the pair of distinguished meridian representatives of $H[R]$ obtained by surgery on $R$ along $\omega$, and the Heegaard diagram $\mathcal{D}'$ of $m$ with respect to $\{D_A,D_B\}$ is not connected, or has a cut-vertex. Then $m$ is isotopic in $\partial H$ to one of $m_1, m_2$.
\end{lem}
 
\begin{proof}
Suppose $m$, $R$, $\partial D_A$, and $\partial D_B$ are all oriented. Then, focusing first on the Heegaard diagram $\mathcal{D}'$ of $m$ with respect to $\{D_A,D_B\}$, note that $m$ cannot have any loops in $\mathcal{D}'$, since the existence of loops in $\mathcal{D}'$ would force $\mathcal{D}$ to have a cut-vertex, contrary to hypothesis. It follows, since $\mathcal{D}'$ is loopless, but $\mathcal{D}'$ is not connected, or $\mathcal{D}'$ has a cut-vertex, that the signed intersections of $m$ with $\partial D_A$ all have the same sign, the signed intersections of $m$ with $\partial D_B$ all have the same sign, and there are either no edges of $m$ connecting vertices $A^+$ and $A^-$ of  $\mathcal{D}'$, or there are no edges of $m$ connecting vertices $B^+$ and $B^-$ of $\mathcal{D}'$.

Next, note that, since $m$ is disjoint from $R$ and $m$ is a meridian representative of $H[R]$, $m$ must intersect the distinguished wave $\omega$ based at $R$ algebraically zero times.

Now there are two cases:
\begin{enumerate}
\item $R$ has both positive and negative signed intersections with one of $\partial D_A$, $\partial D_B$.
\item The intersections of $R$ with $\partial D_A$ all have the same sign, and the intersections of $R$ with $\partial D_B$ all have the same sign.
\end{enumerate}

First, suppose Case 1 holds. Then there are intersections of $R$ with at least one of $\partial D_A$, $\partial D_B$ which have different signs and are adjacent along $\partial D_A$ or $\partial D_B$. In this case, the distinguished wave $\omega$ based at $R$ is isotopic to a subarc of $\partial D_A$ or $\partial D_B$. But then $|m \cap \,\omega| = 0$. So $m$ is disjoint from $\omega$, which means $m$ is isotopic to one of $m_1$, $m_2$ in $\partial H$.

Next, suppose Case 2 holds. In this case, the diagram $\mathcal{D}$ of $R$ is positive, and the distinguished wave $\omega$ based at $R$ is a horizontal wave $\omega_h$ in $\mathcal{D}$, as in Figure~\ref{DPCFig8bs}b. Now there are at most two bands of edges of $m$ in $\mathcal{D}'$ which intersect $\omega$, and these bands have a common vertex of $\mathcal{D}'$ as endpoint. It follows that all of the intersections of $m$ with $\omega$ have the same sign. So again $m$ must be disjoint from $\omega$, and hence isotopic to one of $m_1$, $m_2$ in $\partial H$.
\end{proof} 

\begin{cor}\textbf{\emph{At most two meridian representatives in diagrams of $\boldsymbol{S^3}$, $\boldsymbol{S^1 \times S^2}$, or $\boldsymbol{(S^1 \times S^2)\; \#\; L(p,q)}$ are primitive or proper-powers.}}\hfill
\label{primitive and proper-power meridian reps are distinguished}

Suppose $m$ and $R$ are disjoint nonseparating simple closed curves on the boundary of a genus two handlebody $H$ such that $H[R]$ has incompressible boundary, $H[R]$ embeds in $S^3$, $S^1 \times S^2$, or $(S^1 \times S^2)\; \#\; L(p,q)$, $m$ is primitive or a proper-power in $H$, and $m$ represents the meridian of $H[R]$. Then $m$ is isotopic in $\partial H$ to one of the two distinguished meridian representatives of $H[R]$. In particular, there are at most two meridian representatives of $H[R]$ that are primitive or proper-powers in $H$.
\end{cor}

\begin{proof}
This follows immediately from Lemma \ref{m_1 and m_2 are the only possible primitive or proper-power meridian reps} since, if $m$ is primitive or a proper-power in $H$, the Heegaard diagram $\mathcal{D}'$ of $m$ with respect to $\{D_A,D_B\}$ is either not connected, or has a cut-vertex. So $m$ satisfies the hypothesis of Lemma \ref{m_1 and m_2 are the only possible primitive or proper-power meridian reps}.
\end{proof}

\begin{lem}\textbf{\emph{The two distinguished meridian representatives are the two ``shortest'' meridian representatives of a tunnel of a tunnel-number-one knot in $\boldsymbol{S^3}$, $\boldsymbol{S^1 \times S^2}$, or $\boldsymbol{(S^1 \times S^2)\; \#\; L(p,q)}$.}}\hfill
\label{The two distinguished meridian representatives are the two ``shortest'' meridian representatives of a tunnel of tunnel-number-one knot}

Suppose $m$ and $R$ are disjoint nonseparating simple closed curves on the boundary of a genus two handlebody $H$, $H[R]$ is the exterior of a knot in $S^3$, $S^1 \times S^2$, or $(S^1 \times S^2)\; \#\; L(p,q)$, and $m$ represents the meridian of $H[R]$. In addition, suppose $m_1$ and $m_2$, are the two distinguished meridian representatives of $H[R]$, obtained by surgery on $R$ along the distinguished wave $\omega$ based at $R$, and suppose $\{D_A,D_B\}$ is any complete set of meridian disks for $H$, such that the Heegaard diagram $\mathcal{D}$ of $R$ with respect to $\{D_A,D_B\}$ is connected and has no cut-vertex. Finally, suppose $m$ is not isotopic to $m_1$ or $m_2$ in $\partial H$.

Then there is a subarc $\omega'$ of $\omega$ which is a wave based at $m$ in $\mathcal{D}$. And then Lemma \emph{\ref{m_1 and m_2 are the only possible primitive or proper-power meridian reps}} implies that the Heegaard diagram of $m$ with respect to $\{D_A,D_B\}$ is connected and has no cut-vertex. It follows $\omega'$ is a distinguished wave based at $m$, and $R$ is obtained from $m$ by surgery on $m$ along $\omega'$. This means the complexity of the Heegaard diagram of $m$ with respect to $\{D_A,D_B\}$ is greater than the complexity of the Heegaard diagram of $R$ with respect to $\{D_A,D_B\}$, which in turn is greater than the complexities of the Heegaard diagrams of either $m_1$ or $m_2$ with respect to $\{D_A,D_B\}$.
\end{lem}

\begin{proof}
By Theorem \ref{Thm Distinguished waves are well-defined}, the distinguished  wave $\omega$ based at $R$ depends only on $R$ and $H$, and $\omega$ appears as a wave based at $R$ in any connected Heegaard diagram $\mathcal{D}$ without cut-vertices of $R$ on $\partial H$. If $m$ is any meridian representative of $H[R]$ in $\partial H$, disjoint from $R$, then $m$ must have algebraic intersection number zero with $\omega$. So, if $m$ is not disjoint from $\omega$, there exists a subarc $\omega'$ of $\omega$ which appears in $\mathcal{D}$ as a wave based at $m$. Now Lemma \ref{m_1 and m_2 are the only possible primitive or proper-power meridian reps} shows that, because $m$ is not isotopic to $m_1$, or $m_2$, the Heegaard diagram of $m$ with respect to $\{D_A,D_B\}$ is connected and has no cut-vertex. By Theorem \ref{Only distinguished waves provide meridians}, $\omega'$ is the distinguished wave based at $m$, and $R$ is obtained from $m$ by surgery on $m$ along $\omega'$. This implies the complexities of the Heegaard diagrams of $m$, $R$, $m_1$ and $m_2$, with respect to $\{D_A,D_B\}$ are related as claimed.
\end{proof}

\subsection{\textbf{The ``shortest'' meridian representative of a tunnel of a tunnel-number-one knot in $\boldsymbol{S^3}$, $\boldsymbol{S^1 \times S^2}$, or $\boldsymbol{(S^1 \times S^2)\; \#\; L(p,q)}$.}} \hfill

When at least one of the two distinguished meridian representatives $m_1$, $m_2$ in $\partial H$ of a tunnel of a tunnel-number-one knot $k$ in a 3-manifold $M$ homeomorphic to $S^3$, $S^1 \times S^2$, or $(S^1 \times S^2)\; \#\; L(p,q)$ is not primitive or a proper-power in $H$, more can be said. Namely the ``shortest'' meridian representative of $k$ exists and is well-defined.

For, if one of $m_1$, $m_2$, say $m_1$, is not primitive or a proper-power in $H$, then there exists a complete set $\Delta$ of meridian disks of $H$ and Heegaard diagram $\mathcal{D}$ of $m_1$ and $m_2$ with respect to $\Delta$ such that the subdiagram of $\mathcal{D}$ consisting of the edges of $m_1$ in $\mathcal{D}$ is connected and has no cut-vertex. So $\mathcal{D}$ itself is connected, has no cut-vertex, and is a Heegaard diagram of $M$. Then Theorem \ref{A distinguished wave appears in each connected H-diagram of S^3 without cut-vertices} applies and shows that there is a distinguished wave $\omega$ based at one of $m_1$, $m_2$ in $\mathcal{D}$. Finally, the results of Section \ref{Distinguished waves in genus two Heegaard diagrams of closed manifolds} show that $\omega$ appears in each Heegaard diagram of $m_1$ and $m_2$ which is connected and has no cut-vertex. This justifies the following definition.

\begin{defn}\textbf{The ``shortest'' meridian representative of a tunnel of a tunnel-number-one knot in $\boldsymbol{S^3}$, $\boldsymbol{S^1 \times S^2}$, or $\boldsymbol{(S^1 \times S^2)\; \#\; L(p,q)}$.} \hfill
\label{Shortest meridian rep}

Suppose $R$ is a nonseparating simple closed curve on the boundary of a genus two handlebody, $H[R]$ has incompressible boundary, $H[R]$ is the exterior of a knot in a 3-manifold $M$ homeomorphic to $S^3$, $S^1 \times S^2$, or $(S^1 \times S^2)\; \#\; L(p,q)$, $\omega$ is the distinguished wave based at $R$, and curves $m_1$ and $m_2$, obtained by surgery on $R$ along $\omega$, are the two distinguished meridian representatives in $\partial H$ of the meridian of $H[R]$ in $M$. 

If at least one of $m_1, m_2$ is not primitive or a proper-power in $H$, there exists a connected Heegaard diagram $\mathcal{D}$, without cut-vertices, of $m_1$ and $m_2$ on $\partial H$, such that a distinguished wave $\omega'$, based at one of $m_1,m_2$, appears in $\mathcal{D}$. Then the \emph{shortest} meridian representative of $H[R]$ is the member of $\{m_1,m_2\}$ at which $\omega'$ is not based.
\end{defn}
 
%%%%%%%%%%%%%%%
 
\section{Unknotting tunnels and canonical constituent knots of theta curves in $S^3$, $S^1 \times S^2$, or $(S^1 \times S^2)\; \#\; L(p,q)$.}\hfill
\label{Unknotting tunnels and canonical constituent knots of theta curves}

Suppose $M$ is a 3-manifold homeomorphic to one of $(S^1 \times S^2)\; \#\; L(p,q)$, $S^3$ or $S^1 \times S^2$. If $k$ is a knot in $M$, an unknotting tunnel for $k$ is an arc $\tau$ in $M$ meeting $k$ only in its endpoints such that the exterior of $k \, \cup \, \tau$ in $M$ is a genus two handlebody. When $\tau$ is an unknotting tunnel for a knot $k$, $k \cup \tau$ is a \emph{$\theta$ curve} in $M$, a regular neighborhood of $k \cup \tau$ in $M$ is a genus two handlebody $H'$, and the exterior of $k \cup \tau$ is also a genus two handlebody $H$ such that $H$ and $H'$ together form a genus two Heegaard splitting of $M$ in which the exterior of $k$ in $M$ is obtained by gluing the co-core of $\tau$ to $\partial H$ along a nonseparating simple closed curve $R$ in $\partial H$. 

The endpoints of $\tau$ on $k$ cut $k$ into a pair of arcs, say $\alpha_1$ and $\alpha_2$, in which case, $\tau \cup \alpha_1$ and $\tau \cup \alpha_2$ are said to be \emph{constituent knots} $k_1$ and $k_2$ of the $\theta$ curve $k \cup \tau$. However, in general, constituent knots $k_1$ and $k_2$ of a $\theta$ curve are not well defined, as isotopies of $\tau$ which slide the endpoints of $\tau$ around on $k$ and perhaps pass each other on $k$ can change $k_1$ and $k_2$.

If $H[R]$ has incompressible boundary, i.e. $R$ has essential intersections with each meridional disk in $H$, the following procedure can be used to find a pair of canonical constituent knots $k_1$ and $k_2$ of the $\theta$ curve $k_\tau$ such that $k_1$ and $k_2$ depend only on $k$ and $\tau$.

\begin{enumerate}
\item Locate $R$ on $\partial H$.
\item Locate a complete set, say $\{D_A, D_B\}$, of meridian disks for $H$ such that the Heegaard diagram $\mathcal{D}$ of $R$ with respect to $\{D_A, D_B\}$ is connected and has no cut-vertex.
\item Locate the  distinguished wave $\omega$ based at $R$ in $\mathcal{D}$.
\item Let $N(R \cup \omega)$ be a regular neighborhood of $R \cup \omega$ in $\partial H$, and let $m_1$ and $m_2$ be the two boundary components of $N$ not isotopic to $R$ in $\partial H$.
\item Then $m_1$ and $m_2$ are representatives of the meridian of $H[R]$ in $M$. So $m_1$ and $m_2$ bound meridian disks $D_{m_1}$ and $D_{m_2}$ of $H'$ such that $H'$ cut open along $D_{m_1}$ is a solid torus $T_2$ with meridional disk $D_{m_2}$, and $H'$ cut open along $D_{m_2}$ is a solid torus $T_1$ with meridional disk $D_{m_1}$. Then the core curves $k_1$ of $T_1$, transverse to $D_{m_1}$ at a point, and $k_2$ of $T_2$, transverse to $D_{m_2}$ at a point, are a canonical pair of constituent knots of the $\theta$ curve $k \cup \tau$ since the pair $(k_1, k_2)$ depends only on $k$ and $\tau$.
\end{enumerate}
  
\begin{cor}
\label{(1,1) iff a canonical constituent knot is unknotted}
An unknotting tunnel $\tau$ is a \emph{(1,1)} tunnel of $k$ if and only if one of the canonical constituent knots of the $\theta$ curve $k \cup \tau$ is unknotted.
\end{cor}

\begin{proof}
This follows directly from Theorem \ref{(1,1) tunnel recognition}.
\end{proof}

\begin{rem}
Corollary \ref{(1,1) iff a canonical constituent knot is unknotted} is not true if the restriction to canonical constituent knots is removed. Compare \cite{MS91}
[Proposition (1.3)], which is wrong because the restriction to canonical constituent knots is omitted. 
\end{rem}

%%%%%%%%%%%%%%

\section{Recognizing (1,1) tunnels of (1,1) knots in $S^3$ or $S^1 \times S^2$.}\hfill
\label{Recognizing (1,1) tunnels of (1,1) knots}

Suppose $M$ is a manifold homeomorphic to $S^3$ or $S^1 \times S^2$. A knot $k$ in $M$ is in (1,1) position if $k$ intersects each solid torus $V$, $V'$ of a genus one Heegaard splitting of $M$ in a single unknotted arc. A knot in $M$ which has a (1,1) position is a (1,1) knot. If $k$ is in (1,1) position, then $k$ has an unknotting tunnel $t_V$, respectively $t_{V'}$, such that $t_V$, respectively $t_{V'}$, is a core curve of $V$, respectively $V'$. Such unknotting tunnels $t_V$ and $t_{V'}$ of $k$ form a \emph{dual pair} of (1,1) unknotting tunnels of $k$. 

A (1,1) knot may have unknotting tunnels which are not (1,1) tunnels. A tunnel of a tunnel-number-one knot $k$ which is not a (1,1) tunnel of $k$ is a \emph{regular} tunnel of $k$.

If $k$ is in (1,1) position, drilling the unknotted arc $k \,\cap \,V$ out of $V$, or the unknotted arc $k \cap V'$ out of $V'$, converts $V$, respectively $V'$, into a genus two handlebody $H$, respectively $H'$, in which the boundary of the co-core of each drilled out arc is a meridian of $k$ that is primitive in $H$, respectively $\pi_1(H')$. 

The following theorem provides a way to recognize that a tunnel is a (1,1) tunnel of a (1,1) knot.

\begin{thm}\textbf{\emph{(1,1) tunnel recognition.}} \hfill
\label{(1,1) tunnel recognition}

Suppose $R$ is a nonseparating simple closed curve in the boundary of a genus two handlebody $H$ such that $H[R]$ has incompressible boundary and $H[R]$ is the exterior of a knot $k$ in a 3-manifold $M$ homeomorphic to $S^3$ or $S^1 \times S^2$. In addition, let $m_1$ and $m_2$, be the two simple closed curve distinguished meridian representatives obtained from $R$ by surgery on $R$ along the distinguished wave $\omega$ based at $R$ in $\partial H$. Then $R$ is the boundary of a co-core of a \emph{(1,1)} tunnel for $k$, if and only if one of $m_1$, $m_2$ is primitive in $H$.
\end{thm}

\begin{proof}
Suppose one of $m_1$, $m_2$, say $m_1$, is primitive in $H$. Then filling $H[R]$ along $m_1$ yields a genus two Heegaard splitting $(\Sigma;H,H')$ of $M$ in which $m_1$ and $R$ bounding a complete set of meridian disks of the complementary genus two handlebody $H'$. Furthermore, since $m_1$ is primitive in $H$, there exists a meridian disk $D$ of $H$ such that $m_1 \cap \, D$ is a single point of transverse intersection. Then the meridian disk bounded by $m_1$ in $H'$ and the disk $D$ in $H$ form a trivial handle pair in $(\Sigma;H,H')$. So $(\Sigma;H,H')$ destabilizes to a genus one Heegaard splitting $(\Sigma';V,V')$ of $M$.

Next, the knot $k$ can be isotoped so it lies in $\Sigma$, is disjoint from $R$, and has $k \cap \, m_1$ a single point of transverse intersection. Then $k$ is clearly a (1,1) knot in $(\Sigma';V,V')$, and there exists an unknotting tunnel $\tau$ for $k$ in $\Sigma$ such that $\tau \cap R$ is a single point of transverse intersection. Since the meridian disk of $H'$ bounded by $R$ becomes the meridian disk of $V'$, $\tau$ becomes a core curve in $V'$. So $\tau$ is a (1,1) unknotting tunnel for $k$. This proves one direction of Theorem \ref{(1,1) tunnel recognition}.

For the other direction of Theorem \ref{(1,1) tunnel recognition}, suppose $k$ is in (1,1) position with respect to a genus one Heegaard splitting $(\Sigma';V,V')$ of $M$, and $\tau$ is a (1,1) tunnel of $k$ in $(\Sigma';V,V')$. Then $k\, \cap \, V$, respectively $k \, \cap \, V'$, is an unknotted arc in $V$, respectively $V'$. So drilling the unknotted arc $k \,\cap \, V$ out of $V$ turns $V$ into a genus two handlebody $H$ in a genus two stabilization  $(\Sigma;H,H')$ of $(\Sigma';V,V')$. 

When $(\Sigma';V,V')$ is stabilized to $(\Sigma;H,H')$, the curve $R$ in $\Sigma'$, bounding the meridian disk of $V'$, appears in $\Sigma$ as a curve bounding a meridian disk of $H'$. And, since the arc $k \, \cap \, V'$ is unknotted in $V'$, $k \, \cap \, V'$ can be isotoped, keeping its endpoints fixed, into $\Sigma'$ missing $R$. Then in $(\Sigma;H,H')$, $k$ and $R$ are disjoint curves in $\partial H$, the boundary $m$ of the co-core of the drilled out arc $k \, \cap \, V$ lies in $\partial H$, $m$ is disjoint from $R$, $m$ is primitive in $H$, and $m$ meets $k$ in a single point of transverse intersection.

Now Lemma \ref{m_1 and m_2 are the only possible primitive or proper-power meridian reps} shows that if $m$ is primitive in $H$, then $m$ must be isotopic in $\partial H$ to one of the two distinguished meridian representatives $m_1$ and $m_2$ of $H[R]$ obtained by surgery on $R$ along the distinguished wave $\omega$ based at $R$. So either $m_1$ or $m_2$ must be primitive in $H$.
\end{proof}
 
%%%%%%%%%%%%
 
\section{Recognizing genus two Heegaard diagrams of $S^3$, $S^1 \times S^2$, or $(S^1 \times S^2)\; \#\; L(p,q)$.}\hfill
\label{Recognizing S^3 etc}

The following procedure uses distinguished waves and shortest meridian representatives, along with the major results of \cite{HOT80} and \cite{NO85} that every connected genus two Heegaard diagram of $S^3$, $S^1 \times S^2$, or $(S^1 \times S^2)\; \#\; L(p,q)$ has a wave, to determine if a genus two Heegaard diagram of a closed 3-manifold is a Heegaard diagram of $S^3$, $(S^1 \times S^2)\; \#\; L(p,q)$ or $S^1 \times S^2$. 

\subsection{A procedure to recognized genus two Heegaard diagrams of $\boldsymbol{S^3}$, $\boldsymbol{S^1 \times S^2}$, or $\boldsymbol{(S^1 \times S^2)\; \#\; L(p,q)}$}\hfill
\label{S^3, S1 X S2, and (S1 X S2) + L(p,q) recognition}

Suppose $R_1$ and $R_2$ are disjoint, nonparallel, nonseparating simple closed curves on the boundary of a genus two handlebody $H$, and let $M$ be the closed 3-manifold  obtained by first adding 2-handles to $\partial H$ along $R_1$ and $R_2$ to obtain $H[R_1,R_2]$, and then filling the 2-sphere boundary of $H[R_1,R_2]$ with a 3-ball. Then, leaving aside the case in which $R_1$ or $R_2$ bounds a disk in $H$, the following procedure can be used to determine if $M$ is homeomorphic to $S^3$, $S^1 \times S^2$, or $(S^1 \times S^2)\; \#\; L(p,q)$.

\begin{enumerate}
\item Check, using say Theorem \ref{recognizing primitives}, whether $R_1$ or $R_2$ is primitive or a proper-power in $H$.
\item If $R_1$ or $R_2$ is primitive or a proper-power in $H$, then:
\begin{enumerate}
\item If $H_1(M,\mathbb{Z}) \neq 0, \, \mathbb{Z}$ or $\mathbb{Z} \oplus \mathbb{Z}/p, \; p > 1$, $M$ is not homeomorphic to $S^3$, $S^1 \times S^2$, or $(S^1 \times S^2)\; \#\; L(p,q)$. Stop!
\item If $H_1(M,\mathbb{Z}) = 0$, $M \simeq S^3$; 
\item If $H_1(M,\mathbb{Z}) = \mathbb{Z}$, $M \simeq S^1 \times S^2$;
\item If $H_1(M,\mathbb{Z}) = \mathbb{Z} \oplus \mathbb{Z}/p, \; p > 1$, $M \simeq (S^1 \times S^2)\; \#\; L(p,q)$.
\item Stop!
\end{enumerate}
\item Locate a complete set $\Delta$ of meridional disks of $H$ such that the Heegaard diagram $\mathcal{D}$ of $\{R_1,R_2\}$ with respect to $\Delta$ is connected and has no cut-vertex.
\item Determine if there is a wave $\omega$ in $\mathcal{D}$ based at $R_1$ or $R_2$.
\item If no wave based at $R_1$ or $R_2$ in $\mathcal{D}$ exists, $M$ is not homeomorphic to $S^3$, $S^1 \times S^2$, or $(S^1 \times S^2)\; \#\; L(p,q)$. Stop!
\item If there is a wave $\omega$ based at $R_1$ or $R_2$ in $\mathcal{D}$, and $\omega$ is based at $R_1$, respectively $R_2$, take $R_2$, respectively $R_1$, to be the shortest member of $\{R_1,R_2\}$ in $\mathcal{D}$.
\item If, say $R_1$, is the shortest member of $\{R_1,R_2\}$, look for a complete set of meridian disks $\Delta'$ of $H$ such that the Heegaard diagram $\mathcal{D}'$ of $R_1$ with respect to $\Delta'$ is connected and has no cut-vertex.
\item Locate the distinguished wave $\omega'$ based at $R_1$ in $\mathcal{D}'$.
\item If $R_2$ does not have algebraic intersection number zero with $\omega'$ in $\mathcal{D}'$, $R_2$ does not represent the meridian of $H[R_1]$ and $\mathcal{D}'$ is not a Heegaard diagram of $S^3$, $S^1 \times S^2$, or $(S^1 \times S^2)\; \#\; L(p,q)$. Stop!
\item If $R_2$ has algebraic intersection number zero with $\omega'$ in $\mathcal{D}'$, perform surgery on $R_1$ along $\omega'$ to obtain the distinguished meridian representatives $m_1$ and $m_2$ of $H[R_1]$. Replace $\{R_1,R_2\}$ with $\{m_1,m_2\}$, and go to (1) above.
\end{enumerate}

\begin{thm}Procedure \emph{\ref{S^3, S1 X S2, and (S1 X S2) + L(p,q) recognition}} correctly determines if genus two Heegaard diagrams are Heegaard diagrams of $S^3$, $S^1 \times S^2$ or $(S^1 \times S^2)\; \#\; L(p,q)$.
\end{thm}

\begin{proof}
Most of the proof that Procedure \ref{S^3, S1 X S2, and (S1 X S2) + L(p,q) recognition} correctly recognizes genus two Heegaard diagram of $S^3$, $S^1 \times S^2$ or $(S^1 \times S^2)\; \#\; L(p,q)$ follows directly from the major results of \cite{HOT80} and \cite{NO85} that every connected genus two Heegaard diagram of $S^3$, $S^1 \times S^2$ or $(S^1 \times S^2)\; \#\; L(p,q)$ has a wave. However, we mention two points.

(1) Continuing with the notation of Procedure \ref{S^3, S1 X S2, and (S1 X S2) + L(p,q) recognition}, observe that if $\omega'$ is the distinguished wave based at $R_1$ in Step (8) of Procedure \ref{Recognizing S^3 etc} and $R_2$ intersects $\omega$ algebraically zero times in Step (9) of Procedure \ref{S^3, S1 X S2, and (S1 X S2) + L(p,q) recognition}, then the pair of curves $(R_1,R_2)$ and the pair of curves $(m_1,m_2)$ have the same normal closure in $\pi_1(\partial H)$. It follows that $R_1$ and $R_2$ are the curves of a genus two Heegaard diagram of $M$ if and only if $m_1$ and $m_2$ are the curves of a genus two Heegaard diagram of $M$. Furthermore, the Heegaard diagram of $\{m_1,m_2\}$ with respect to the complete set of meridian disks $\Delta'$ for $H$ in Step (7) of Procedure \ref{S^3, S1 X S2, and (S1 X S2) + L(p,q) recognition} has smaller complexity than the Heegaard diagram of $\{R_1,R_2\}$ with respect to the complete set of meridian disks $\Delta'$ for $H$.

(2) If $R_2$ does not intersect $\omega'$ algebraically zero times in Step (9) of Procedure ~\ref{S^3, S1 X S2, and (S1 X S2) + L(p,q) recognition}, Lemma \ref{waves provide meridians} together with Lemma \ref{horizontal wave lemma} imply that the Heegaard diagram $\mathcal{D}'$ of $\{R_1,R_2\}$ on $\partial H$ in Step (7) of Procedure \ref{S^3, S1 X S2, and (S1 X S2) + L(p,q) recognition} is not a Heegaard diagram of $S^3$, $S^1 \times S^2$ or $(S^1 \times S^2)\; \#\; L(p,q)$.
\end{proof}

%%%%%%%%%%%%%%%%%
 
\section{Computing the depth of a tunnel of a tunnel-number-one knot in $S^3$, $S^1 \times S^2$, or $(S^1 \times S^2)\; \#\; L(p,q)$.}
\label{Computing depth}

Suppose $M$ is a manifold homeomorphic to one of $(S^1 \times S^2)\; \#\; L(p,q)$, $S^3$ or $S^1 \times S^2$, and suppose $R$ is a nonseparating simple closed curve on the boundary of a genus two handlebody $H$ such that the manifold $H[R]$, obtained by adding a 2-handle to $H$ along $R$, is the exterior of a knot $k$ in $M$. In addition, let $\tau$ be the unknotting tunnel of $k$ dual to $R$ on $\partial H$. In \cite{CM1}--\cite{CM9} Cho and McCullogh described procedures for computing the ``depth'' of an unknotting tunnel of a knot in $S^3$, along with procedures for computing sets of ``rational invariants'' of unknotting tunnels of knots in $S^3$. Later, in \cite{I11}, Ishihara adapted these procedures to compute the depth and rational invariants of knots in $S^3$ using Heegaard diagrams of $S^3$. The procedure described in \ref{Depth computation} below uses distinguished waves in genus two Heegaard diagrams of $M$ to compute the depth of tunnel-number-one knots in $M$, and so generalizes the depth computation procedures of Cho, McCullogh and Ishihara from $S^3$ to $S^1 \times S^2$ and $(S^1 \times S^2)\; \#\; L(p,q)$.

\begin{defn}\textbf{Unknotting paths of an unknotting tunnel.}\hfill
\label{Unknotting paths of an unknotting tunnel}

Suppose $R$ is a nonseparating simple closed curve on the boundary of a genus two handlebody $H$ such that the manifold $H[R]$, obtained by adding a 2-handle to $H$ along $R$, is the exterior of a knot $k$ in a manifold $M$ homeomorphic to $S^3$, $S^1 \times S^2$, or $(S^1 \times S^2)\; \#\; L(p,q)$. Then $R$ determines an unknotting tunnel $\tau$ for $k$, and there is an \emph{unknotting path} of $\tau$ consisting of a finite sequence $R_i$, $0 \leq i \leq n$, of simple closed curves on $\partial H$, starting with $R_0$ = $R$ and ending with $R_n$ primitive or a proper-power in $H$, such that for $0 \leq i < n$, $R_i$ and $R_{i+1}$ are disjoint and $R_{i+1}$ represents the meridian of $H[R_i]$ in $M$.
\end{defn}

\begin{defn}\textbf{Depth of an unknotting tunnel.} \hfill

The \emph{depth} of a tunnel $\tau$ of a tunnel-number-one knot $k$ in $S^3$, $S^1 \times S^2$, or $(S^1 \times S^2)\; \#\; L(p,q)$ is the minimal value of $n$ that occurs over all unknotting paths of $\tau$.
\end{defn}

%%%%%%%%%%%%%%%

\subsection{A procedure to compute the depth of an unknotting tunnel.}\hfill
\label{Depth computation}

Suppose $R$ is a nonseparating simple closed curve on the boundary of a genus two handlebody $H$ such that the manifold $H[R]$, obtained by adding a 2-handle to $H$ along $R$, is the exterior of a knot $k$ in $S^3$, $S^1 \times S^2$, or $(S^1 \times S^2)\; \#\; L(p,q)$, and $\tau$ is the unknotting tunnel of $k$ dual to $R$ on $\partial H$. Then the following procedure computes the depth of $\tau$ by producing a well-defined unique minimal length sequence $\mathcal{P} = \{R_i\}$, $0 \leq i \leq n$, of simple closed curves on $\partial H$, starting with $R_0$ = $R$ and ending with $R_n$ primitive or a proper-power in $H$, such that for $0 \leq i < n$, $R_i$ and $R_{i+1}$ are disjoint and $R_{i+1}$ represents the meridian of $H[R_i]$.

\begin{enumerate}
\item Set depth = 0.
\item Determine if $R$ is primitive or a proper-power in $H$, using, say, Theorem \ref{recognizing primitives}.
\item If $R$ is primitive or a proper-power in $H$, stop.
\item Locate a complete set $\{D_A, D_B \}$ of meridional disks of $H$ such that the Heegaard diagram $\mathcal{D}$ of $R$ with respect to $\{D_A, D_B \}$ is connected and has no cut-vertex.
\item Locate the distinguished wave $\omega$ based at $R$ in $\mathcal{D}$, and perform surgery on $R$ along $\omega$ to obtain the distinguished meridian representatives $m_1$ and $m_2$ of $H[R]$.
\item  If the Heegaard diagram $\mathcal{D}'$ of $\{m_1, m_2\}$ with respect to $\{D_A, D_B \}$ is not connected, both $m_1$ and $m_2$ are primitive or proper-powers in $H$. In this case, increment depth by 1, replace $R$ with either $m_1$ or $m_2$, and go to step (2) above.
\item Otherwise, determine which member of $\{m_1,m_2\}$ is shortest, replace $R$ with the shortest member of $\{m_1,m_2\}$, increment depth by 1, and go to step (2) above.
\end{enumerate}

Note that each time the current curve $R$ is replaced in Procedure \ref{Depth computation}, it is replaced by a curve of smaller complexity. So Procedure \ref{Depth computation} terminates in a finite number of steps.

Next, Lemma \ref{Minimal length unknotting paths always pass from a curve to one of the curves distinguished meridian representatives} shows that minimal length unknotting paths always pass from a curve $c$ to one of the two distinguished meridian representatives of $H[c]$. This implies there exists a finite graph $G$ such that $G$ carries all possible minimal length unknotting paths, as per Definition \ref{An unknotting graph G}, of a given initial curve. Then analysis of the minimal length unknotting paths in $G$ will establish Theorem \ref{Depth computation works} below.

\begin{thm}
\label{Depth computation works}
Procedure \emph{\ref{Depth computation}} correctly computes depth.
\end{thm}

\begin{lem}
Minimal length unknotting paths always pass from a curve $c$ to one of the two distinguished meridian representatives of $H[c]$.
\label{Minimal length unknotting paths always pass from a curve to one of the curves distinguished meridian representatives}
\end{lem}

\begin{proof}
Suppose, to the contrary, that there exists a minimal length unknotting path, say $\mathcal{P}$, which at some point does not pass from a curve $c$ to one of the two distinguished meridian representatives of $H[c]$. 

Then, consider Figure \ref{WaveDefFig23b}, and suppose Figure \ref{WaveDefFig23b}a has the following properties:
\begin{enumerate}
\item Vertices $\mathrm{A}$ and $\mathrm{B}$ of Figure \ref{WaveDefFig23b}a are successive vertices of $\mathcal{P}$.
\item The curve $R_B$ in $\partial H$, represented by vertex $\mathrm{B}$ of $\mathcal{P}$, is disjoint from the curve $R_A$ in $\partial H$, represented by vertex $\mathrm{A}$ of $\mathcal{P}$, $R_B$ is a meridian of $H[R_A]$, but $R_B$ is not one of the two distinguished meridian representatives of $H[R_A]$.
\item $\mathrm{A}$ and $\mathrm{B}$ are the last pair of successive vertices in $\mathcal{P}$ which satisfy (2).
\end{enumerate}
Lemma \ref{m_1 and m_2 are the only possible primitive or proper-power meridian reps} shows that if $c$ is a curve disjoint from $R_A$, $c$ represents a meridian of $H[R_A]$, and $c$ is primitive or a proper-power in $H$, then $c$ is one of the two distinguished meridian representatives of $H[R_A]$. It follows the curve $R_B$ is not primitive or a proper-power in $H$. Thus $B$ is not the terminal vertex of $\mathcal{P}$, and so $B$ is followed immediately in $\mathcal{P}$ by a vertex $\mathrm{C}$.

Next, since (3) holds, $R_C$ must be one of the two distinguished meridian representatives of $H[R_B]$. In addition, since $R_A$ and $R_B$ are disjoint, and neither $R_A$ nor $R_B$ is primitive or a proper-power in $H$, one of $R_A$, $R_B$ must have a distinguished wave missing the other. Since $R_B$ is not one of the two distinguished meridian representatives of $R_A$, there must be a distinguished wave based at $R_B$ missing $R_A$. But then $R_A$ and $R_C$ are both distinguished meridian representatives of $H[R_B]$. Now $R_C$ can't be isotopic to $R_A$ in $\partial H$, since otherwise $\mathcal{P}$ backtracks from $A \rightarrow B \rightarrow A$. So $R_A$ and $R_C$ are disjoint, $R_C$ is a meridian representative of $H[R_A]$, and $\mathcal{P}$ can be shortened by going directly from $A$ to $C$, as suggested by Figure~\ref{WaveDefFig23b}b.
\end{proof}

\begin{figure}[ht]
\includegraphics[width = 0.5\textwidth]{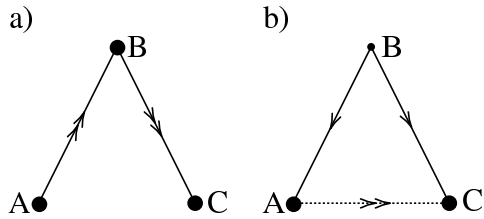}
\caption{\textbf{Minimal length unknotting paths of a tunnel-number-one
knot in $\boldsymbol{S^3}$, $\boldsymbol{S^1 \times S^2}$, or $\boldsymbol{(S^1 \times S^2)\; \#\; L(p,q)}$ always pass from a curve $\boldsymbol{c}$ to a curve which is one of the two distinguished meridian representatives of $\boldsymbol{H[c]}$.}}
\label{WaveDefFig23b}
\end{figure}

The following result of Cohen, Metzler, and Zimmerman makes it possible to determine if a cyclically reduced word in a free group of rank two is primitive or a proper-power of a primitive. 

\begin{thm}[CMZ81]
\label{recognizing primitives}
Suppose a cyclic conjugate of 
\[w = A^{m_1}B^{n_1} \dots A^{m_j}B^{n_j}\]
 is a member of a basis of $F(A,B)$, where $j \geq 1$ and each indicated exponent is nonzero. Then, after perhaps replacing $A$ by $A^{-1}$ or $B$ by $B^{-1}$, there exists $e > 0$ such that:
\[
m_1 = \dots = m_j = 1,
\quad
\text{and}
\quad
\{n_1, \dots ,n_j\} = \{e, e+1\},
\]
or
\[
\{m_1, \dots ,m_j\} = \{e, e+1\},
\quad
\text{and}
\quad
n_1 = \dots = n_j = 1.
\]
\end{thm}

Note that if $w$ in $F(A,B)$ has the form $w = AB^{n_1} \dots AB^{n_j}$, say, with $j \geq 1$ and $\{n_1, \dots ,n_j\} = \{e, e+1\}$, then the automorphism $A \mapsto AB^{-e}$ of $F(A,B)$ reduces the length of $w$, so repeated applications of such automorphisms can be used to determine if a given word $w$ in $F(A,B)$ is a primitive.

\subsection{A graph carrying all minimal length unknotting paths.}\hfill
\label{An unknotting graph G}

Lemma \ref{Minimal length unknotting paths always pass from a curve to one of the curves distinguished meridian representatives} shows that minimal length unknotting paths always pass from a curve $c$ to one of the two distinguished meridian representatives of $H[c]$. This implies there exists a finite graph $G$ with the property that $G$ carries all possible minimal length unknotting paths, as per Definition \ref{An unknotting graph G}, of a given initial curve.

The graph $G$ can be produced by starting with an initial vertex representing the initial curve whose minimal length unknotting paths are being sought, and applying the following procedure recursively to add new vertices and edges to $G$ until the procedure fails to add new vertices to $G$.
\begin{enumerate}
\item If $V$ is a vertex of $G$, there are no directed edges of $G$ leaving $V$, and the curve $R_V$ in $\partial H$ represented by $V$, is not primitive in $H$, then:
\begin{enumerate}
\item Let $m_1$ and $m_2$ be the two distinguished meridian representatives of $H[R_V]$ obtained from $R_V$ by performing surgery on $R_V$ along the distinguished wave based at $R_V$ in $\partial H$, and if $m_1$, respectively $m_2$, is not represented by a vertex of $G$, add a vertex, $V_1$ representing $m_1$, respectively, $V_2$ representing $m_2$, to $G$.
\item Add a directed edge to $G$ which leaves $V$ and connects $V$ to the vertex of $G$ representing $m_1$. Similarly, add a directed edge to $G$ which leaves $V$ and connects $V$ to the vertex of $G$ representing $m_2$.
\end{enumerate}
\item If $V$ is a vertex of $G$, there are no directed edges of $G$ leaving $V$, and the curve $R_V$ in $\partial H$ represented by $V$ is primitive in $H$, do nothing.
\end{enumerate}

\subsection{A subgraph $\boldsymbol{G^*}$ of $\boldsymbol{G}$ which carries the unknotting path described in Subsection \ref{Depth computation}}\hfill
\label{Subgraph G^* of G}

Subsection \ref{Depth computation} describes a procedure for computing a particular
unknotting path $\mathcal{P}$ of a given initial curve. This subsection describes a subgraph $G^*$ of $G$ whose vertices are the simple closed curves in $\partial H$ examined while computing $\mathcal{P}$. 

The unknotting path $\mathcal{P}$, as described in \ref{Depth computation}, is a sequence $\mathcal{P} = \{R_i\}$, $0 \leq i \leq n$, of simple closed curves on $\partial H$ which starts with an initial curve $R_0$ and ends with a curve $R_n$, primitive or a proper-power in $H$, such that for $0 \leq i < n$, $R_i$ and $R_{i+1}$ are disjoint, and at each step with $0 \leq i < n$, $R_{i+1}$ is obtained from $R_i$ by choosing one of the two distinguished meridian representatives of $H[R_i]$ in $\partial H$ as $R_{i+1}$. 

Given $\mathcal{P}$, let the vertices of $G^*$ be the vertices of $G$ which represent $R_0$ and both distinguished meridian representatives of each $R_i$ in $\mathcal{P}$, for $0 \leq i < n$. In addition, if $V_i$, $0 \leq i < n$, is a vertex of $G^*$ which represents a curve in $\mathcal{P}$, add the pair of directed edges of $G$ leaving $V_i$ to the set of directed edges of $G^*$. 

Next, recall that, when $R_i$, $0 \leq i < n$, is a curve in $\mathcal{P}$, one member of the pair of distinguished meridian representatives of $H[R_i]$ lies in $\mathcal{P}$, the other member of the pair of distinguished meridian representatives of $H[R_i]$ does not lie in $\mathcal{P}$. In particular, if $m_{i,1}$ and $m_{i,2}$, are the two distinguished meridian representatives of $H[R_i]$, $m_{i,1} \in \mathcal{P}$, $m_{i,2} \notin \mathcal{P}$, $m_{i,2}$ is not primitive or a proper-power in $H$, $V_{i,1}$ is the vertex of $G^*$ representing $m_{i,1}$, and $V_{i,2}$ is the vertex of $G^*$ representing $m_{i,2}$, add a directed edge to $G^*$ pointing from $V_{i,2}$ to $V_{i,1}$.

\subsection{Drawing $\boldsymbol{G}$ And its subgraph $\boldsymbol{G^*}$.}\hfill
\label{Drawing G and G^*}

The graphs $G^*$ and $G$ have relatively simple structures which become apparent when $G^*$ and $G$ are drawn. Since $G^*$ is planar, we begin by describing a planar embedding of $G^*$, and then indicate how to add edges and perhaps vertices to the embedding of $G^*$ to obtain $G$.

The vertices of $G^*$ and $G$ can be embedded in the plane so that they lie in two disjoint vertical columns, with the vertices of $G^*$ representing curves in $\mathcal{P}$ lying in the first column, and all other vertices of $G^*$ and $G$ lying in the second column.
To wit, suppose $\mathrm{X}$, $\mathrm{Y}$ and $\mathrm{Y'}$ are three vertices of $G^*$ such that $\mathrm{X}$ and $\mathrm{Y}$ represent successive members of $\mathcal{P}$, the curve $\mathrm{Y'}$ represents is not in $\mathcal{P}$, the curve $R_X$ in $\partial H$ that $\mathrm{X}$ represents is not primitive or a proper-power in $H$, and the curves $R_Y$ and $R_{Y'}$ in $\partial H$ that $\mathrm{Y}$ and $\mathrm{Y'}$ represent are the two distinguished meridians of $H[R_X]$. Then embed $\mathrm{X}$ and $\mathrm{Y}$ in the first column with $\mathrm{Y}$ directly below $\mathrm{X}$, and embed $\mathrm{Y'}$ in the second column at a height between the heights of $\mathrm{X}$ and $\mathrm{Y}$. Then add appropriate directed edges of $G^*$, after which a generic local embedded subgraph of $G^*$ should have the form of Figure \ref{WaveDefFig23}a.

With $G^*$ embedded, the edges and vertices of $G$ which do not belong to $G^*$ can be added to the embedding of $G^*$. These edges and vertices of $G$ are the edges and vertices of $G$ which can only be traversed or reached via a directed path from the initial vertex of $G^*$ that passes through a ``primed'' vertex of $G^*$ which represents a curve in $\partial H$ that is not primitive or a proper-power in $H$.

For example, consider vertex $\mathrm{B'}$ of Figure \ref{WaveDefFig23}, and let $R$ be the curve in $\partial H$ which $\mathrm{B'}$ represents. Since there is a directed edge of $G^*$ leaving vertex $\mathrm{B'}$, $R$ is not primitive or a proper-power in $H$, and so surgery on $R$ along the distinguished wave based at $R$ in $\partial H$ yields two distinguished meridian representatives, say $m_1$ and $m_2$, of $H[R]$ in $\partial H$. Furthermore, since the directed edge of $G^*$ leaving vertex $\mathrm{B'}$ in $G^*$ ends at vertex $\mathrm{B}$ of $G^*$, vertex $\mathrm{B}$ of $G^*$ represents one of $m_1$, $m_2$. So, without loss, let vertex $\mathrm{B}$ of $G^*$ represent $m_1$. Then one of the following occurs:
\begin{enumerate}
\item Both $m_1$ and $m_2$ are primitives or proper-powers in $H$.\\
In this case Figure \ref{WaveDefFig23ee}b illustrates the situation. In Figure \ref{WaveDefFig23ee}, vertices of $G$ representing curves in $\partial H$ that are primitive or proper-powers in $H$ are shown as black squares, vertex $\mathrm{C'}$ of Figure \ref{WaveDefFig23ee}b represents the curve $m_2$, and there are no directed edges of $G$ leaving either vertex $\mathrm{B}$ or vertex $\mathrm{C'}$ of $G$. Hence no vertices or edges of $G$ are embedded in either column below $\mathrm{B}$ or $\mathrm{C'}$.

\item The distinguished wave based at $m_1$ in $\partial H$ is disjoint from $m_2$.\\
In this case, vertex $\mathrm{B}$ represents a curve, say $S$ in $\partial H$, which is not primitive or a proper-power in $H$, so surgery on $S$ along the distinguished wave based at $S$ in $\partial H$ yields two distinguished meridian representatives, say $m_{S1}$ and $m_{S2}$ of $H[S]$ in $\partial H$, one of which must be isotopic to $m_2$ in $\partial H$. Thus $m_2$ is represented either by vertex $\mathrm{C}$ or vertex $\mathrm{C'}$ of $G^*$, and there is a directed edge of $G$ connecting vertex $\mathrm{B'}$ to either $\mathrm{C}$ or vertex $\mathrm{C'}$ of $G^*$. Figures \ref{WaveDefFig23aa}a and \ref{WaveDefFig23aa}b depict these two possibilities.

\item The distinguished wave based at $m_2$ in $\partial H$ is disjoint from $m_1$.\\
In this case, there may exist several vertices of $G$ which are only accessible via directed paths in $G$ that start at the initial vertex of $G$ and pass through vertex $\mathrm{B'}$ of $G^*$. The simplest case occurs when there is only one such vertex, say $\mathrm{E'}$. In which case, $\mathrm{E'}$ represents the curve $m_2$ in $\partial H$. 

Then there must be two directed edges of $G$ leaving $\mathrm{E'}$, one of which must connect $\mathrm{E'}$ to vertex $\mathrm{B}$ of $G$, while the second edge of $G$ leaving $\mathrm{E'}$ must connect $\mathrm{E'}$ to either vertex $\mathrm{C}$ or vertex $\mathrm{C'}$ of $G^*$, as indicated by Figures \ref{WaveDefFig23aa}c and \ref{WaveDefFig23aa}d.
(Note the curve in $\partial H$ which vertex $\mathrm{B}$ represents is not primitive or a proper-power in $H$ in this case, since otherwise there exists more than one vertex of $G$ which is only accessible via directed paths in $G$ that start at the initial vertex of $G$ and pass through vertex $\mathrm{B'}$ of $G^*$. Hence there are directed edges of $G^*$ connecting vertex $\mathrm{B}$ of $G^*$ to vertices $\mathrm{C}$ and $\mathrm{C'}$ of $G^*$.)

In the general case, there is more than one vertex of $G$ which is only accessible via directed paths in $G$ that start at the initial vertex of $G$ and pass through vertex $\mathrm{B'}$ of $G^*$. 

In this case, there exists a sequence of pairs of disjoint curves $(m_{i,1},m_{i,2})$, $0 \leq i \leq n$, in $\partial H$ with each $m_{i,1}$ isotopic to $m_1$ in $\partial H$, for $0 \leq i \leq n$, and $m_{0,2}$ = $m_2$. Furthermore, no $m_{i,2}$ is primitive or a proper-power in $H$ for $0 \leq i < n$, and there is a distinguished wave $\omega_i$ based at each $m_{i,2}$ in $\partial H$, for $0 \leq i < n$, with $\omega_i$ disjoint from $m_{i,1}$, such that $m_{i+1,1}$ and $m_{i+1,2}$ are the two distinguished meridian representatives of $H[m_{i,2}]$ obtained by surgery on $m_{i,2}$ along $\omega_i$.

We mention that each curve $m_{i+1,2}$, $0 \leq i < n$, can be obtained by surgery on $m_{i+1,2}$ along $\omega_i$ in a Heegaard diagram of $m_{i,2}$ that has minimal complexity. It follows $i$ reaches a value $n$ such that either $m_{n,2}$ is primitive or a proper-power in $H$, or $m_{n,2}$ is not primitive or a proper-power in $H$, but the distinguished wave $\omega_n$ based at $m_{n,2}$ in $\partial H$ is not disjoint from $m_{n,1}$.

Suppose $\mathrm{V_i'}$ is the vertex of $G$ representing $m_{i,2}$ for $0 \leq i \leq n$. Then, for $0 \leq i < n$, one of the two directed edges of $G$ leaving $\mathrm{V_i'}$ connects $\mathrm{V_i'}$ to vertex $\mathrm{B}$ of $G$, while the other directed edge leaving $\mathrm{V_i'}$ connects $\mathrm{V_i'}$ to $\mathrm{V_{i+1}'}$.

The location of $\mathrm{V_n'}$ in $G$ depends upon whether the curve $m_1$, which vertex $\mathrm{B}$ represents in $G^*$, is primitive or a proper-power in $H$, and upon whether the curve $m_{n,2}$, which vertex $\mathrm{V_n'}$ represents in $G$, is primitive or a proper-power in $H$. In particular, there are the following cases:
\begin{enumerate}
\item The curve $m_1$ is primitive in $H$.\\
In this case, the subgraph of $G$ only accessible via vertex $\mathrm{B'}$ of $G^*$, must appear as in Figure \ref{WaveDefFig23ee}c with vertex $\mathrm{V_{n-1,2}}$ of $G$ appearing as vertex $\mathrm{E'}$ in Figure \ref{WaveDefFig23ee}c, and vertex $\mathrm{V_{n,2}}$ of $G$ appearing as vertex $\mathrm{C'}$ in Figure \ref{WaveDefFig23ee}c. 
\item The curve $m_1$ is not primitive or a proper-power in $H$.\\
In this case, the subgraph of $G$ only accessible via vertex $\mathrm{B'}$ of $G^*$, must appear as in Figure \ref{WaveDefFig23aa}c or Figure \ref{WaveDefFig23aa}d with vertex $\mathrm{V_{n-1,2}}$ of $G$ appearing as vertex $\mathrm{E'}$ in Figure \ref{WaveDefFig23aa}c or Figure \ref{WaveDefFig23aa}d, and vertex $\mathrm{V_{n,2}}$ of $G$ appearing as vertex $\mathrm{C'}$ in Figure \ref{WaveDefFig23aa}c or as vertex $\mathrm{C}$ in Figure~\ref{WaveDefFig23aa}d, depending upon whether the curve $m_{n,2}$ is isotopic in $\partial H$ to the curve  which vertex $\mathrm{C'}$ of $G^*$ represents or to the curve which vertex $\mathrm{C}$ of $G^*$ represents.
\end{enumerate}
\end{enumerate}

\begin{figure}[ht]
\includegraphics[width = 0.35\textwidth]{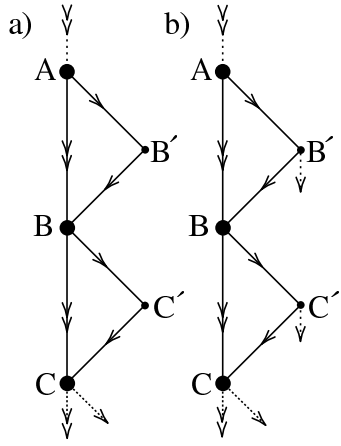}
\caption{\textbf{Representative local generic subgraphs of $\boldsymbol{G^*}$.} (Figure \ref{WaveDefFig23}b is meant to indicate that there are additional edges and perhaps additional vertices of $G$ accessible from the initial vertex of $G^*$ only via directed paths in $G^*$ passing through a ``primed'' vertex of $G^*$, e.g. vertex $\mathrm{B'}$ or vertex $\mathrm{C'}$, of Figure \ref{WaveDefFig23}a.)} 
\label{WaveDefFig23}
\end{figure}

\begin{figure}[ht]
\includegraphics[width = 0.70\textwidth]{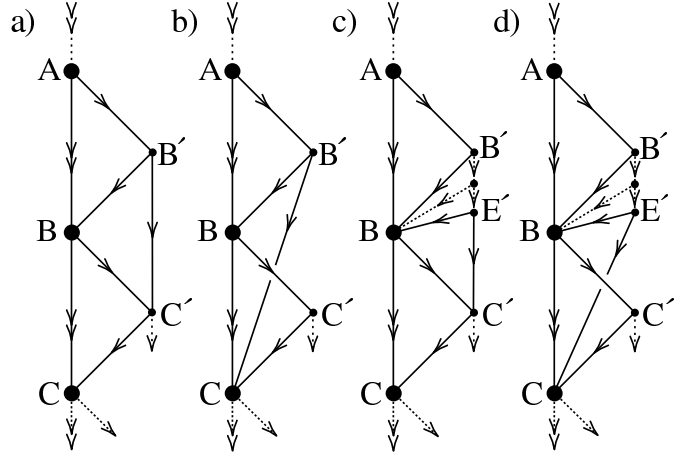}
\caption{\textbf{Four types of subgraphs of $\boldsymbol{G}$.} The four types of subgraphs of $G$ which can arise from the process of adding edges to Figure \ref{WaveDefFig23}b that are only accessible from vertex $\mathrm{B'}$ in Figure \ref{WaveDefFig23}b when the curve which vertex $\mathrm{B}$ represents in $\partial H$ is not primitive or a proper-power in $H$. (There may be several descendants of $\mathrm{B'}$ which appear as intervening vertices on the line from vertex $\mathrm{B'}$ to vertex $\mathrm{E'}$, as suggested by the dotted lines in Figure \ref{WaveDefFig23aa}c and Figure \ref{WaveDefFig23aa}d.)}
\label{WaveDefFig23aa}
\end{figure}

%%%%%%%%%%%%

\subsection{Minimal length directed paths in $\boldsymbol{G}$.}\hfill
\label{Computing min length paths in G}

\begin{defn}\textbf{Minimal directed path length in $\boldsymbol{G}$.}
\label{Minimal path length defn}

If $\mathrm{V}$ is a vertex of $G$, and $P$ is a directed path in $G$ from the initial vertex of $G$ to $\mathrm{V}$, let $L(P)$ be the number of edges of $G$ which occur in $P$. Then let $L(\mathrm{V})$ be the minimum value of $L(P)$ over the set of all directed paths in $G$ from the initial vertex of $G$ to $\mathrm{V}$.
\end{defn}

\begin{lem}\emph{\textbf{Minimal path length equalities in $\boldsymbol{G}$.}} \hfill
\label{Minimal path length equalities in G}

Suppose $\mathrm{A}$, $\mathrm{B}$ and $\mathrm{C}$ are three successive vertices of $G^*$ representing three successive curves lying in the particular unknotting path $\mathcal{P}$ described in Subsection \emph{\ref{Depth computation}}. In addition, suppose $\mathrm{B'}$ and $\mathrm{C'}$ are vertices of $G^*$ such that the two directed edges of $G^*$ leaving vertex $\mathrm{A}$ point to vertices $\mathrm{B}$ and $\mathrm{B'}$, while the two directed edges of $G^*$ leaving vertex $\mathrm{B}$ point to vertices $\mathrm{C}$ and $\mathrm{C'}$. Then $L(\mathrm{B}) = L(\mathrm{B'})$ implies $L(\mathrm{C}) = L(\mathrm{C'})$.
\end{lem}

\begin{proof}
First, consider Figure \ref{WaveDefFig23ff}a. Then $L(\mathrm{C}) = \text{min}\{L(\mathrm{B}) + 1, L(\mathrm{C'}) + 1\}$, and $L(\mathrm{C'}) = \text{min}\{L(\mathrm{B}) + 1, L(\mathrm{B'}) + 1\}$.  By hypothesis, $L(\mathrm{B}) = L(\mathrm{B'})$. So we have $L(\mathrm{C'}) = L(\mathrm{B}) + 1$, $L(\mathrm{C}) = L(\mathrm{B}) + 1$, and $L(\mathrm{C}) = L(\mathrm{C'})$. Next, observe that $L(\mathrm{C}) = L(\mathrm{B}) + 1$, and $L(\mathrm{C}) = L(\mathrm{C'})$ still hold when $\mathrm{B'}$ and $\mathrm{C'}$ are connected by a path in $G$ containing vertices of $G$. 

Second, consider Figure \ref{WaveDefFig23ff}b. Then $L(\mathrm{C}) = \text{min}\{L(\mathrm{B}) + 1, L(\mathrm{B'}) + 1,L(\mathrm{C'}) + 1\}$, and $L(\mathrm{C'}) = L(\mathrm{B}) + 1$. By hypothesis, $L(\mathrm{B}) = L(\mathrm{B'})$. So $L(\mathrm{C'}) = L(\mathrm{B}) + 1$, $L(\mathrm{C}) = L(\mathrm{B}) + 1$, and $L(\mathrm{C}) = L(\mathrm{C'})$. Next, observe that $L(\mathrm{C}) = L(\mathrm{B}) + 1$, and $L(\mathrm{C}) = L(\mathrm{C'})$ still hold when $\mathrm{B'}$ and $\mathrm{C}$ are connected by a path in $G$ containing vertices of $G$.

Finally, note the conclusions above still hold when:
\begin{enumerate}
\item There are directed edges of $G$ ending at $\mathrm{C}$ which originate at vertices of $G$, distinct from $\mathrm{C'}$, that are only accessible from the initial vertex of $G$ via directed paths in $G$ which pass through $\mathrm{C'}$.
\item There is no edge of $G$ connecting $\mathrm{C'}$ and $\mathrm{C}$.
\end{enumerate}
\end{proof}

\begin{figure}[ht]
\includegraphics[width = 0.45\textwidth]{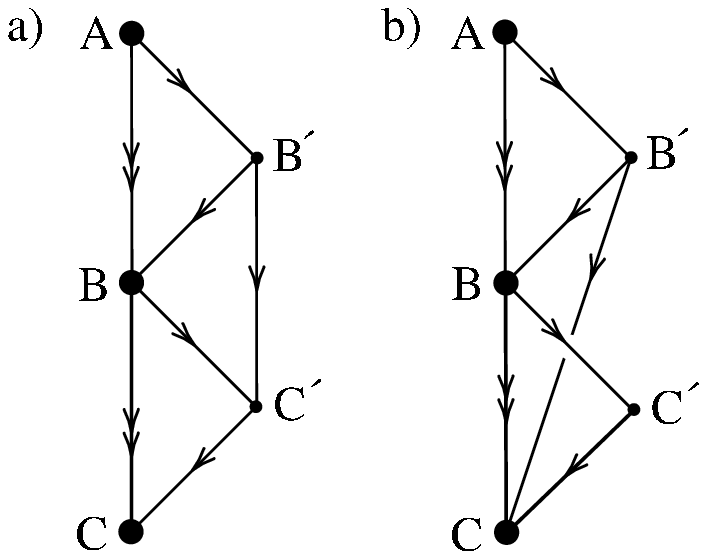}
\caption{\textbf{Two graphs used in showing that $\mathbf{L(B) \boldsymbol{=} L(B')}$ implies $\mathbf{L(C) \boldsymbol{=} L(C')}$.}}
\label{WaveDefFig23ff}
\end{figure}

%%%%%%%%%%%%%%

\begin{figure}[ht]
\includegraphics[width = 0.55\textwidth]{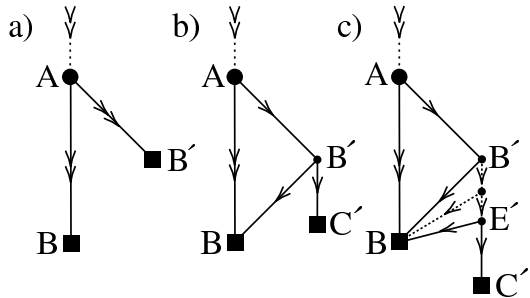}
\caption{\textbf{Local directed paths in $\boldsymbol{G}$ which terminate with vertices---shown as filled rectangles---representing curves in $\boldsymbol{\partial H}$ that are primitives or proper-powers in $\boldsymbol{H}$.}}
\label{WaveDefFig23ee}
\end{figure}

\begin{proof}[Proof of Theorem \emph{\ref{Depth computation works}}] \hfill

It follows readily from the forms of $G^*$ and $G$ as detailed above, that there is only one `unprimed' vertex of $G$ which represents a curve in $\partial H$ that is primitive or a proper-power in $H$, and only one `primed' vertex of $G$ which represents a curve in $\partial H$ that is primitive or a proper-power in $H$, and that these two vertices must arise in $G$ in one of the three configurations illustrated in Figure \ref{WaveDefFig23ee}. Then Lemma \ref{Minimal path length equalities in G}, applied inductively, shows that $L(\mathrm{B}) = L(\mathrm{B'})$ in each subfigure of Figure \ref{WaveDefFig23ee}. It follow that, in all cases, Procedure \ref{Depth computation} produces a minimal length unknotting path for a given curve.
\end{proof}

%%%%%%%%%%%%%%

\section{Other waves based at a fixed curve.}
\label{Other waves based at a fixed curve}

The results of the previous sections show that if $R$ is a nonseparating simple closed curve on the boundary of a genus two handlebody $H$ such that $H[R]$ has incompressible boundary, then there exists a distinguished wave $\omega$, based at $R$ in $\partial H$, such that $\omega$ appears in each connected Heegaard diagram $\mathcal{D}$, without cut-vertices, of $R$ on boundary $H$. However, a nonseparating simple closed curve $R$ on the boundary of a genus two handlebody may also have additional waves which appear only in some Heegaard diagrams of $R$ on $\partial H$. 
This section provides some examples in which such additional ephemeral waves appear, and also provides some results which restrict the appearance of such ephemeral waves.

\begin{defn}\textbf{Positive and Nonpositive genus two curves.}\hfill

A nonseparating simple closed curve $R$ on the boundary of a genus two handlebody $H$ such that $H[R]$ has incompressible boundary is \emph{positive} if there exists a complete set of meridian disks $\Delta$ for $H$ such that the Heegaard diagram $\mathcal{D}$ of $R$ with respect to $\Delta$ is connected, has no cut-vertices, and is positive. Otherwise $R$ is \emph{nonpositive}.
\end{defn}

\begin{prop}\emph{\textbf{Nonpositive curves have only distinguished waves.}}\hfill

Suppose $R$ is a nonseparating, nonpositive, simple closed curve on the boundary of a genus two handlebody $H$, $H[R]$ has incompressible boundary, $\mathcal{D}$ is a Heegaard diagram of $R$ on $\partial H$---which may have cut-vertices---and $\omega$ is a wave based at $R$ in $\mathcal{D}$. Then $\omega$ is a distinguished wave based at $R$.
\end{prop}

\begin{proof}
Recall that Lemmas \ref{wave reduction} and \ref{Sequences of connected Heegaard diagrams without cut-vertices} together show that any complete set of meridian disks of a genus two handlebody $H$ can be transformed into any other complete set of meridian disks of $H$ via a finite sequence $\{\Delta_i\}$, $0 \leq i \leq n$, of complete sets of meridian disks for $H$ in which at each step the current complete set of meridian disks, say $\Delta_i$, is first augmented by the addition of a meridian disk of $H$ whose boundary  is a bandsum in $\partial H$ of the boundaries of the two disks in $\Delta_i$, and then one of the original disks in $\Delta_i$ is discarded, leaving a new complete set of meridian disks $\Delta_{i+1}$ for $H$.

Now, suppose $\Delta_i$, $0 \leq i < n$, is a complete set of meridian disks for $H$ in such a sequence of complete sets of meridian disks for $H$, and consider Figure~\ref{PPFig5jaa} which shows a twice-punctured torus $F$ obtained by cutting $\partial H$ open along $R$. Since $H[R]$ has incompressible boundary, we may assume the set $F \cap \partial \Delta_i$ consists of properly embedded essential arcs in $F$. And then, since $R$ is nonpositive, $F \cap \partial \Delta_i$ must contain a nonempty subset of properly embedded arcs $b^+$, each with both ends on $R^+$, and a nonempty subset of properly embedded arcs $b^-$, each with both ends on $R^-$, such that, letting $|b^+|$ and $|b^-|$ denote the number of arcs in $b^+$ and $b^-$ respectively, $|b^+| > 0$, $|b^-| > 0$, and $|b^+| = |b^-|$.

Next, let $\Delta'$ = $\Delta_{i+1}$. Then, since $R$ is nonpositive, $F \cap \partial \Delta'$ must likewise contain a nonempty subset of properly embedded arcs $b'^+$, each with both ends on $R^+$, and a nonempty subset of properly embedded arcs $b'^-$, each with both ends on $R^-$, such that, letting $|b'^+|$ and $|b'^-|$ denote the number of arcs in $b'^+$ and $b'^-$ respectively, $|b'^+| > 0$, $|b'^-| > 0$, and $|b'^+| = |b'^-|$.

Finally, let $D'$ be the meridian disk of $\Delta'$ that is not a member of $\Delta$. Then the set of arcs $F \cap D'$ is disjoint from the arcs in $F \cap \partial \Delta$, and Lemma \ref{nonpositive diagrams have unique waves} implies each arc of $F \cap D'$ with both endpoints on $R^+$, respectively $R^-$, is properly isotopic in $F$ to an arc of $b^+$, respectively $b^-$. The conclusion follows.
\end{proof}

\begin{figure}[ht]
\includegraphics[width = 0.50\textwidth]{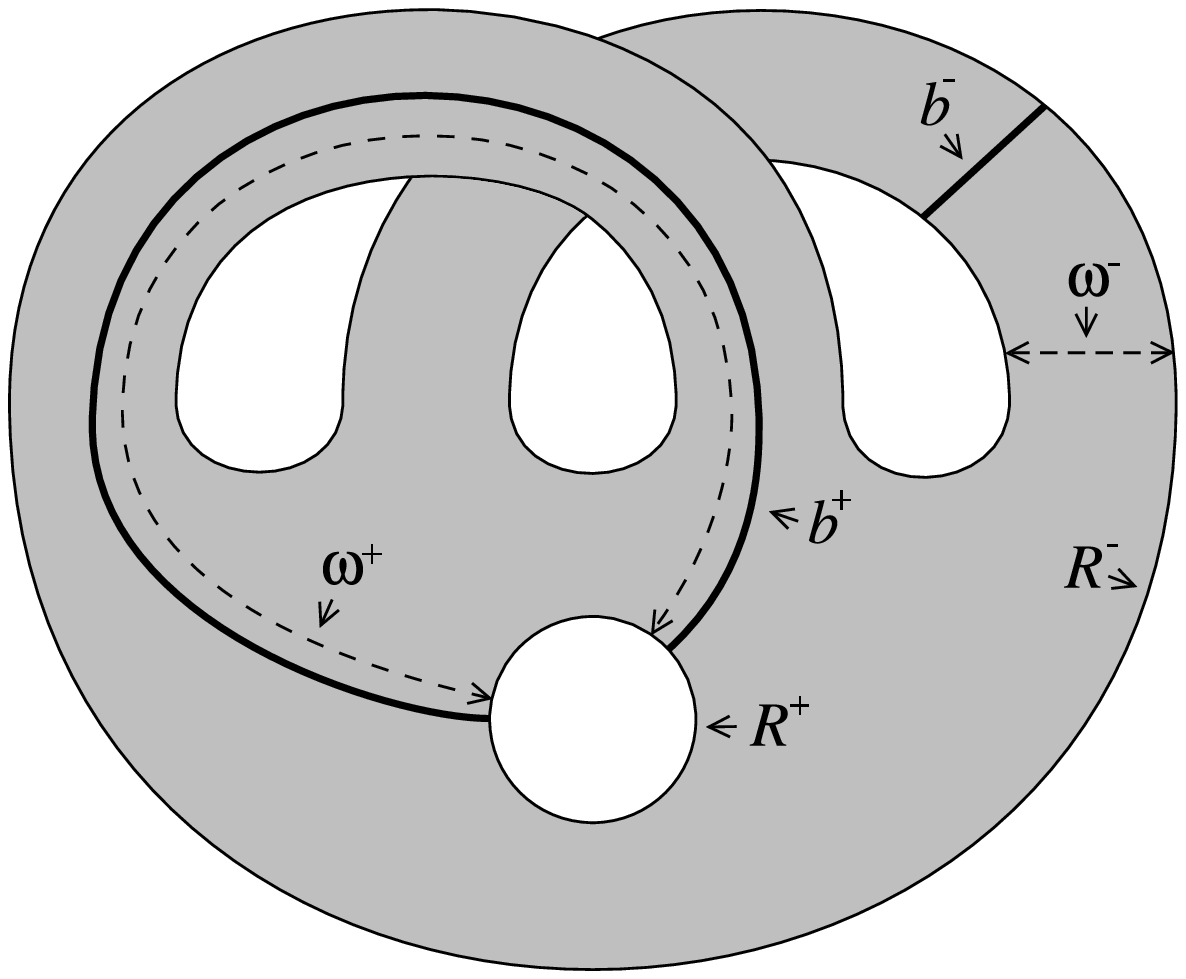}
\caption{\textbf{Nonpositive, nonseparating, simple closed curves have unique waves.} Figure \ref{PPFig5jaa} displays a twice-punctured torus $F$ obtained by cutting the boundary of a genus two handlebody $H$ open along a nonseparating simple closed curve $R$ in $\partial H$ along with bands of arcs $b^+$ and $b^-$ in $F \cap \partial \Delta$, and representative waves $\omega^+$ and $\omega^-$ based at $R^+$ and $R^-$ respectively.}
\label{PPFig5jaa}
\end{figure}

Here is a simple test which can be used to verify many nonseparating simple closed curves on the boundary of a genus two handlebody are nonpositive.

\begin{lem}\emph{\textbf{A nonpositivity test.}}\hfill
\label{A nonpositivity test}

If an oriented nonseparating simple closed curve $R$ on the boundary of a genus two handlebody $H$ has a Heegaard diagram $\mathcal{D}$ whose graph has the form of Figure~\emph{\ref{DPCFig8as}a} with edges of opposite orientation connecting vertices $A^+$ and $A^-$ of Figure \emph{\ref{DPCFig8as}a}, and edges of opposite orientation connecting vertices $B^+$ and $B^-$ of Figure \emph{\ref{DPCFig8as}a}, then $R$ is nonpositive.
\end{lem}

\begin{proof}
It is enough to show that if $D$ is a meridian disk of $H$, and $\partial D$ is oriented, then $R$ and $\partial D$ have essential intersections with opposite signs. 

So, let $\{D_A,D_B\}$ be the complete set of meridian disks for $H$ in $\mathcal{D}$, and suppose first that $\partial D$ is disjoint from $\partial D_A$ and $\partial D_B$ in $\mathcal{D}$. Then either $\partial D$ is isotopic in $\partial H$ to $\partial D_A$ or $\partial D_B$, or $\partial D$ is a bandsum of $\partial D_A$ with $\partial D_B$ in $\partial H$. And clearly, in each of these cases, $\partial D$ and $R$ have essential intersections in $\partial H$ with opposite signs.

So, suppose $\partial D$ has essential intersections with $\partial D_A \cup \, \partial D_B$. Then there are arcs of $\partial D$ which are loops based at vertices $A^+$ and $A^-$ or at vertices $B^+$ and $B^-$ in the graph of $\mathcal{D}$ in Figure \ref{DPCFig8as}a. But, each such loop has essential intersections of opposite signs with $R$. It follows $R$ is nonpositive.
\end{proof}

\begin{rem}
\label{B19}
\cite{B19} describes an algorithm which determines whether a given curve $R$ on the boundary of a genus two handlebody $H$ is positive or nonpositive, and, when $H[R]$ has incompressible boundary and $R$ is positive, produces a complete set of meridian disks $\{D_A,D_B\}$ for $H$ such that the Heegaard diagram of $R$ with respect to $\{D_A,D_B\}$ is connected, positive, and has no cut-vertex. It is also shown there that, if $H[R]$ is the exterior of a tunnel-number-one knot in a lens space or a connected sums of lens spaces and $H[R]$ has a non-trivial lens space or reducible surgery, then $R$ is positive. This leads, for example, to an easy recovery of the known result that non-torus 2-bridge knots in $S^3$ have no cyclic or reducing exceptional surgeries.
\end{rem}

Although, as following examples show, positive curves may have numerous distinct wave-determined slopes, there are some restrictions, as the next proposition shows.

\begin{prop}\textbf{\emph{Slopes determined by vertical waves in connected, positive diagrams, without cut-vertices, are distance one from the slope determined by the distinguished horizontal wave.}} \hfill
\label{Slopes determined by vertical waves and distance one}

Suppose $R$ is a nonseparating simple closed curve on the boundary of a genus two handlebody $H$, $\{D_A,D_B\}$ and $\{D_A',D_B'\}$ are two complete sets of meridional disks for $H$, while $\mathcal{D}$ and $\mathcal{D}'$ are Heegaard diagrams of $R$ with respect to $\{D_A,D_B\}$ and $\{D_A',D_B'\}$ such that $\mathcal{D}$ and $\mathcal{D}'$ are each connected, positive, and have no cut-vertices. In addition, let $\omega_h$ and $\omega_v$ be the horizontal and vertical waves of $\mathcal{D}$, and let $\omega_h'$ and $\omega_v'$ be the horizontal and vertical waves of $\mathcal{D}$. Then $\omega_h$ and $\omega_h'$ are each distinguished waves based at $R$, and so determine the same slope, say $m$, on $\partial H[R]$. So the slopes on $\partial H[R]$ determined by $\omega_v$ and $\omega_v'$ are each distance one from $m$.

In particular, if $H[R]$ is the exterior of a knot in $S^3$, then $m$ represents the meridian of $H[R]$, and $\omega_v$ and $\omega_v'$ determine integral slopes on $\partial H[R]$.
\end{prop}

\begin{proof}
This follows immediately from the fact that $\omega_h'$ is isotopic in $\partial H$, keeping its endpoints on $R$, to $\omega_h$, up to the action of the hyperelliptic involution on $H$, together with the fact that $\omega_v \cap \omega_h$ and $\omega_v' \cap \omega_h'$ are each a single point of transverse intersection.
\end{proof}

%%%%%%%%%%%%

\subsection{Positive curves with more than 2 distinct wave-determined slopes.}\hfill
\label{Positive curves yielding more than two distinct wave-determined slopes}

Unlike nonpositive nonseparating curves, which yield only one wave-determined slope, positive nonseparating curves always yield at least two distinct wave-determi-ned slopes, may yield at least three distinct wave-determined slopes, as example \ref{At least three distinct wave-determined slopes} shows, or may yield an infinite number of distinct wave-determined slopes, as example \ref{An infinite number of wave-determined slopes} shows.

\begin{ex}\textbf{A positive curve whose connected Heegaard diagrams, without cut-vertices, yield at least three distinct wave-determined slopes.}\hfill
\label{At least three distinct wave-determined slopes}

Figure \ref{WaveDefFig5}a is an R-R diagram of a positive curve R with respect to a complete set $\Delta$ of meridian disks of the underlying handlebody $H$ such that R = $A^7B^2A^4B^2$ in $\pi_1(H)$. In addition, Figure \ref{WaveDefFig5}b is an R-R diagram of the same curve with respect to a complete set $\Delta'$ of meridian disks of the underlying handlebody $H$ such that R = $A^7BA^3BA^7B^2$ in $\pi_1(H)$. 

In Figure \ref{WaveDefFig5}a filling $H[\mathrm{R}]$ along the $\omega_v$-determined slope yields $L(11,3) \# L(2,1)$, while filling $H[\mathrm{R}]$ along the $\omega_v$-determined slope in Figure \ref{WaveDefFig5}b yields $L(23,7)$. (In Figure \ref{WaveDefFig5}a, $A^3B$ represents the distinguished invariant $\omega_h$-determined slope, while in Figure \ref{WaveDefFig5}b, $A^4B$ represents the distinguished invariant $\omega_h$-determined slope. And, in each case, filling $H[R]$ along the $\omega_h$-determined slope yields $S^3$.)

\begin{figure}[ht]
\includegraphics[width = 1.0\textwidth]{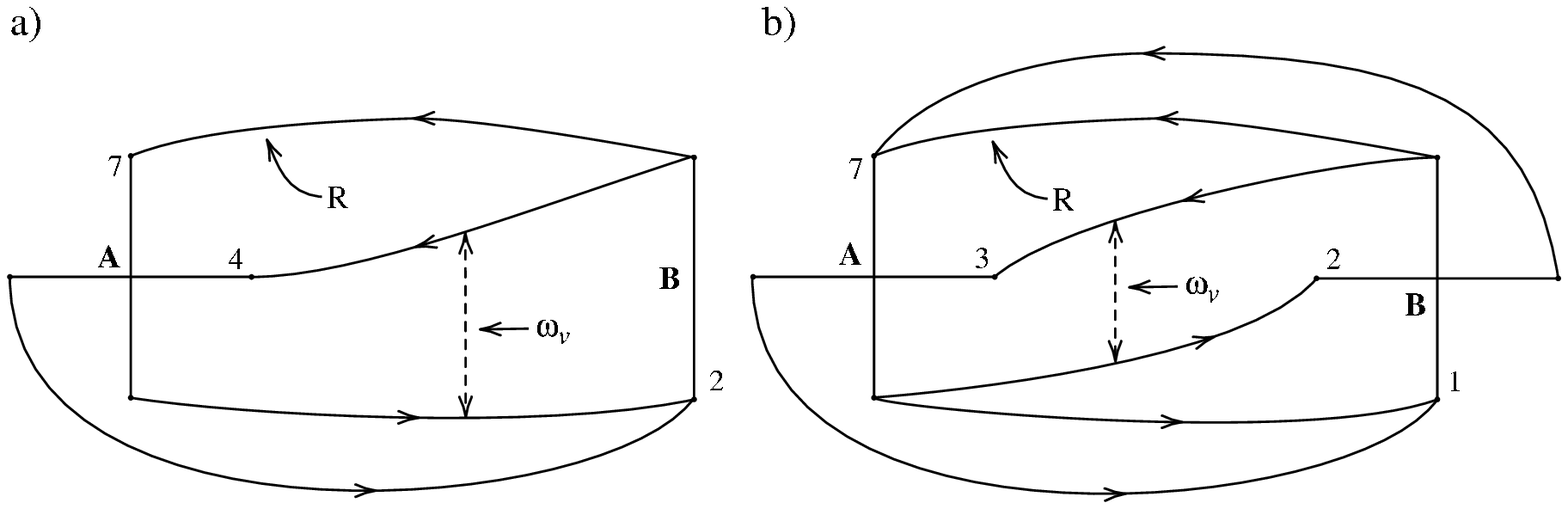}
\caption{}
\label{WaveDefFig5}
\end{figure}
\end{ex}

%%%%%%%%%%%%

\begin{ex}\textbf{Trefoil knot exteriors have an infinite number of distinct wave-determined slopes.}\hfill
\label{An infinite number of wave-determined slopes}

Suppose $S > 3$ and $\gcd(S,3) = 1$. Then Figure \ref{WaveDefFig6} is an R-R diagram of a simple closed curve $R$ on $\partial H$, such that R = $AB^SAB^{S+3}$ in $\pi_1(H)$, $R$ is an automorph of the trefoil relator $A^2B^3$, and $H[R]$ embeds in $S^3$ with meridional simple closed curve $m$, disjoint from $R$, such that $m = AB^{3(j+1)}$ in $\pi_1(H)$ if $S = 3j + 1$, and $m = AB^{3j}$ in $\pi_1(H)$ if $S = 3j -1$. 

Figure \ref{WaveDefFig6} shows surgery on $R$ along $\omega_1$, $\omega_2$, and $\omega_3$ yields $m_1 = A$, $m_2 = B^3$, and $m_3 = B^3A$ respectively. Then filling $H[R]$ along $m_1$, $m_2$, and $m_3$ yields respectively $L(2S+3,S)$, $L(2,1) \# L(3,1)$, and $L(2S-3,S-3)$. (Note neither $m_1, m_2$ nor $m_3$ is a meridian of $H[R]$ when $S > 3$.)

\begin{figure}[ht]
\includegraphics[width = 0.55\textwidth]{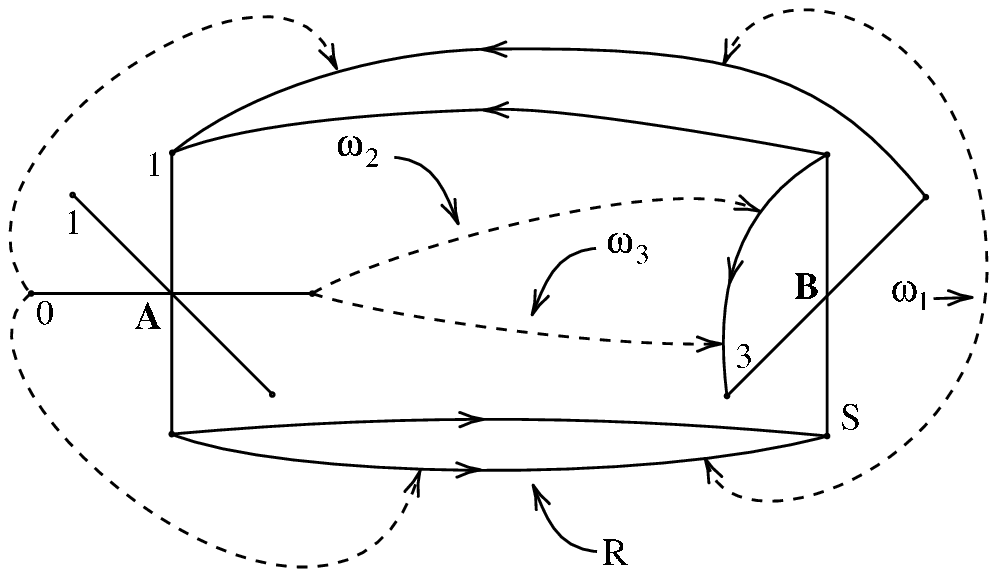}
\caption{}
\label{WaveDefFig6}
\end{figure}
\end{ex}

%%%%%%%%%%%

\begin{ex} \textbf{Horizontal and vertical waves in positive Heegaard diagrams without cut-vertices do not necessarily appear as waves in Heegaard diagrams with cut-vertices.}\hfill

Suppose a curve $R$ has a positive connected genus two Heegaard diagram, without cut-vertices, with respect to a complete set of meridional disks $\{D_A,D_B\}$ of an underlying handlebody $H$. Then $R$ has horizontal and vertical waves $\omega_h$ and $\omega_v$. 

However, Figure \ref{DPCFig89an}a and Figure \ref{DPCFig89an}b show that $\omega_h$, respectively $\omega_v$, does not appear as a wave in the Heegaard diagram with cut-vertices obtained when one member of $D_A,D_B$ is replaced by a bandsum of $D_A$ with $D_B$ along the arc connecting $A^-$ with $B^-$ in Figure \ref{DPCFig89an}a, respectively Figure \ref{DPCFig89an}b.

\begin{figure}[ht]
\includegraphics[width = .70\textwidth]{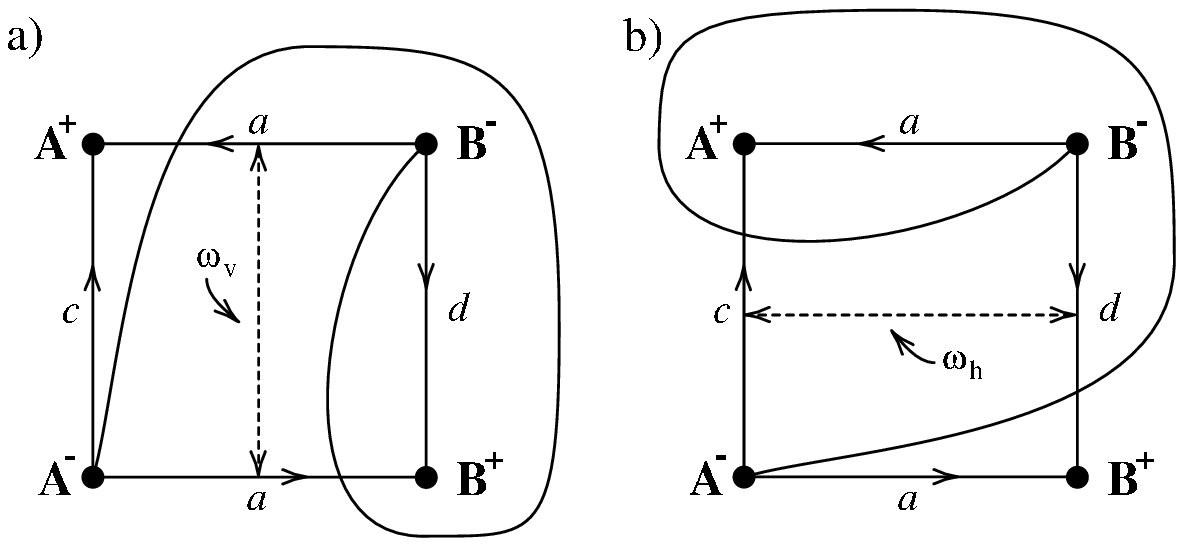}
\caption{}
\label{DPCFig89an}
\end{figure}
\end{ex}

%%%%%%%%%%%

\begin{ex}\textbf{The `unique' waves in different nonpositive genus two diagrams with cut-vertices of a relator of a knot in $\boldsymbol{S^3}$ may be distinct.}\hfill
\label{Nonpositive distinct waves}

Consider Figure \ref{WaveDefFig6}. This figure displays a parameterized R-R diagram of a trefoil knot relator $R$ on the boundary of a genus two handlebody $H$ with respect to various complete sets $\{D_{A_S},D_{B_S}\}$ of meridional disks of $H$, such that, in each case, the graph of the Heegaard diagram of $R$ with respect to $\{D_{A_S},D_{B_S}\}$ has the form of Figure \ref{DPCFig89an}a with $(a,c,d)$ = $(2,0,2S+1)$. Hence, for given $S$, if $D_{A_S}$ is replaced by a disk $D_{C_S}$ such that $\partial D_{C_S}$ is a bandsum of $\partial D_{A_S}$ with $\partial D_{B_S}$ along a band, as depicted in Figure \ref{DPCFig89an}a, then the Heegaard diagram $\mathcal{D}_S$ of $R$ with respect to $\{D_{C_S},D_{B_S}\}$ has a cut-vertex, has loops, and is nonpositive.

Since $\mathcal{D}_S$ is nonpositive, there is a unique wave $\omega_{v_S}$ based at $R$ in each $\mathcal{D}_S$, as in Figure \ref{DPCFig89an}a. However, $\omega_{v_S}$ is properly isotopic, keeping its endpoints on $R$ to the wave $m_1$ in Figure \ref{WaveDefFig6}. So filling $H[R]$ along the slope determined by $\omega_{v_S}$ yields $L(2S+3,S)$, which depends on $S$. It follows that $\omega_{v_S}$ and $\omega_{v_{S'}}$ are distinct waves based at $R$ when $S,S' > 3$, $\gcd(S,3) = \gcd(S',3) = 1$, and $S \neq S'$.
\end{ex}

%%%%%%%%%%%%%

\subsection{Restricting the set of possible waves based at a positive curve.}\hfill

\begin{ex}\textbf{A family of positive curves whose Heegaard diagrams yield at most two distinct wave-determined slopes.}\hfill
\label{At most two distinct wave-determined slopes}

Consider Figure \ref{WaveDefFig3}, which is an R-R diagram $\mathbb{D}$ of a nonseparating simple closed curve $R$ on the boundary of a genus two handlebody $H$. In addition, let $\mathcal{D}$ be the Heegaard diagram of $R$ on $\partial H$ corresponding to $\mathbb{D}$. Note that if the parameters of Figure \ref{WaveDefFig3} satisfy $P,R,S,U > 1$, $\gcd(P,R) = \gcd(S,U) =~1$, $Q = P+R$, and $T = S+U$, then $\mathcal{D}$ is positive, connected, has no cut-vertices, and $\mathcal{D}$ has two non-rectangular faces $F_1$ and $F_2$, which also appear in $\mathbb{D}$.

\begin{prop}
Suppose $\omega$ is an arc in $\partial H$ meeting $R$ only in its endpoints, and there exists a Heegaard diagram $\mathcal{D}'$ of $R$ on $\partial H$ in which $\omega$ appears as a wave based at $R$. Then $\omega$ can be properly isotoped in $\partial H$ into one of $F_1$, $F_2$.
\end{prop}

\begin{proof}
Suppose $\{D_A,D_B\}$ is the complete set of meridian disks of $H$ in diagrams $\mathbb{D}$ and $\mathcal{D}$, and $\omega$ is an arc in $\partial H$, disjoint from $R$, except at its endpoints, such that $\omega$ appears as a wave based at $R$ in $\mathcal{D}'$, but $\omega$ cannot be properly isotoped in $\partial H$, keeping its endpoints on $R$, into $F_1$ or $F_2$. Then $\omega$ has essential intersections with $D_A \cup D_B$, such that the points of $|\omega \cap (D_A \cup D_B)|$ cut $\omega$ into a set of arcs $\mathcal{A}$, and the arcs in $\mathcal{A}$, which do not contain the endpoints of $\omega$, form the edges of a Heegaard diagram $\mathcal{D}^*$ of $\omega$ with respect to $\{D_A,D_B\}$. 

If there exists a Heegaard diagram $\mathcal{D}'$ of $R$ in which $\omega$ appears as a wave based at $R$, then the complete set of meridian disks of $H$ used to form $\mathcal{D}'$ must be disjoint from $\omega$. Hence $\mathcal{D}^*$ must not be connected, or have a cut-vertex. 

To see that, to the contrary, $\mathcal{D}^*$ is actually connected and does not have a cut-vertex, let $A^+$, $A^-$, $B^+$ and $B^-$ be the vertices of $\mathcal{D}$ and $\mathcal{D}^*$, and observe that because $P,Q,R,S,T,U > 1$, $\mathcal{D}^*$ must contain edges connecting $A^+$ to $A^-$ and edges connecting $B^+$ to $B^-$. This suffices to show that $\mathcal{D}^*$ is connected and has no cut-vertex. It follows $\omega$ does not exist.
\end{proof}

\begin{figure}[ht]
\includegraphics[width = 0.55 \textwidth]{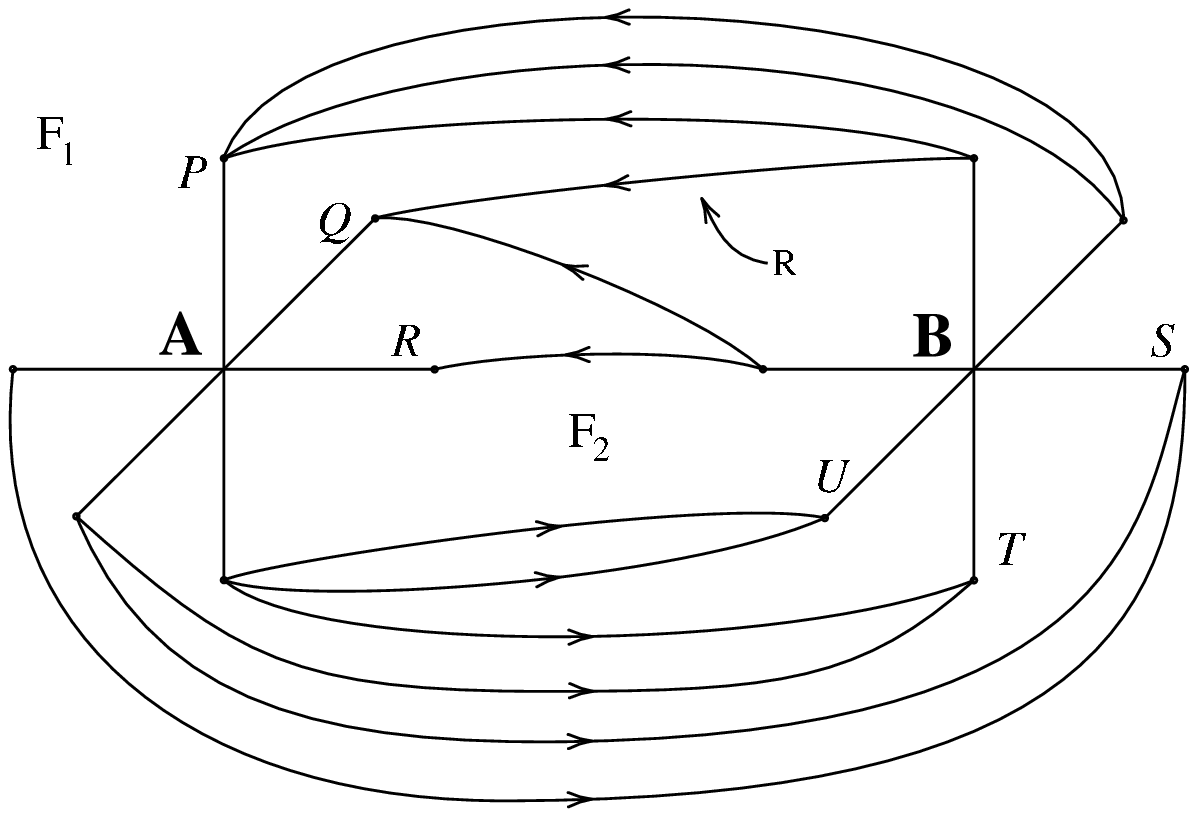}
\caption{}
\label{WaveDefFig3}
\end{figure}
\end{ex}

%%%%%%%%%%%%

\newpage
\section{Distinguished-wave-determined slope dependence on tunnels in tunnel-number-one manifolds which do not embed in $S^3$, $S^1 \times S^2$, or $(S^1 \times S^2)\; \# \; L(p,q)$.}
\label{Distinguished-wave-determined slope dependence on tunnels}

Before describing some examples of tunnel-number one manifolds which exhibit distinct distinguished-wave-deter-mined slopes, we recall the following result, which shows tunnel-number-one manifolds have unique distinguished-wave-determined slopes.

\subsection{Tunnels of high-distance tunnel-number-one manifolds are unique.}
\hfill
\label{tunnels of high-distance tunnel-number-one manifolds are unique}

Johnson proved the following result for knots in $S^3$ in \cite{J06}. The general case follows from results of Scharlemann and Tomova in \cite{ST06}.

\begin{thm}[\cite{J06}, \cite{ST06}] \textbf{\emph{High-distance implies a unique tunnel.}} \hfill
\label{high-distance implies a unique tunnel}

If $R$ and $R'$ are nonseparating simple closed curves on the boundaries of genus two handlebodies $H$ and $H'$ respectively, $H[R]$ and $H'[R']$ are homeomorphic, and $R$ is distance greater than five in the curve-complex of $\partial H$ from each essential simple closed curve in $\partial H$ that bounds a disk in $H$, then the pairs $(H,R)$ and $(H',R')$ are homeomorphic.
\end{thm}

\subsection{One-cusped manifolds in the SnapPy census of the set of all orientable hyperbolic manifolds with triangulations using at most 9 ideal tetrahedra which have tunnels with distinct distinguished-wave-deter-mined slopes.} \hfill
\smallskip

Among the data files accompanying the 3-manifold computer program SnapPy \cite{CDGW} is a census, the OrientableCuspedCensus or `OCC', of the set of all  orientable cusped hyperbolic 3-manifolds which have ideal triangulations with at most 9 ideal tetrahedra.

A search of the 59,107 1-cusped manifolds in the OCC was conducted using the author's \emph{Heegaard} program \ref{OCC data}, which looked for tunnels of these manifolds. This resulted in the discovery of 113,378 tunnels for 49,933 of the manifolds in the OCC, so that at least 49,933, or approximately 84.5\% of the 59,107 1-cusped manifolds in the OCC are tunnel-number-one manifolds. 

Continuing, Heegaard refined this list of 49,933 manifolds by discarding those which embed in $S^3$, $S^1 \times S^2$, or $(S^1 \times S^2)\;\#\;L(p,q)$, or for which Heegaard found only one tunnel. This left a set of 38,287 OCC manifolds, for each of which Heegaard found non-homeomorphic tunnels. Finally, Heegaard found that 7,917 of these remaining 38,287 manifolds, or about $ 20.7\%$, have distinct distinguished-wave-determined slopes. The following tables provide some details.

\begin{table}[ht]
\label{Table 1}
\caption{Table 1 below lists some statistics for the OCC manifolds partitioned according to the number of non-homeomorphic tunnels Heegaard found for each manifold. For example, line 4 of Table 1 shows that Heegaard found 4 non-homeomorphic tunnels for each of 4,632 of the 49,933 tunnel-number-one manifolds in the OCC, of which 151 embed in $S^3$, 15 embed in $S^1 \times S^2$, and none embed in $(S^1 \times S^2)\;\#\;L(p,q)$. So 913/(4,632-151-15), or about 20,4\% of these manifolds have distinct distinguished-wave-determined slopes.
Finally, the bottom line of the table shows that 7,917/48,688, or about 16.3\% of the OCC manifolds, which do not embed in $S^3$, $S^1 \times S^2$, or $(S^1 \times S^2)\;\#\;L(p,q)$, have tunnels with distinct distinguished-wave-determined slopes.}
\renewcommand{\arraystretch}{1.1}
\begin{tabular}{|c|c|c|c|c|c|c|}
\hline
Tunnels & Mflds  & $S^3$ & $S^1 \times S^2$ & $(S^1 \times S^2)\; \#\; L(p,q)$ & $> 1\; \omega$-Slope & $\%$ \\ \hline 
1 	 &10,622	& 151 	& 57 		& 13		& 0		& 0\%           \\ \hline
2 	 &20,243	& 441 	& 68   	& 0		& 3,003	& $15.2 \%$ \\ \hline
3 	 &14,219	& 276 	& 45   	& 0		& 3,979	& $28.6 \%$ \\ \hline
4 	 &4,632	& 151 	& 15   	& 0		& 913	& $20.4 \%$ \\ \hline
5 	 &217	& 28	  	& 0 	    	& 0		& 22		& $11.6 \%$ \\ \hline
Totals &49,933 & 1,047   	& 185	& 13	        & 7,917	& $16.3 \%$ \\ \hline
\end{tabular}
\end{table}

\begin{table}[ht]
\label{Table 2}
\caption{Table 2 lists the number of times a specific unordered pair of distinguished-wave-determined slopes occurs among the 7,917 manifolds with distinct distinguished-wave-determined slopes. For instance, the unordered pair ((1,0), (-1,1)) occurs 1,392 times. (Note the members of each such pair are distance one apart.)}
\renewcommand{\arraystretch}{1.1}
\begin{tabular}{|c|c|c|c|c|c|c|}
\hline
& (1,0) & (-2,1) & (-1,1) & (0,1) & (1,1) & (2,1) \\ \hline 
(1,0) 	 & 		& 97 		& 1,392 	& 151	& 1,122	& 48 \\ \hline
(-2,1) & 97	&  		& 0   		& 0		& 0		& 0 \\ \hline
(-1,1) & 1,392	& 0 		&    		& 2,196	& 0		& 0 \\ \hline
(0,1) 	 & 151	& 0 		& 2,196   	& 		& 2,911	& 0 \\ \hline
(1,1)	 & 1,122	& 0	  	& 0 	    	& 2,911	& 		& 0 \\ \hline
(2,1)  &48		& 0	  	& 0     	& 0		& 0		&  \\ \hline
\end{tabular}
\end{table}

\newpage

\begin{table}[ht]
\label{Table 3}
\caption{Distribution of distinguished-wave-determined slopes among the 39,002 OCC manifolds that do not embed in $S^3$, $S^1 \times S^2$, or $(S^1 \times S^2)\;\#\;L(p,q)$, and whose tunnels exhibit only one distinguished-wave-determined slope.}
\renewcommand{\arraystretch}{1.1}
\begin{tabular}{|c|c|c|}
\hline
Mflds 	& Slope  	& \%			\\ \hline 
26,593 	&(1,0) 	& $68.2\%$	\\ \hline
7,599 	&(0,1) 	& $19.5\%$	\\ \hline
2,411 	&(1,1) 	& $6.2\%$		\\ \hline
2,399 	&(-1,1)	& $6.2\%$		\\ \hline
\end{tabular}
\end{table}

\begin{table}[ht]
\label{Table 4}
\caption{Distribution of the distinguished-wave-determined meridional slopes for tunnels of the OCC manifolds which embed in $S^3$, $S^1 \times S^2$, or $(S^1 \times S^2)\;\#\;L(p,q)$. (This shows that about 9.2\% of the time the shortest geodesic on the cusp of a tunnel-number-one OCC manifold that embeds in $S^3$, $S^1 \times S^2$, or $(S^1 \times S^2)\;\#\;L(p,q)$ is not the meridian.)}
\renewcommand{\arraystretch}{1.1}
\begin{tabular}{|c|c|c|c|} \hline
Slope	& $S^3$	& $S^1 \times S^2$	& $(S^1 \times S^2)\;\#\;L(p,q)$	 \\ \hline 
(1,0)		& 2,382	& 347			& 24			 			 \\ \hline
(0,1) 		& 168	& 37				& 2						 \\ \hline
(1,1)		& 29		& 9				& 0	 					 \\ \hline
(-1,1)	& 26		& 7				& 0						 \\ \hline
(1,0)$\%$	& $91.4 \%$	& $86.8 \%$	& $92.3 \%$				 \\ \hline
\end{tabular}
\end{table}

\begin{rem} \textbf{{Count discrepancies due to isometries in OCC manifolds.}} \hfill

In Heegaard's survey data, each of the five manifolds below appears with the two distinguished-wave-determined slopes shown. For instance, m003 appears with distinguished-wave-determined slopes (0,1) and (-1,1). (SnapPy shows each of these manifolds has an isometry exchanging the two distinguished-wave-determined slopes.)
\begin{enumerate}
\item m003(0,1) $\simeq$ m003(-1,1);
\item s725(1,0) $\simeq$ s725(0,1);
\item s911(0,1) $\simeq$ s911(-1,1);
\item s955(0,1) $\simeq$ s955(-1,1);
\item t07733(1,0) $\simeq$ t07733(0,1).
\end{enumerate}
\end{rem}

\subsection{Examples of OCC manifolds with tunnels exhibiting distinct distin-guished-wave-determined slopes.}
\hfill
\label{Example distinct distinguished-wave-determined slopes}

\begin{ex}\textbf{The first tunnel-number-one manifold in the OCC that has tunnels with distinct distinguished-wave-determined slopes.} \hfill
\smallskip

The manifold m011(0,0) is the first tunnel-number-one manifold in the OCC that has tunnels with different distinguished-wave-determined slopes. It also provides an example for which it is easy to verify by hand that a pair of tunnels yields distinct distinguished-wave-determined slopes. 

Some details. According to SnapPy, $\pi_1(m011(0,0))$ has a presentation $\mathcal{P}$ of the form:
$$\mathcal{P} = \langle \, A,B,C \mid AABCB, ACCAb \, \rangle,$$
which Heegaard readily verifies has a unique realization.
Since each of the two relators of $\mathcal{P}$ is primitive in the underlying genus three handlebody, the realization of $\mathcal{P}$ destabilizes in two obvious ways. This yields the following pair of realizable presentations: 
$$\mathcal{P}_1 = \langle \, A,B \mid R_1 = AAABAbbbAB \, \rangle, \quad
\mathcal{P}_2 = \langle \, A,B \mid R_2 = AABBAABaBaB \, \rangle.$$
Next, from the realization of $\mathcal{P}_1$ on the boundary of a genus two handlebody $H$, one finds the curves $m_{1,1}$ and $m_{1,2}$ obtained by surgery on $R_1$ along the distinguished wave based at $R_1$ are $m_{1,1} = Abb$ and $m_{1,2} = AAAAB$. Similarly, from the realization of $\mathcal{P}_2$ on the boundary of a genus two handlebody $H$, one finds the curves $m_{2,1}$ and $m_{2,2}$ obtained by surgery on $R_2$ along the distinguished wave based at $R_2$ are $m_{2,1} = ABaB$ and $m_{2,2} = AABBB$.

Finally, $H_1(H[R_1,m_1];\mathbb{Z}) = \mathbb{Z}/9$, while $H_1(H[R_2,m_2];\mathbb{Z}) = \mathbb{Z}/4$. So the distin-guished-wave-determined slopes of the two tunnels corresponding to $H[R_1]$ and $H[R_2]$ are different. In particular, Heegaard identifies $H[R_1,m_1]$ as the lens space $m011(0,1) \simeq L(9,2)$, and Heegaard identifies $H[R_2,m_2]$ as the Seifert fibered space $m011(1,1) \simeq SF(0;1;1/2,1/2,1/3)$. 
\end{ex}

\begin{ex}\textbf{An OCC manifold with exceptional fillings but no exceptional distinguished-wave-determined slope.}\hfill

Heegaard's survey of distinguished-wave-determined slopes for OCC manifolds suggests that when a tunnel-number-one manifold in the OCC census has exceptional fillings it is often the case that there is a tunnel for that manifold whose distinguished-wave-determined slope yields one of the exceptional fillings. However, this is not always the case, as the manifold v2293(0,0) in the OCC census has no exceptional distinguished-wave-determined slope.
 
Some details. According to SnapPy, the fundamental group of the manifold v2293(0,0) in the OCC has a presentation $\mathcal{P}$ of the form:
$$ \mathcal{P} = \langle \, A,B,C,D,E,F \mid bAA, dCa, DCE, FCFBeB, FDeF \, \rangle. $$
Heegaard, after simplifying and destabilizing $\mathcal{P}$, found only one realizable 1-relator, 2-generator presentation $\mathcal{P}_2$ = $\langle \, A,B \mid R \, \rangle$ for v2293(0,0).
 
Then, from the relator $R$ of $\mathcal{P}_2$, which has the form:
$$R = AABABAABaBAABaBBaBAABaB,$$
Heegaard found that the two distinguished curves $m_1$ and $m_2$, obtained from $R$ by surgery on $R$ along the distinguished wave based at $R$, are:
$$m_1 = BBaBAABaB, \quad \text{and} \quad m_2 = ABABAABaBAAB.$$
It follows the filled manifold $H[R,m_1]$ has the realizable presentation: 
$$\mathcal{P}_3 = \langle \, A,B \mid R, m_1 \, \rangle, \quad \text{or:}$$
$$\mathcal{P}_3 = \langle \, A,B  \mid AABABAABaBAABaBBaBAABaB, BBaBAABaB \, \rangle,$$
which Heegaard destabilized and simplified to:
$$\mathcal{P}_4 = \langle \, A,B \mid AAABAbbAB, AABBaBaBB \, \rangle.$$
Next, SnapPy provided the following presentation for the filled manifold v2293(1,1):
$$\mathcal{P}_5 = \langle \, A,B,C,D,E,F \mid bAA, dCa, DCE, FCFBeB, FDeF, BFdCAeBe \, \rangle.$$
Heegaard, after simplifying and destabilizing $\mathcal{P}_5$, found the following two minimal-length 2-generator presentations $\mathcal{P}_6$ and $\mathcal{P}_7$ for v2293(1,1), which presumably represent two distinct genus two Heegaard splittings of v2293(1,1).
$$\mathcal{P}_6 = \langle \, A,B \mid AAABAbbAB, AABBaBaBB \, \rangle, \quad \text{and} $$
$$\mathcal{P}_7 = \langle \, A,B \mid  AABAAbabAAbabAAbab, ABBABaBaB \, \rangle.$$
Comparison of $\mathcal{P}_4$ and $\mathcal{P}_6$ shows they are identical. It follows, since $\mathcal{P}_4$ and $\mathcal{P}_6$ have unique realizations, that $H[R,m_1]$ is homeomorphic to v2293(1,1).

Finally, Dunfield's paper \cite{D19a} and associated list \cite{D19b} of the exceptional fillings of the manifolds in the OCC show that v2293(0,0) has exceptional filling slopes (1,0) and (0,1), but the distinguished-wave-determined slope (1,1) is not exceptional.
\end{ex}

\begin{ex}\textbf{An OCC manifold with four exceptional fillings, one cyclic, which has two distinguished-wave-determined exceptional slopes, both non-cyclic.}\hfill

The OCC manifold s768(0,0) provides such an example. From SnapPy, the fundamental group $\mathcal{P}$ of s768(0,0) has the following realizable presentation:
$$\mathcal{P} = \langle\;A,B,C,D,E,F \mid Baa, cFcBA, Dff, EBd, EEDC \; \rangle.$$
Then, using Heegaard to simplify and destabilize $\mathcal{P}$ produced the following four 2-generator, 1-relator presentations. 
\begin{align*}
\mathcal{P}_1 & = \langle\;A,B \mid R_1 = AAAAABBAABBAAbbbAABBAABB\; \rangle;\\
\mathcal{P}_2 & = \langle\;A,B \mid R_2 = AAABBAAbAABBAAABBAAbAAbAABB \; \rangle;\\
\mathcal{P}_3 & = \langle\;A,B \mid R_3 = AAAAABBBAABAABBBAAAAABBBAABBB \; \rangle; \\
\mathcal{P}_4 & = \langle\;A,B \mid R_4 = AABAABaBAABaBBaBAABaBBaBAABaB \; \rangle.
\end{align*}
Since the cyclic words which $R_1$, $R_2$, $R_3$, and $R_4$ represent in $F_2(A,B)$ are distinct up to automorphisms of $F_2(A,B)$, the pairs $(H,R_1)$, $(H,R_2)$, $(H,R_3)$, and $(H,R_4)$ are also distinct up to homeomorphism. So, up to homeomorphism, s768(0,0) has at least four distinct tunnels.

Next, Heegaard finds the following four pairs of distinguished curves $(m_{i,1},m_{i,2})$, obtained by performing surgery on each of $R_1$, $R_2$, $R_3$, and $R_4$ along the distinguished waves based at each $R_i$ for $i = 1,2,3,4$:
\begin{alignat*}{3}
m_{1,1} &= aabbaaB, & m_{1,2} &= aabbaabbaaaaa ;\\
m_{2,1} &= aabbaaa, & m_{2,2} &= baaBaabbaaabbaaBaa ;\\
m_{3,1} &= bbaaa, & m_{3,2} &= aabbbaabaabbbaaaaabbbaab ;\\
m_{4,1} &= baabAbaaba, \quad \quad & m_{4,2} &= bAbaabAbbAbaabAbb ;
\end{alignat*}
Then filling the manifolds with presentations $\mathcal{P}_1$, $\mathcal{P}_2$, $\mathcal{P}_3$, and $\mathcal{P}_4$ at distinguished-wave-determined slopes yields manifolds with presentations $\mathcal{P}_5$, $\mathcal{P}_6$, $\mathcal{P}_7$, and $\mathcal{P}_8$ respectively.
\begin{alignat*}{2}
\mathcal{P}_5 & = \langle\;A,B \mid m_{1,1},m_{1,2}\; \rangle  = \langle\;A,B \mid aabbaaB,aabbaabbaaaaa\; \rangle; \\
\mathcal{P}_6 & = \langle\;A,B \mid m_{2,1},m_{2,2}\; \rangle  = \langle\;A,B \mid aabbaaa,baaBaabbaaabbaaBaa\; \rangle; \\
\mathcal{P}_7 & = \langle\;A,B \mid m_{3,1},m_{3,2}\; \rangle  = \langle\;A,B \mid bbaaa,aabbbaabaabbbaaaaabbbaab\; \rangle; \\
\mathcal{P}_8 & = \langle\;A,B \mid m_{4,1},m_{4,2}\; \rangle  = \langle\;A,B \mid baabAbaaba,bAbaabAbbAbaabAbb\; \rangle.
\end{alignat*}
Simplifying these, Heegaard finds that $\mathcal{P}_5$, $\mathcal{P}_6$, and $\mathcal{P}_7$ each simplify to $\mathcal{P}_9$, while $\mathcal{P}_8$ simplifies to $\mathcal{P}_{10}$.
\begin{alignat*}{1}
\mathcal{P}_9 & = \langle\;A,B \mid  AABaBABaB, ABBAb \; \rangle; \\
\mathcal{P}_{10} & = \langle\;A,B \mid  AABAAbAABAAbAb, AAbabbab \; \rangle. 
\end{alignat*}
For comparison, $\mathcal{P}_{11}$ and $\mathcal{P}_{12}$ below are SnapPy's realizable presentations of the fundamental groups of s768(0,1), and s768(-1,1) respectively. 
\begin{alignat*}{1}
\mathcal{P}_{11} & = \langle\;A,B,C,D,E,F \mid Baa, cFcBA, Dff, EBd, EEDC, cFcBA, \; \rangle. \\
\mathcal{P}_{12} & = \langle\;A,B,C,D,E,F \mid Baa, cFcBA, Dff, EBd, EEDC, eDfCcFcee \; \rangle.
\end{alignat*}
Simplifying these, Heegaard finds one of the presentations $\mathcal{P}_{11}$ simplifies to is identical with $\mathcal{P}_9$, and one of the presentations $\mathcal{P}_{12}$ simplifies to is identical to $\mathcal{P}_{10}$.

Finally, Dunfield's paper \cite{D19a} and associated list \cite{D19b} of the exceptional fillings of the manifolds in the OCC show that s768(0,0) has exceptional filling slopes (0,1), (1,0), (1,1), and (-1,1), of which only (1,0) = L(43,12) is cyclic. So the two distinguished-wave-determined slopes (0,1) and (-1,1) are each exceptional, but the cyclic slope (1,0) does not appear as a distinguished-wave-determined slope.
\end{ex}

\subsection{Some OCC survey motivated conjectures.} \hfill
\smallskip

\begin{conj}
The set of tunnels of a tunnel-number-one orientable 3-manifold with incompressible torus boundary exhibits no more than two distinguished-wave-determined slopes.
\end{conj}

\begin{conj}
If a tunnel-number-one orientable 3-manifold $M$ with incompressible torus boundary has tunnels exhibiting distinct distinguished-wave-deter-mined slopes, then those slopes are distance one apart on $\partial M$.
\end{conj}

\begin{conj}
A tunnel-number-one orientable 3-manifold with incompressible torus boundary has no more than five distinct tunnels, up to homeomorphism.
\end{conj}

\begin{rem} \textbf{{Data for Heegaard's Distinguished-Wave-Determined slopes for tunnel-number-one OCC manifolds, Heegaard program source and documentation.}} \hfill
\label{OCC data}

A 32 MB zipped archive $Distinguished\_Waves\_and\_Slopes.zip$ containing 9 files with 300 MB of data from Heegaard's survey of distinguished-wave-determined slopes for tunnel-number-one OCC manifolds is at: 
\begin{verbatim}
   < https://bitbucket.org/JOIBerge/heegaard_results2/src/master >. 
\end{verbatim}
(Attempting to view the raw data in $Distinguished\_Waves\_and\_Slopes.zip$ should cause the archive to download.)

Finally, the 3-manifold topology software archive maintained by Marc Culler and Nathan Dunfield at:
\begin{verbatim}
                 < https://t3m.math.uic.edu >
\end{verbatim}
contains links to some Heegaard program software material.

\end{rem}

%%%%%%%%%%%%%%%%%%%%%

\end{document}